\documentclass{amsart}
\usepackage{amssymb, amsmath, latexsym, amscd, graphicx, color}
\usepackage[all]{xy}
\usepackage[dvipdfm]{hyperref}
\usepackage[all]{hypcap}
\newtheorem{theorem}{Theorem}[section]
\newtheorem{lemma}[theorem]{Lemma}
\newtheorem{cor}[theorem]{Corollary}
\newtheorem{definition}[theorem]{Definition}
\newtheorem{example}[theorem]{Example}
\newtheorem{remark}[theorem]{Remark}
\newtheorem{proposition}[theorem]{Proposition}
\def\pagenumber{1}
\begin{document}
\setcounter{page}{\pagenumber}
\newcommand{\T}{\mathbb{T}}
\newcommand{\R}{\mathbb{R}}
\newcommand{\Q}{\mathbb{Q}}
\newcommand{\N}{\mathbb{N}}
\newcommand{\Z}{\mathbb{Z}}
\newcommand{\tx}[1]{\quad\mbox{#1}\quad}
\parindent=0pt
\def\SRA{\hskip 2pt\hbox{$\joinrel\mathrel\circ\joinrel\to$}}
\def\tbox{\hskip 1pt\frame{\vbox{\vbox{\hbox{\boldmath$\scriptstyle\times$}}}}\hskip 2pt}
\def\circvert{\vbox{\hbox to 8.9pt{$\mid$\hskip -3.6pt $\circ$}}}
\def\IM{\hbox{\rm im}\hskip 2pt}
\def\COIM{\hbox{\rm coim}\hskip 2pt}
\def\COKER{\hbox{\rm coker}\hskip 2pt}
\def\TR{\hbox{\rm tr}\hskip 2pt}
\def\GRAD{\hbox{\rm grad}\hskip 2pt}
\def\RANK{\hbox{\rm rank}\hskip 2pt}
\def\MOD{\hbox{\rm mod}\hskip 2pt}
\def\DEN{\hbox{\rm den}\hskip 2pt}

\title[Exotic Heat PDE's.II]{\mbox{}\\[1cm] EXOTIC HEAT PDE's.II}
\author{Agostino Pr\'astaro}
\maketitle
\vspace{-.5cm}
{\footnotesize
\begin{center}
Department of Methods and Mathematical Models for Applied
Sciences, University of Rome ''La Sapienza'', Via A.Scarpa 16,
00161 Rome, Italy. \\
E-mail: {\tt Prastaro@dmmm.uniroma1.it}
\end{center}
}
\vspace{1cm}
\centerline{\it This paper is dedicated to Stephen Smale in occasion of his 80th birthday}

\begin{abstract} \noindent
Exotic heat equations that allow to prove the Poincar\'e conjecture and its generalizations to any dimension are considered. The methodology used is the PDE's algebraic topology, introduced by A. Pr\'astaro in the geometry of PDE's, in order to characterize global solutions. In particular it is shown that this theory allows us to identify $n$-dimensional {\em exotic spheres}, i.e., homotopy spheres that are homeomorphic, but not diffeomorphic to the standard $S^n$.
\end{abstract}


\vspace{.08in} \noindent \textbf{Keywords}: Integral bordisms in PDE's;
Existence of local and global solutions in PDE's; Conservation laws;
Crystallographic groups; Singular PDE's; PDE's on complex manifolds category; PDE's on quantum supermanifolds category; Exotic spheres.

\section{Introduction}

\rightline{\it ''How exotic are exotic spheres ?''}

\vskip 1cm
The term ''exotic sphere'' was used by J. Milnor to characterize smooth manifolds that are homotopy equivalent and homeomorphic to $S^n$, but not diffeomorphic to $S^n$.\footnote{In this paper we will use the following notation: $\thickapprox$ homeomorphism; $\cong$ diffeomorphism; $\approxeq$ homotopy equivalence; $\simeq$ homotopy.} This strange mathematical phenomenon, never foreseen before the introduction just by J. Milnor of the famous $7$-dimensional exotic sphere \cite{MILNOR}, has stimulated a lot of mathematical research in algebraic topology. The starting points, were, other than the cited paper by J. W. Milnor, also a joint paper with M. A. Kervaire \cite{KERVAIRE-MILNOR} and some papers by S. Smale \cite{SMALE1}, Freedman \cite{FREEDMAN} and J. Cerf \cite{CERF} on generalizations of the Poincar\'e conjecture in dimension $n\ge 4$. There the principal mathematical tools utilized were Morse theory (Milnor), h-cobordism theory (Smale), surgery techniques and Hirzebruch signature formula. Surprising, from this beautiful mathematical architecture was remained excluded just the famous Poincar\'e conjecture for $3$-dimensional manifolds. In fact, the surgery techniques do not give enough tools in low dimension ($n<5$), where surgery obstructions disappear. Really, it was necessary to recast the Poincar\'e problem as a problem to find solutions in a suitable PDE equation ({\em Ricci flow equation}), to be able to obtain more informations just on dimension three. (See works by R.S. Hamilton \cite{HAMIL1, HAMIL2, HAMIL3, HAMIL4, HAMIL5}, G. Perelman \cite{PER1, PER2} and A. Pr\'astaro \cite{PRA14, AG-PRA1}.) The idea by R. S. Hamilton to recast the problem in the study of the Ricci flow equation has been the real angular stone that has allowed to look to the solution of the Poincar\'e conjecture from a completely new point of view. In fact, with this new prospective it was possible to G. Perelman to obtain his results and to A. Pr\'astaro to give a new proof of this conjecture, by using his PDE's algebraic topologic theory. To this respect, let us emphasize that the usual geometric methods for PDE's (Spencer, Cartan), were able to formulate for nonlinear PDE's, local existence theorems only, until the introduction, by A. Pr\'astaro, of the algebraic topologic methods in the PDE's geometric theory. These give suitable tools to calculate integral bordism groups in PDE's, and to characterize global solutions. Then, on the ground of integral bordism groups, a new geometric theory of stability for PDE's and solutions of PDE's has been built. These general methodologies allowed to A. Pr\'astaro to solve fundamental mathematical problems too, other than the Poincar\'e conjecture and some of its generalizations, like characterization of global smooth solutions for the Navier-Stokes equation and global smooth solutions with mass-gap for the quantum Yang-Mills superequation. (See \cite{PRA3, PRA4, PRA5, PRA6, PRA7, PRA8, PRA9, PRA10, PRA11, PRA12, PRA13, PRA14, PRA15}.\footnote{See also Refs. \cite{AG-PRA1, AG-PRA2, PRA-RAS}, where interesting related applications of the PDE's Algebraic Topology are given.})

The main purpose of this paper is to show how, by using the PDE's algebraic topology, introduced by A. Pr\'astaro, one can prove the Poincar\'e conjecture in any dimension for the category of smooth manifolds, but also to identify exotic spheres. In the part I \cite{PRA16} we have just emphasized as in dimension $3$, the method followed by A. Pr\'astaro allows us to prove the Poincar\'e conjecture and to state also that $3$-dimensional homotopy spheres are diffeomorphic to $S^3$. (Related problems are considered there too.) In the framework of the PDE's algebraic topology, the identification of exotic spheres is possible thanks to an interaction between integral bordism groups of PDE's, conservation laws, surgery and geometric topology of manifolds. With this respect we shall enter in some details on these subjects, in order to well understand and explain the meaning of such interactions. So the paper splits in three sections other this Introduction. 2. Integral bordism groups in Ricci flow PDE's. 3. Morse theory in Ricci flow PDE's. 4. h-Cobordism in Ricci flow PDE's. The main result is contained just in this last section and it is Theorem \ref{integral-h-cobordism-in-ricci-flow-pde} that by utilizing the previously results considered states (and proves) the following.\footnote{In order to allow a more easy understanding, this paper has been written in a large expository style.}

\textbf{Theorem \ref{integral-h-cobordism-in-ricci-flow-pde}.} {\em The generalized Poincar\'e conjecture, for any dimension $n\ge 1$ is true, i.e., any $n$-dimensional homotopy sphere $M$ is homeomorphic to $S^n$: $M\approx S^n$.

For $1\le n\le 6$, $n\not=4$, one has also that $M$ is diffeomorphic to $S^n$: $M\cong S^n$. But for $n> 6$,  it does not necessitate that $M$ is diffeomorphic to $S^n$. This happens when the Ricci flow equation, under the {\em homotopy equivalence full admissibility hypothesis}, (see below for definition), becomes a $0$-crystal.

Moreover, under the {\em sphere full admissibility hypothesis}, the Ricci flow equation becomes a $0$-crystal in any dimension $n\ge 1$.}

\section{\bf INTEGRAL BORDISM GROUPS IN RICCI FLOW PDE's}

In this section we shall characterize the local and global solutions of the Ricci flow equation, following the geometric approach of some our previous works on this equation \cite{PRA4, PRA14, PRA12, AG-PRA1}.
Let $M$ be a $n$-dimensional smooth manifold and let us consider
the following fiber bundle $\bar\pi:E\equiv\mathbb{R}\times
\widetilde{S^0_2M}\to{\mathbb R}\times M$, $(t,x^i,y_{ij})_{1\le
i,j\le n}\mapsto(t,x^i)\equiv(x^\alpha)_{0\le\alpha\le n}$, where
$\widetilde{S^0_2M}\subset S^0_2M$ is the open subbundle of
non-degenerate Riemannian metrics on $M$. Then the Ricci flow
equation is the closed second order partial differential relation,
(in the sense of Gromov \cite{GROMOV}), on the fiber bundle
$\bar\pi:E\to{\mathbb R}\times M$, $(RF)\subset J{\it D}^2(E)$,
defined by the differential polynomials on $J{\it
D}^2(E)$ given in (\ref{differential-polynomials-ricci-flow-pde}):

\begin{equation}\label{differential-polynomials-ricci-flow-pde}
\left\{\begin{array}{ll}
 F_{jl}&\equiv|y|[
y_{ik}](y_{il,jk}+y_{jk,il}-y_{jl,ik}-y_{ik,jl})\\
&+[y_{ik}][y_{rs}]([jk,r][il,s]-[jl,r][ik,s])+{{|y|^2}\over{2}}y_{jl,t}\\
&\equiv
S_{jl}(y_{rs},y_{rs,\alpha},y_{rs,pq})+{{|y|^2}\over{
2}}y_{jl,t}=0,\\
\end{array}
\right.
\end{equation}
where $[ij,r]$ are the usual Christoffels symbols,
given by means of the coordinates $y_{rs,i}$, $|y|=\det(y_{ik})$,
and $[y_{ik}]$ is the algebraic complement of $y_{ik}$. The ideal
${\frak p}\equiv<F_{jl}>$ is not prime in ${\mathbb
R}[y_{rs},y_{rs,\alpha},y_{rs,ij}]$. However, an irreducible
component is described by the system in solved form:
$y_{rs,t}=-{{2}\over{|y|^2}}S_{jl}$. This is formally integrable
and also completely integrable.\footnote{We shall denote with the
same symbol $(RF)$ the corresponding algebraic manifold. For a
geometric algebraic theory of PDE's see the monograph \cite{PRA3}, and
references quoted there.} In fact,
\begin{equation}\label{dimensions-structures-ricci-flow-pde}
\left\{\begin{array}{l}
 \dim J{\it D}^{2+s}(E)=n+1+\sum_{0\le r\le
2+s}{{(n+1)n}\over{2}}{{(n+r)!}\over{r!n!}}\\
\dim(RF)_{+s}=n+1+{{(n+1)n}\over{2}}[\sum_{0\le r\le
2+s}{{(n+r)!}\over{r!n!}}-\sum_{0\le r'\le
s}{{(n+r')!}\over{r'!n!}}]\\
\dim g_{2+s}={{(n+1)n}\over{2}}{{(n+2+s)!}\over{(2+s)!n!}}-{{(n+1)n}\over{2}}{{(n+s)!}\over{s!n!}}.\\
\end{array}\right.
\end{equation}

Therefore, one has: $\dim (RF)_{+s}=\dim (RF)_{+(s-1)}+\dim
g_{2+s}$. This assures that one has the exact sequences in (\ref{short-exact-sequence-for-surjectivity-prolongation-ricci-flow-pde}).
\begin{equation}\label{short-exact-sequence-for-surjectivity-prolongation-ricci-flow-pde}
\xymatrix{(RF)_{+s}\ar[r]&(RF)_{+(s-1)}\ar[r]& 0},\: s\ge 1.
\end{equation}
One can also see that the symbol $g_2$ is not involutive. By the way a
general theorem of the formal geometric theory of PDE's assures
that after a finite number of prolongations, say $s$, the
corresponding symbol $g_{2+s}$ becomes involutive. (See
refs.{\em\cite{GOLD, PRA3}}.) Then, taking into account the surjectivity of
the mappings (\ref{short-exact-sequence-for-surjectivity-prolongation-ricci-flow-pde}), we get that $(RF)$ is formally
integrable. Furthermore, from the algebraic character of this
equation, we get also that is completely integrable. Therefore, in
the neighborhood of any of its points $q\in(RF)$ we can find
solutions. (These can be analytic ones, but also smooth if we
consider to work on the infinity prolongation $(RF)_{+\infty}$,
where the Cartan distribution is ''involutive'' and of dimension
$(n+1)$.) Finally, taking into account that $\dim
(RF)>2(n+1)+1=2n+3$, we can use Theorem 2.15 in \cite{PRA4} to calculate the
$n$-dimensional singular integral bordism group, $\Omega^{(RF)}_{n,s}$, for $n$-dimensional
closed smooth admissible integral manifolds bording by means of
(singular) solutions. (Note that the symbols of $(RF)$ and its
prolongations are non-zero.) This group classifies the structure
of the global singular solutions of the Ricci-flow equation. One has:
\begin{equation}\label{integral-singular-bordism-group-ricci-flow-pde}
\Omega^{(RF)}_{n,s}\cong\bigoplus_{r+s=n}H_r(M;{\mathbb
Z}_2)\otimes_{\mathbb{Z}_2}\Omega_s,
\end{equation}
 where $\Omega_s$ is the
bordism group for $s$-dimensional closed smooth
manifolds.\footnote{We used the fact that the fiber of $E\to M$ is
contractible.}

It is important to underline that with the term ''$n$-dimensional
closed smooth admissible integral manifolds'' we mean smooth integral manifolds, $N\subset (RF)\subset J{\it D}^2(E)$, that diffeomorphically project on their image on $E$, via the canonical projection $\pi_{2,0}: J{\it D}^2(E)\to E$. In \cite{PRA14} we have proved, that any smooth section $g:M\to \widetilde{S^0_2M}$, identifies a space-like $n$-dimensional smooth integral manifold $N\subset (RF)$, and that for such a Cauchy manifold pass local smooth solutions, contained in a tubular neigbourhood $N\times[0,\epsilon)\subset(RF)$, for suitable $\epsilon>0$. Therefore, we can represent any $n$-dimensional smooth compact Riemannian manifold $(M,\gamma)$ as a space-like Cauchy manifold $N_0\subset(RF)_{t_0}$, for some initial time $t_0$, and ask if there are solutions that bord $N_0$ with $(S^n,\gamma')$, where $\gamma'$ is the canonical metric of $S^n$, identified with another space-like Cauchy manifold  $N_1\subset(RF)_{t_1}$, with $t_0< t_1$. The answer depends on the class of solution that we are interested to have. For weak-singular solutions the corresponding integral bordism group $\Omega^{(RF)}_{n,s}$ is given in (\ref{integral-singular-bordism-group-ricci-flow-pde}).
The relation with the integral bordism group $\Omega^{(RF)}_{n}$, for smooth solutions of $(RF)$ is given by the exact commutative diagram (\ref{smooth-integral-bordism-group-rf-pde}) where is reported the relation with the bordism group $\Omega_n$ for smooth manifolds.
\begin{equation}\label{smooth-integral-bordism-group-rf-pde}
\xymatrix{0\ar[rd]&&&0\ar[d]&\\
&\overline{K}^{(RF)}_{n}\ar[rd]\ar[rr]&&K^{(RF)}_{n,s;n}\ar[d]\ar[r]&0\\
0\ar[r]&K^{(RF)}_{n,s}\ar[u]\ar[r]&\Omega^{(RF)}_{n}\ar[rd]^{c}\ar[r]^{a}&\Omega^{(RF)}_{n,s}\ar[d]^{b}\ar[r]&0\\
&0\ar[u]\ar[rr]&&\Omega_{n}\ar[d]&\\
&&&0&\\}
\end{equation}
\begin{theorem}
Let $M$ in the Ricci flow equation $(RF)\subset J{\it D}^2(E)\subset J^2_n(W)$ be a smooth compact $n$-dimensional manifold homotopy equivalent to $S^n$. Then the $n$-dimensional singular integral bordism group of $(RF)$ is given in {\em(\ref{singular-integral-bordism-homotopy-sphere-n-dimensional})}.
\begin{equation}\label{singular-integral-bordism-homotopy-sphere-n-dimensional}
  \Omega^{(RF)}_{n,s}=\Omega_{n}\bigoplus\mathbb{Z}_2.
\end{equation}
Then one has the exact commutative diagram given in {\em(\ref{smooth-integral-bordism-group-rf-pde-homotopy-sphere-case})}
\begin{equation}\label{smooth-integral-bordism-group-rf-pde-homotopy-sphere-case}
\xymatrix{0\ar[rd]&&&0\ar[d]&\\
&\overline{K}^{(RF)}_{n}\ar[rd]\ar[rr]&&K^{(RF)}_{n,s;n}\ar[d]\ar[r]&0\\
0\ar[r]&K^{(RF)}_{n,s}\ar[u]\ar[r]&\Omega^{(RF)}_{n}\ar[rd]^{c}\ar[r]^{a}&\Omega_n\bigoplus \mathbb{Z}_2\ar[d]^{b}\ar[r]&0\\
&0\ar[u]\ar[rr]&&\Omega_{n}\ar[d]&\\
&&&0&\\}
\end{equation}
One has the isomorphisms reported in {\em(\ref{isomorphisms-smooth-integral-bordism-group-rf-pde-homotopy-sphere-case})}.
\begin{equation}\label{isomorphisms-smooth-integral-bordism-group-rf-pde-homotopy-sphere-case}
\left\{\begin{array}{ll}
{\rm(a)}:& \Omega^{(RF)}_{n}/ K^{(RF)}_{n,s}\cong\Omega^{(RF)}_{n,s}\\
\\
{\rm(b)}:&   \Omega^{(RF)}_{n}/\overline{K}^{(RF)}_{n}\cong\Omega_n\\
\end{array}\right.
\end{equation}

\end{theorem}

\begin{proof}
Since we have assumed that $M$ is homotopy equivalent to $S^n$, ($M\approxeq S^n$), we can state that $M$ has the same homology groups of $S^n$. Therefore we get the isomorphisms reported in {\em(\ref{homology-groups-homotopy-sphere-n-dimensional})}.

\begin{equation}\label{homology-groups-homotopy-sphere-n-dimensional}
H_p(M;\mathbb{Z}_2)\cong H_p(S^n;\mathbb{Z}_2)\cong\left\{\begin{array}{l}
                                                          \mathbb{Z}_2\, ,p=0,n \\
                                                         0\, ,\hbox{\rm otherwise}\\
                                                        \end{array}\right.
\end{equation}
Therefore, taking into account (\ref{integral-singular-bordism-group-ricci-flow-pde}) we get the isomorphism (\ref{singular-integral-bordism-homotopy-sphere-n-dimensional}).
\end{proof}

\begin{example}
In Tab. \ref{examples-singular-integral-bordism-groups-for-n-dimensional-homotopy-spheres} we report some explicitly calculated cases of integral singular bordism groups for $1\le n\le 7$.
\begin{table}[h]
\caption{Examples of singular integral bordism groups for $n$-dimensional homotopy spheres}
\label{examples-singular-integral-bordism-groups-for-n-dimensional-homotopy-spheres}
\begin{tabular}{|c|c|}
  \hline
  &\\
  \hfil{\rm{\footnotesize $n$}}\hfil& \hfil{\rm{\footnotesize $\Omega^{(RF)}_{n,s}$}}\hfil\\
  &\\
\hline
\hfil{\rm{\footnotesize $1$}}\hfil& \hfil{\rm{\footnotesize $\mathbb{Z}_2$}}\hfil\\
\hfil{\rm{\footnotesize $2$}}\hfil& \hfil{\rm{\footnotesize $\mathbb{Z}_2\oplus\mathbb{Z}_2$}}\hfil\\
\hfil{\rm{\footnotesize $3$}}\hfil& \hfil{\rm{\footnotesize $\mathbb{Z}_2$}}\hfil\\
\hfil{\rm{\footnotesize $4$}}\hfil& \hfil{\rm{\footnotesize $\mathbb{Z}_2\oplus\mathbb{Z}_2\oplus\mathbb{Z}_2$}}\hfil\\
\hfil{\rm{\footnotesize $5$}}\hfil& \hfil{\rm{\footnotesize $\mathbb{Z}_2\oplus\mathbb{Z}_2$}}\hfil\\
\hfil{\rm{\footnotesize $6$}}\hfil& \hfil{\rm{\footnotesize $\mathbb{Z}_2\oplus\mathbb{Z}_2\oplus\mathbb{Z}_2\oplus\mathbb{Z}_2$}}\hfil\\
\hfil{\rm{\footnotesize $7$}}\hfil& \hfil{\rm{\footnotesize $\mathbb{Z}_2\oplus\mathbb{Z}_2$}}\hfil\\
\hline
\end{tabular}
\end{table}

\end{example}

\section{\bf MORSE THEORY IN RICCI FLOW PDE's}

Let us now give the fundamental theorem that describe quantum tunnel
effects in solutions of PDEs, i.e., the change of sectional topology in the (singular) solutions of $(RF)$.

\begin{theorem}[Topology transitions as quantum tunnel effects in Ricci flow equation]\label{tunnel-effect-pde}
Let $N_0, N_1\subset(RF)\subset J{\it D}^2(E)$ be space-like Cauchy manifolds of $(RF)$, at two different times $t_0\not= t_1$. Let $V\subset (RF)$ be a (singular) solution such that $\partial V=N_0\sqcup N_1$. Then there exists an admissible Morse function $f:V\to [a,b]\subset\mathbb{R}$ such that:

{\em(A) (Simple quantum tunnel  effect).} If $f$ has a critical point
$q$ of index $k$ then there exists a $k$-cell
$e^k\subset V-N_1$  and an $(n-k+1)$-cell
 ${e_*}^{n-k+1}\subset V-N_0$ such that:

{\em(i)} $e^k\cap N_0=\partial e^k$;

{\em(ii)} ${e_*}^{n-k+1}\cap N_1=\partial {e_*}^{n-k+1}$;

{\em(iii)} there is a
deformation retraction of $V$ onto $N_0\cup e^k$;

{\em(iv)} there is a deformation retraction of $V$ onto $N_1\cup {e_*}^{n-k+1}$;

{\em(v)} ${e_*}^{n-k+1}\cap e^k=q$; ${e_*}^{n-k+1} \pitchfork
e^k$.

{\em(B) (Multi quantum tunnel  effect).} If $f$ is of
type $(\nu _0,\cdots,\nu_{n+1})$ where $\nu_k$ denotes the number of
critical points with index $k$ such that  $f$ has only one critical
values $c$, $a<c<b$, then there are disjoint $k$-cells $e^k_i\subset
V\setminus N_1$ and disjoint $(n-k+1)$-cells $(e_*)^{n-k+1}_i\subset
V\setminus N_0$, $1\le  i\le \nu_k$, $k=0,\cdots,n+1$,
such that:

{\em(i)} $e^k_i\cap N_0=\partial e^k_i$;

{\em(ii)} ${e_*}^{n-k+1}_i\cap N_1=\partial {e_*}^{n-k+1}_i$;

{\em(iii)} there
is a deformation retraction of $V$ onto $N_0\bigcup
\left\{\cup_{i,k}(e_*)^k_i\right\}$;

{\em(iv)} there is a
deformation retraction of $V$ onto $N_1\bigcup\left\{\cup_{i,k}
(e_*)^{n-k}_i\right\}$;

{\em(v)} $(e_*)^{n-k}_i\cap e^k_i=q_i$;
$(e_*)^{n-k+1}_i \pitchfork  e^k_i$.

{\em(C) (No topology transition).} If $f$ has no critical point then
$V\cong N_0\times I$ where $I\equiv[0,1]$.\end{theorem}

\begin{proof} The
proof can be conducted by adapting to the Ricci flow equation $(RF)$ Theorem 23  in \cite{PRA00}. Let us emphasize here some important lemmas only.

\begin{lemma}[Morse-Smale functions]\label{morse-smale-functions}
{\em 1)} On a closed connected compact smooth manifold $M$, there exists a Morse function $f:M\to\mathbb{R}$ such that the critical values are ordened with respect to the indexes, i.e., $f(x_\lambda)=f(x_\mu)$, if $\lambda=\mu$, and $f(x_\lambda)>f(x_\mu)$, if $\lambda>\mu$, where $x_\lambda$, (resp. $x_\mu$), is the critical point of $f$ with index $\lambda$ (resp. $\mu$). Such functions are called {\em regular functions}, or {\em Morse-Smale functions}, and are not dense in $C^\infty(M,\mathbb{R})$, as Morse functions instead are. Furthermore, such functions can be chosen in such a way that they have an unique maximum point (with index $\lambda=n=\dim M$), and an unique minimum point (with index $\lambda=0$).

{\em 2)} To such functions are associated vector fields $\zeta=\GRAD f:M\to TM$ such that $\zeta(x_\lambda)=0$ iff $x_\lambda$ is a critical point. Then in a neighborhood of a $x_\lambda$, the integral curves of $\zeta$ are of two types: ingoing in $x_\lambda$, and outgoing from $x_\lambda$. These fit in two different disks $D^\lambda$ and $D^{n-\lambda}$ contained in $M$ called {\em separatrix diagram}. (See Fig. \ref{passing-through-critical-point-attaching-handle}.)
\end{lemma}

\begin{lemma}[Morse functions and CW complexes]\label{morse-functions-and-cw-complexes}
{\em 1)} Let $M$ be a compact $n$-dimension manifold and $f:M\to[a,b]$ an admissible Morse function of type $(\nu_0,\cdots,\nu_n)$ such that $\partial M=f^{-1}(b)$. Then $M$ has the homotopy type of a finite CW complex having $\nu_k$ cells at each dimension $k=0,\cdots,n$. Furthermore $(M,\partial M)$ has the homotopy type of a CW-pair of dimension $n$.

Furthermore $(M,f^{-1}(a))$ has the homotopy of relative CW complex having $\nu_k$ cells of dimension $k$, for each $k=0,\cdots,n$.\footnote{A {\em relative CW complex} $(Y,X)$ is a space $Y$ and a closed subspace $X$ such that $Y=\bigcup\limits^{\infty}_{r=-1}Y_r$, such that $X=Y_{-1}\subset Y_{0}\subset\cdots$, and $Y_r$ is obtained from $Y_{r-1}$ by attaching $r$-cells.}

{\em 2)} An $n$-dimension manifold $M$ has the homotopy type of a CW complex of dimension $\le n$.

{\em 3)} Let $M$ be a compact manifold and $N\subset M$ a compact submanifold with $\partial N=\partial M=\varnothing$. Then $(M,N)$ has the homotopy type of CW pair.

{\em 4)} The cell decomposition of a closed connected compact smooth manifold $M$, related to a Morse-Smale function, is obtained attaching step-by-step a cell of higher dimension to the previous ones.
\end{lemma}

\begin{lemma}[Homological Euler characteristic]\label{homological-euler-characteristic}
{\em 1)} If the compact solution $V$ of $(RF)$ is characterized by an admissible Morse function $f:V\to[a,b]$ of type $(\nu_0,\cdots,\nu_{n+1})$, then its homological Euler characteristic $\chi_{hom}(V)$ is given by the formula {\em(\ref{homological-euler-characteristic-solution-ricci-flow-equation})}.
\begin{equation}\label{homological-euler-characteristic-solution-ricci-flow-equation}
  \chi_{hom}(V)=\sum_{0\le k\le(n+1)}(-1)^{k}\beta_k,\hskip 4pt \beta_k=\dim_F H_k(V,f^{-1}(a);F)
\end{equation}
where $F$ is any field.\footnote{Let $H_k(Y,X;F)$ denote the singular homology group of the pair $(Y,X)$ with coefficients in the field $F$. $\beta_k=\dim_F H_k(Y,X;F)$ are called the {\em$F$-Betti numbers} of $(Y,X)$. If these numbers are finite and only finitely are nonzero, then the {\em homological Euler characteristic} of $(Y,X)$ is defined by the formula: $\chi_{hom}(Y,X)=\sum_{0\le k\le\infty}(-1)^{k}\beta_k$. When $Y$ is a compact manifold and $X$ is a compact submanifold, then $\chi_{hom}(Y,X)$ is defined. The evaluation of homological Euler characteristic for topological spaces coincides with that of Euler characteristic for CW complexes $X$ given by $\chi(X)=\sum_{i\ge 0}(-1)^{i}k_i$, where $k_i$ is the number of cells of dimension $i$. For closed smooth manifolds $M$, $\chi(M)$ coincides with the {\em Euler number}, that is the Euler class of the tangent bundle $TM$, evalued on the fundamental class of $M$. For closed Riemannian manifolds, $\chi(M)$ is given as an integral on the curvature, by the {\em generalized Gauss-Bonnet theorem}: $\chi(V)=\frac{1}{(2\pi)^n}\int_VPf(\Omega)$, where $\partial V=\varnothing$, $\dim V=2n$, $\Omega$ is the curvature of the Levi-Civita connection and $Pf(\Omega)=\frac{1}{2^nn!}\sum_{\sigma\in S_{2n}}\epsilon(\sigma)\prod_{i=1}^{n}\Omega_{\sigma(2i-1)\sigma(2i)}$, where $(\Omega_{rs})$ is the skew-symmetric $(2n)\times(2n)$ matrix representing $\Omega:V\to\mathfrak{so}(2n)\bigotimes\Lambda^0_{2}(V)$, hence $Pf(\Omega):V\to \Lambda^0_{2n}(V)$. In Tab. \ref{properties-euler-characteristic-and-examples} are reported some important properties of Euler characteristic, that are utilized in this paper.}
Furthermore, if $f^{-1}(a)=\varnothing$, then $\beta_k$ in {\em(\ref{homological-euler-characteristic-solution-ricci-flow-equation})} is given by $\beta_k=\dim_F H_k(V;F)$, and $\chi_{hom}(V)=\chi(V)$.

{\em 2)} If $M$ is a compact odd dimensional manifold with $\partial M=\varnothing$, then $\chi_{hom}(M)=0$.

{\em 3)} Let $M$ be a compact manifold such that its boundary can be divided in two components: $\partial M=\partial_{-}M\bigcup\partial_{+}M$, then $\chi_{hom}(M,\partial_{+}M)=\chi_{hom}(M,\partial_{-}M)$.

{\em 4)} Let $M$ be a compact manifold such that $\partial M=N_0\sqcup N_1$, with $N_i$, $i=0,1$, disjoint closed sets. Let $f:M\to\mathbb{R}$ be a $C^2$ map without critical points, such that $f(N_0)=0$, $f(N_1)=1$. Then one has the diffeomorphisms: $M\cong N_0\times I$, $M\cong N_1\times I$.

{\em 5)} Let $M$ be a compact $n$-dimensional manifold with $\partial M=\varnothing$, such that has a Morse function $f:M\to\mathbb{R}$ with only two critical points. Then $M$ is homeomorphic to $S^n$.
\end{lemma}
\begin{table}[t]
\caption{Euler characteristic $\chi$: properties and examples.}
\label{properties-euler-characteristic-and-examples}
\scalebox{0.55}{$\begin{tabular}{|l|l|l|l|}
  \hline
  \hfil{\rm{\footnotesize Definition}}\hfil& \hfil{\rm{\footnotesize $\chi$}}\hfil&\hfil{\rm{\footnotesize Remarks}}\hfil& \hfil{\rm{\footnotesize Examples}}\hfil\\
  \hline
{\rm{\footnotesize $\partial V=\varnothing$, $\dim V=2n+1$, $n\ge 0$}}& {\rm{\footnotesize $\chi(V)=0$}}&{\rm{\footnotesize (from Poincar\'e duality)}}& \\
  \hline
{\rm{\footnotesize $M\approxeq N$}}& {\rm{\footnotesize $\chi(M)=\chi(N)$}}&{\rm{\footnotesize from $H^\bullet(M)\cong H^\bullet(N)$}}& {\rm{\footnotesize $\chi(pt)=1$}}\\
{\rm{\footnotesize (homotopy equivalence)}}&&&{\rm{\footnotesize $\chi(S^n)=1+(-1)^n=0\hskip 3pt(n={\rm odd}), 2\hskip 3pt(n={\rm even})$}}\\
&&&{\rm{\footnotesize $\chi(D^3)=\chi(P^3)=\chi(\mathbb{R}^n)=1$, $P^3$={\rm convex polyhedron})}}\\
&&&{\rm{\footnotesize $\chi(S^2)=\chi(\partial P^3)=2$, $\partial P^3$={\rm surface convex polyhedron})}}\\
\hline
{\rm{\footnotesize $V=M\sqcup N$}}&{\rm{\footnotesize $\chi(V)=\chi(M)+\chi(N)$}}&{\rm{\footnotesize (from homology additivity)}}&{\rm{\footnotesize $\chi(\underbrace{S^2\sqcup\cdots\sqcup S^2}_{n})=2n$}}\\
  \hline
{\rm{\footnotesize excision couple}}&{\rm{\footnotesize $\chi(M\cup N)=\chi(M)+\chi(N)-\chi(M\cap N)$}}&& {\rm{\footnotesize $\chi(S^2)=\chi(D^2)+\chi(D^2)-\chi(S^1)=1+1-0=2$}}\\
{\rm{\footnotesize $(M,N)$}}&&& {\rm{\footnotesize $\chi(K_{lb})=\chi(M_{ob})+\chi(M_{ob})-\chi(S^1)=0+0-0=0$}}\\
&&& {\rm{\footnotesize $\chi(C_{rc}\bigcup_{S^1}D^2=\mathbb{R}P^2)=\chi(M_{ob})+\chi(D^2)-\chi(S^1)=0+1-0=1$}}\\
\hline
{\rm{\footnotesize $V=M\times N$}}&{\rm{\footnotesize $\chi(V)=\chi(M).\chi(N)$}}&&{\rm{\footnotesize $\chi(T^n)=\chi(\underbrace{S^1\times\cdots\times S^1}_{n})=0$}}\\
  \hline
{\rm{\footnotesize $p:V\to M$}}& {\rm{\footnotesize $\chi(V)=\chi(F).\chi(M)$}}&{\rm{\footnotesize (from Serre-spectral sequence)}}&{\rm{\footnotesize $S^n\to\mathbb{R}P^n$: $\chi(S^n)=\chi(\{1,-1\}).\chi(\mathbb{R}P^n)=2.\chi(\mathbb{R}P^n)$}}\\
{\rm{\footnotesize orientable fibration over field}}&&{\rm{\footnotesize (also from transfer map)}}&{\rm{\footnotesize $\chi(\mathbb{R}P^n)=0\hskip 3pt (n= {\rm odd}), 1 \hskip 3pt (n= {\rm even})$.}}\\
{\rm{\footnotesize with fibre $F$}}&&{\rm{\footnotesize ($\tau:H_\bullet(M)\to H_\bullet(V)$)}}&\\
{\rm{\footnotesize $M$ path-connected}}&&{\rm{\footnotesize ($p_\bullet\circ\tau=\chi(F).1_{H_\bullet(M)}$)}}&\\
 \hline
 {\rm{\footnotesize $p:V\to M$}}& {\rm{\footnotesize $\chi(V)=k.\chi(M)$}}&&{\rm{\footnotesize  $p:M_{ob}\to S^1$: $M_{ob}$=M\"obius strip: $\chi(M_{ob})=2\chi(S^1)=0$}}\\
  {\rm{\footnotesize $k$-sheeted covering}}&&&\\
\hline
{\rm{\footnotesize $V=\partial M$, $\dim M=2n$, $n\ge 0$}}& {\rm{\footnotesize $\chi(V)=2m$, $m\ge 0$}}&{\rm{\footnotesize (from excision couple)}}&{\rm{\footnotesize $\mathbb{R}P^{2n}\not=\partial M$, since $\chi(\mathbb{R}P^{2n})=1$}}\\
  \hline
  \multicolumn{4}{l}{\rm{\footnotesize $\chi:\Omega^O_{2i}\to\mathbb{Z}$ is a surjective mapping for $i\ge 1$ and an isomorphism for $i=1$.}}\\
\multicolumn{4}{l}{\rm{\footnotesize $\chi(\partial P^3)=V-E+F$, $V$=vertex-number, $E$=edge-number, $F$=face-number.}}\\
 \multicolumn{4}{l}{\rm{\footnotesize Closed oriented surfaces: $\chi=2-2g$, $g$=genus, (number of handles).}}\\
\multicolumn{4}{l}{\rm{\footnotesize Closed nonorientable surfaces: $\chi=2-\kappa$, $\kappa$=nonorientable genus, (number of real projective planes in a connected decomposition).}}\\
\multicolumn{4}{l}{\rm{\footnotesize Examples of nonorientable surfaces: $K_{lb}$= Klein bottle ($\partial K_{lb}=\varnothing$); $M_{ob}$= M\"obius strip ($\partial M_{ob}=S^1$); $\mathbb{R}P^2$= Projective plane ($\partial \mathbb{R}P^2=\varnothing$).}}\\
\multicolumn{4}{l}{\rm{\footnotesize Examples of nonorientable surfaces: $C_{rc}\approxeq M_{ob}$= cross-cap: surface homotopy equivalent to M\"obius strip.}}\\
\end{tabular}$}
\end{table}

\begin{lemma}
Let $X$ and $Y$ be closed compact differentiable manifolds without boundaries, then there exists a compact manifold $V$, such that $\partial V=X\sqcup Y$ iff $Y$ is obtained from $X$ by a sequence of surgeries. (For details see below Theorem \ref{handle-decomposition-bordisms}.)
\end{lemma}

\end{proof}

\begin{theorem}[Smooth solutions and characteristic vector fields]\label{change-topology-and-smooth-solutions}
The characteristic vector field $\xi$, propagating a space-like $n$-dimensional smooth, compact, Cauchy manifold $N\subset V$, where $V$ is a smooth solution of $(RF)$, hence a time-like, $(n+1)$-dimensional smooth integral manifold of $(RF)$, cannot have zero points.
\end{theorem}
\begin{proof}
In fact, the characteristic vector field $\xi$ coincides with the time-like $\zeta_0\equiv\partial x_0+\sum_{|\beta|\ge 0}y^j_{\alpha\beta}\partial y_j^{\beta}$, where $y^j_{\alpha\beta}$ are determined by the infinity prolongation $(RF)_{+\infty}$ of $(RF)$. Therefore such a vector field cannot have zero points on a compact smooth solution $V$, of $(RF)$, such that $\partial V=N_0\sqcup N_1$. On the other hand, if $f:V\to\mathbb{R}$ is the Morse function whose gradient gives just the vector field $\xi$, then $f$ cannot have critical points.\footnote{Let us recall that a compact connected manifold $M$ with boundary $\partial M\not=\varnothing$, admits a nonvanishing vector field. Furthermore, a compact, oriented $n$-dimensional submanifold $M\subset\mathbb{R}^{2n}$ has a nonvanishing normal vector field. Therefore, above statements about smooth solutions of $(RF)$ agree with well known results of differential topology. (See, e.g. \cite{HIRSCH2}.}
\end{proof}

\begin{cor}
A $(n+1)$-dimensional smooth, compact, manifold $V\subset(RF)$, smooth solution of $(RF)$, such that $\partial V= N_0\sqcup N_1$, where $N_i$, $i=0,1$, are smooth Cauchy manifolds, cannot produce a change of topology from $N_0$ to $N_1$, hence these manifolds must necessarily be homeomorphic.
\end{cor}

The following theorem emphasizes the difference between homeomorphic manifolds and diffeomorphic ones.
\begin{theorem}[Exotic differentiable structures on compact smooth manifolds]\label{exotic-differentiable-structures-on-compact-smooth-manifolds}
Let $M$ and $N$ be $n$-dimensional  homeomorphic compact smooth manifolds. Then it does not necessitate that $M$ is diffeomorphic to $N$.
\end{theorem}

\begin{proof}
Since $M$ is considered homeomorphic to $N$, there exist continuous mappings $f:M\to N$ and $g:N\to M$, such that $g\circ f=id_M$, and $f\circ g=id_N$. Let us consider, now, the following lemma.
\begin{lemma}\label{continuous-mappings-dense-in-differentiable-mappings}
Let $M$ and $N$ be $C^s$ manifolds, $1\le s\le\infty$, without boundary. Then $C^s(M,N)$ is dense in $C^r_S(M,N)$, (in the strong topology), $0\le r<s$.
\end{lemma}

\begin{proof}
See, e.g., \cite{HIRSCH2}.
\end{proof}
From Lemma \ref{continuous-mappings-dense-in-differentiable-mappings} we can state that the above continuous mappings $f$ and $g$ can be approximated with differentiable mapping, but these do no necessitate to be diffeomorphisms. In fact we have the following lemma.

\begin{lemma}\label{lower-differentiable-mappings-dense-in-higher-differentiable-ones}
Let $G^k(M,N)\subset C^k(M,N)$, $k\ge 1$, denote any one of the following subsets: diffeomorphisms, embeddings, closed embeddings, immersions, submersions, proper maps. Let $M$ and $N$ be compact $C^s$ manifolds, $1\le s\le\infty$, without boundary. Then $G^s(M,N)$ is dense in $G^r(M,N)$ in the strong topology, $1\le r<s$. In particular, $M$ and $N$ are $C^s$ diffeomorphic iff they are $C^r$ diffeomorphic with $r\ge 1$.
\end{lemma}

\begin{proof}
See, e.g., \cite{HIRSCH2}.
\end{proof}
Above lemma can be generalized also to compact manifolds with boundary. In fact, we have the following lemma.

\begin{lemma}\label{lower-differentiable-mappings-dense-in-higher-differentiable-ones-compact-manifolds-with-boundary}
Let us consider compact manifolds with boundary. Then the following propositions hold.

{\em (i)} Every $C^r$ manifold $M$, $1\le r<\infty$, is $C^\infty$ diffeomorphic to a $C^\infty$ manifold and the latter is unique up to $C^\infty$ diffeomorphisms.

{\em (ii)}  Let $(M,\partial M)$ and $(N,\partial N)$, be $C^s$ manifold pairs, $1\le s\le \infty$. Then, the inclusion $C^s(M,\partial M;N,\partial N)\hookrightarrow C^r(M,\partial M;N,\partial N)$, $0\le r<s$, is dense in the strong topology. If, $1\le r<s$ and $(M,\partial M)$ and $(N,\partial N)$ are $C^r$ diffeomorphic, they are also $C^s$ diffeomorphic.
\end{lemma}

\begin{proof}
See, e.g., \cite{HIRSCH2}.
\end{proof}

Therefore, it is not enough to assume that compact smooth manifolds should be homeomorphic in order to state that they are also diffeomorphic, hence the proof of Theorem \ref{exotic-differentiable-structures-on-compact-smooth-manifolds} is complete. (To complement Theorem \ref{exotic-differentiable-structures-on-compact-smooth-manifolds} see also Lemma \ref{hirsch-munkres-lemma} and Lemma \ref{kirby-siebenman-lemma} below.)
\end{proof}
From Theorem \ref{exotic-differentiable-structures-on-compact-smooth-manifolds} we are justified to give the following definition.

\begin{definition}\label{definition-exotic-manifolds}
Let $M$ and $N$ be two $n$-dimensional smooth manifolds that are homeomorphic but not diffeomorphic. Then we say that $N$ is an {\em exotic substitute} of $M$.
\end{definition}

\begin{example}
The sphere $S^7$ has 28 exotic substitutes, just called {\em exotic $7$-dimensional spheres}. (See \cite{MILNOR, KERVAIRE-MILNOR}.) These are particular $7$-dimensional manifolds, built starting from oriented fiber bundle pairs over $S^4$. More precisely let us consider $(D^4,S^3)\to (W,V)\to S^4$. The $4$-plane bundle $D^4\to S^4$ is classified by the isomorphism $[S^4,BSO(4)]\cong \mathbb{Z}\bigoplus\mathbb{Z}$, given by $\omega\mapsto(\frac{1}{4}(2\chi(\omega)+p_1(\omega)),\frac{1}{4}(2\chi(\omega)-p_1(\omega)))$, where $\chi(\omega),p_1(\omega)\in H^4(S^4)=\mathbb{Z}$ are respectively the Euler number and the Pontrjagin class of $\omega$, (related by the congruence $p_1(\omega)=2\chi(\omega)\: (\MOD\: 4)$. Let us denote by $(W(\omega),V(\omega))$ the above fiber bundle pair identified by $\omega$. The homology groups of $V(\omega)\to S^4$, are given in {\em(\ref{homology-groups-3-sphere-bundle-over-4-sphere})}.\footnote{Recall that an odd dimensional oriented compact manifold $M$, with $\partial M=\varnothing$ has $\chi(M)=0$. In particular $\chi(S^{2k+1})=0$, instead $\chi(S^{2k})=2$. Furthermore, if $M$ and $N$ are compact oriented manifolds with $\partial M=\partial N=\varnothing$, then $\chi(M\times N)=\chi(M)\chi(N)$.}
\begin{equation}\label{homology-groups-3-sphere-bundle-over-4-sphere}
    H_p(V(\omega))=\left\{
    \begin{array}{l}
     \mathbb{Z}\hskip 5pt\hbox{\rm if $p=0,7$}\\
     \COKER( \chi(\omega):\mathbb{Z}\to\mathbb{Z})\hskip 5pt\hbox{\rm if $p=3$} \\
     \ker( \chi(\omega):\mathbb{Z}\to\mathbb{Z})\hskip 5pt\hbox{\rm if p=4} \\
     0\hskip 5pt\hbox{\rm otherwise}.
    \end{array}
    \right.
\end{equation}
The Euler number $\chi(\omega)$ is the Hopf invariant of $J(\omega)\in\pi_{7}(S^4)$, i.e., $\chi(\omega)=Hopf(J(\omega))\in\mathbb{Z}$. If $\chi(\omega)=1\in\mathbb{Z}$, then $V(\omega)$ is a homotopy $7$-sphere which is boundary of an oriented $8$-dimensional manifold $W(\omega)$. In fact ${}^+\Omega_7=0$ and for $\chi(\omega)=1\in\mathbb{Z}$ one has $H_p(V(\omega))=H_p(S^7)$. Let $k$ be an odd integer and let $\omega_k:S^4\to BSO(4)$ be the classifying map for orientable $4$-plane bundle over $S^4$ with $p_1(\omega_k)=2k$, $\chi(\omega_k)=1\in\mathbb{Z}$. There exists a Morse function $V(\omega_k)\to\mathbb{R}$ with two critical points, such that $V(\omega_k)\setminus\{pt\}\cong\mathbb{R}^7$, hence $V(\omega_k)$ is homeomorphic to $S^7$. Let us investigate under which conditions $V(\omega_k)$ is diffeomorphic to $S^7$ too. So let us assume that such diffeomorphism $f:V(\omega_k)\cong S^7$ exists. Then let us consider the closed oriented $8$-dimensional manifold $M=W(\omega_k)\bigcup_{f}D^8$. For such a manifold we report in {\em(\ref{inresection-form-signature-8-dimensional-manifold-built-starting-4-plane-bundle})} its intersection form and signature.
\begin{equation}\label{inresection-form-signature-8-dimensional-manifold-built-starting-4-plane-bundle}
   \left\{
   \begin{array}{l}
    (H^4(M),\lambda)=(\mathbb{Z},1)\\
    \sigma(M)=\sigma(H^4(M),\lambda)=1.\\
   \end{array}
   \right.
\end{equation}
By the Hirzebruch signature theorem one has $\sigma(M)=<\mathcal{L}_2(p_1,p_2),[M]>=1\in\mathbb{Z}$, with $<\mathcal{L}_2(p_1,p_2),[M]>=\frac{1}{45}(7p_2(M)-p_1(M)^2)=1\in H^8(M)=\mathbb{Z}$, $p_1(M)=2k$, $p_2(M)=\frac{1}{7}(45+4k^2)=\frac{4}{7}(k^2-1)+7\in H^4(M)=\mathbb{Z}$. Since $p_2(M)$ is an integer, it follows that must be $k^2\equiv 1\hskip 2pt(\MOD\: 7)$. This condition on $k$, comes from the assumption that $V(\omega_k)$ is diffeomorphic to $S^7$, therefore, it follows that under the condition $k^2\not\equiv 1\hskip 2pt(\MOD\: 7)$, $V(\omega_k)$ can be only homeomorphic to $S^7$, but not diffeomorphic, hence it is an exotic sphere, and $M$ is only a $8$-dimensional topological manifold, to which the Hirzebruch signature theorem does not apply.
\end{example}

\begin{example}
The $4$-dimensional affine space $\mathbb{R}^4$ has infinity exotic substitutes, just called {\em exotic $\mathbb{R}^4$}. (See \cite{DONALDSON, FREEDMAN}.)\end{example}

The surgery theory is a general algebraic topological framework to decide if a homotopy equivalence between $n$-dimensional manifolds is a diffeomorphism. (See, e.g. \cite{WALL2}.) We shall resume here some definitions and results about this theory. In the following section we will enter in some complementary informations and we will continue to develop such approach in connection with other algebraic topological aspects.

\begin{definition}\label{geometric-poincare-complex}
An {\em $n$-dimensional geometric Poincar\'e complex} is a finite CW complex such that one has the isomorphism $H^p(X;\Lambda)\cong H_{n-p}(X;\Lambda)$, induced by the cap product, i.e. $[\omega]\mapsto [X]\cap[\omega]$, for every $\mathbb{Z}[\pi_1(X)]$-module $\Lambda$.
\end{definition}

\begin{theorem}[Geometric Poincar\'e complex properties]\label{properties-geometric-poincare-complex}
{\em 1)} An $n$-dimensional manifold is an $n$-dimensional geometric Poincar\'e complex.

{\em 2)} Let $X$ be a  geometric Poincar\'e complex and $Y$ another CW complex homotopy related to $X$. Then also $Y$ is a  geometric Poincar\'e complex.

{\em 3)} Any CW complex homotopy equivalent to a manifold is a  geometric Poincar\'e complex.

{\em 4)} Geometric Poincar\'e complexes that are not homotopy equivalent to a manifold may be obtained by gluing together $n$-dimensional manifolds with boundary, $(M,\partial M)$, $(N,\partial N)$, having an homotopic equivalence on the boundaries, $f:\partial M\approxeq\partial N$, which is not homotopic to a diffeomorphism.

{\em 5) (Transfer or Umkehr map).} Let $f:N\to M$ be a mapping between oriented, compact, closed manifolds of arbitrary dimensions. Then the Poincar\'e duality identifies an homomorphism $\tau:H^\bullet(N;\mathbb{Z})\to H^{\bullet-d}(M;\mathbb{Z})$, where $d=\dim N-\dim M$. More precisely one has the commutative diagram {\em(\ref{transfer-map-definition})} that defines $\tau$.
\begin{equation}\label{transfer-map-definition}
  \xymatrix{H^\bullet(N;\mathbb{Z})\ar[d]_{D_N}\ar[r]^{\tau}&H^{\bullet-d}(M;\mathbb{Z})\\
 H_{\dim N-\bullet}(N;\mathbb{Z})\ar[r]_{f_*}&H_{\dim M-\bullet}(M;\mathbb{Z})\ar[u]^{D^{-1}_M}}
\end{equation}
where $d=\dim N-\dim M$. $D_N$ and $D_M$ are the Poincar\'e isomorphisms on $N$ and $M$ respectively. One has $\tau(f^*(x)\cup y)=x\cup\tau(y)$, $\forall x\in H^\bullet(M;\mathbb{Z})$ and $y\in H^\bullet(N;\mathbb{Z})$.\footnote{If $f:N\to M$ is an orientable fiber bundle with compact, orientable fiber $F$, integration over the fiber provides another definition of the transfer map: $\tau:H^\bullet_{de-Rham}(N)\to H_{de-Rham}(M)^{\bullet-r}$, where $r=\dim F$.}

In particular when $f:\widetilde{M}\to M$ is a covering map, then one can write $\tau(x)(\sigma)=x(\sum_{f(\widetilde{\sigma})=\sigma})\widetilde{\sigma}$, $\forall x\in C^\bullet(\widetilde{M})$ and $\sigma\in C_\bullet(M)$.
\end{theorem}

\begin{definition}\label{manifold-structure}
Let $X$ be a closed $n$-dimensional geometric Poincar\'e complex. A {\em manifold structure} $(M,f)$ on $X$ is a closed $n$-dimensional manifold $M$ together with a homotopy equivalence $f:M\approxeq X$. We say hat such two manifold structures $(M,f)$, $(N,g)$ on $X$ are equivalent if there exists a bordism $(F;f,g):(V;M,N)\to X\times(I;\{0\},\{1\})$, with $F$ a homotopy equivalence. (This means that $(V;M,N)$ is an h-cobordism, (see the next section).) Let $\mathfrak{S}(X)$ denote the set of such equivalence classes. We call $\mathfrak{S}(X)$ the {\em manifold structure set} of $X$. $\mathfrak{S}(X)=\varnothing$ means that $X$ is without manifold structures.
\end{definition}

\begin{theorem}[Manifold structure set properties]\label{properties-manifold-structure-set}
{\em 1)} $\mathfrak{S}(X)$ is homotopy invariant of $X$, i.e., a homotopy equivalence $f:X\approxeq Y$ induces a bijection $\mathfrak{S}(X)\to\mathfrak{S}(Y)$.

{\em 2)} A homotopy equivalence $f:M\approxeq N$ of $n$-dimensional manifolds determines an element $(M,f)\in\mathfrak{S}(N)$, such that $f$ is h-cobordant to $1:N\to N$ iff $(M,f)\in[(N,1)]\in\mathfrak{S}(N)$.

{\em 3)} Let $M$ be a $n$-dimensional closed differentiable manifold. If $\mathfrak{S}(X)=\{pt\}$ then $M$ does not admit exotic substitutes.

{\em 4) (Differential structures by gluing manifolds together).} Let $M$ and $N$ be $n$-dimensional manifolds such that their boundary are diffeomorphic: $\partial M\cong\partial N$. Let $\alpha$ and $\beta$ be two differential structures on $M\bigcup_f N$ that agree with the differential structures on $M$ and $N$ respectively. Then there exists a diffeomorphism $h:W_\alpha\cong W_\beta$ such that $h|_M=1_M$.

\end{theorem}
\begin{definition}\label{degree-1-normal-map}
A {\em degree $1$ normal map} from an $n$-dimensional manifold $M$ to an $n$-dimensional geometric Poincar\'e complex $X$ is given by a couple $(f,b)$, where $f:M\to X$ is a mapping such that $f_*[M]=[X]\in H_n(X)$, and $b:\nu_M\to\eta$ is a stable bundle map over $f$, from the stable normal bundle $\nu_M:M\to BO$ to the stable bundle $\eta:X\to BO$. We write also $(f,b):M\to X$.
\end{definition}

\begin{theorem}[Obstructions for manifold structures on geometric Poincar\'e complex]\label{obstructions-manifold-structures}
{\em 1)} Let $X$ be an $n$-dimensional geometric Poincar\'e complex. Then the criterion to decide if $X$ is homotopy equivalent to an $n$-dimensional manifold $M$, is to verify that are satisfied the following two conditions.

{\em (i)} $X$ admits a degreee $1$ normal map $(f,b):M\to X$. This is the case when the map $t(\nu_X):X\to B(G/O)$, given in {\em(\ref{composition-map-stable-normal-bundles})}, is null-homotopic.
\begin{equation}\label{composition-map-stable-normal-bundles}
    \xymatrix{X\ar@/_2pc/[rr]_{t(\nu_X)}\ar[r]&BG\ar[r]&B(G/O)}
\end{equation}
Then there exists a null-homotopy $t(\nu_X)\simeq\{*\}$ iff the Spivak normal fibration $\nu_X:X\to BG=\mathop{\lim}\limits_{\overrightarrow{k}}BG(k)$ admits a vector bundle reduction $\widetilde{\nu_X}:X\to BO$.

{\em (ii)} $(f,b):M\to X$ is bordant to a homotopy equivalence $(g,h):N\approxeq X$.

{\em 2) (J. H. C. Whitehad's theorem).} $f:M\to X$ is a homotopy equivalence iff $\pi_*(f)=0$.  Let $n=2k$, or $n=2k+1$. It is always possible to kill $\pi_i(f)$, for $i\le k$, i.e., there is a bordant degree $1$ normal map $(h,b):N\to X$, with $\pi_i(h)=0$ for $i\le k$.  There exists a normal bordism of $(f,b)$ to a homotopy equivalence iff it is also possible kill $\pi_{k+1}(h)$. In general there exists an obstruction to killing $\pi_{k+1}(h)$, which for $n\ge 5$ is of algebraic nature.

{\em 3) (C. T. C. Wall's surgery obstruction theorem).}\cite{WALL2} For any group $\pi$ there are defined algebraic L-groups $L_{n}(\mathbb{Z}[\pi])$ depending only on $n (\MOD\, 4)$ as group of stable isomorphism classes of $(-1)^k$-quadratic forms over $\mathbb{Z}[\pi]$ for $n=2k$, or as group of stable automorphisms of such forms for $n=2k+1$. An $n$-dimensional degree $1$ normal map $(f,b):N\to X$ has a {\em surgery obstruction} $\sigma_*(f,b)\in L_n(\mathbb{Z}[\pi_1(X)])$, such that $\sigma_*(f,b)=0$ if (and for $n\ge 5$ only if) $(f,b)$ is bordant to a homotopy equivalence.
\end{theorem}

\begin{example}
The simply-connected surgery obstruction groups are given in Tab. \ref{simply-connected-surgery obstruction-groups}.
\begin{table}[h]
\caption{Simply connected surgery obstruction groups}
\label{simply-connected-surgery obstruction-groups}
\begin{tabular}{|c|c|c|c|c|}
  \hline
  \hfil{\rm{\footnotesize $n\: (\MOD\, 4)$}}\hfil& \hfil{\rm{\footnotesize $0$}}\hfil&\hfil{\rm{\footnotesize $1$}}\hfil& \hfil{\rm{\footnotesize $2$}}\hfil&\hfil{\rm{\footnotesize $3$}}\hfil\\
  \hline
 \hfil{\rm{\footnotesize $L_n(\mathbb{Z})$}}\hfil& \hfil{\rm{\footnotesize $\mathbb{Z}$}}\hfil&\hfil{\rm{\footnotesize $0$}}\hfil& \hfil{\rm{\footnotesize $\mathbb{Z}_2$}}\hfil&\hfil{\rm{\footnotesize $0$}}\hfil\\
\hline
\end{tabular}
\end{table}
In particular, we have the following.

$\bullet$\hskip 3pt The surgery obstruction of a $4k$-dimensional normal map $(f,b):M\to X$ with $\pi_1(X)=\{1\}$ is $\sigma_*(f,b)=\frac{1}{8}\sigma(K_{2k}(M),\lambda)\in L_{4k}(\mathbb{Z})=\mathbb{Z}$, with $\lambda$ the nonsingular symmetric form on the middle-dimensional homology kernel $\mathbb{Z}$-module
$$K_{2k}(M)=\ker(f_*:H_{2k}(M)\to H_{2k}(X)).$$

$\bullet$\hskip 3pt The surgery obstruction of a $(4k+2)$-dimensional normal map $(f,b):M\to X$ with $\pi_1(X)=\{1\}$ is $\sigma_*(f,b)=Arf(K_{2k+1}(M;\mathbb{Z}_2),\lambda,\mu)\in L_{4k+2}(\mathbb{Z})=\mathbb{Z}_2$, with $\lambda, \mu$ the nonsingular quadratic form on the middle-dimensional homology $\mathbb{Z}_2$-coefficient homology kernel $\mathbb{Z}_2$-module
$$K_{2k+1}(M;\mathbb{Z}_2)=\ker(f_*:H_{2k+1}(M;\mathbb{Z}_2)\to H_{2k+1}(X;\mathbb{Z}_2)).$$
\end{example}

\begin{theorem}[Browder-Novikov-Sullivan-Wall's surgery exact sequence]\label{surgery-exact-sequence}
One has the following propositions.

{\em (i)} Let $X$ be an $n$-dimensional geometric Poincar\'e complex with $n\ge 5$. The manifold structure set $\mathfrak{S}(X)\not=\varnothing$ iff there exists a normal map $(f,b):M\to X$ with surgery obstruction $\sigma_*(f,b)=0\in L_n(\mathbb{Z}[\pi_1(X)])$.

{\em (ii)} Let $M$ be an $n$-dimensional manifold. Then $\mathfrak{S}(M)$ fits into the surgery exact sequence of pointed sets reported in {\em(\ref{surgery-exact-sequence})}.
\begin{equation}\label{surgery-exact-sequence}
 \xymatrix{\cdots L_{n+1}(\mathbb{Z}[\pi_1(M)])\ar[r]&\mathfrak{S}(M)\ar[r]&[M,G/O]\to L_n(\mathbb{Z}[\pi_1(M)])}.
\end{equation}

\end{theorem}

\section{\bf h-COBORDISM IN RICCI FLOW PDE's}

In this section we shall relate the h-cobordism with the geometric properties of the Ricci flow equation considered in the previous two sections. With this respect let us recall first some definitions and properties about surgery on manifolds.

\begin{table}[t]
\caption{Calculated groups $\Theta_n$ for $1\le n\le 20$ and some related groups.}
\label{calculated-h-cobordism-groups-homotopy-sphere}
\scalebox{0.7}{$\begin{tabular}{|c|c|c|c|c|c|c|c|c|c|c|c|c|c|c|c|c|c|c|c|c|}
  \hline
  \hfil{\rm{\footnotesize $n$}}\hfil& \hfil{\rm{\footnotesize $1$}}\hfil & \hfil{\rm{\footnotesize $2$}}\hfil& \hfil{\rm{\footnotesize $3$}}\hfil & \hfil{\rm{\footnotesize $4$}}\hfil & \hfil{\rm{\footnotesize $5$}}\hfil & \hfil{\rm{\footnotesize $6$}}\hfil & \hfil{\rm{\footnotesize $7$}}\hfil& \hfil{\rm{\footnotesize $8$}}\hfil & \hfil{\rm{\footnotesize $9$}}\hfil & \hfil{\rm{\footnotesize $10$}}\hfil& \hfil{\rm{\footnotesize $11$}}\hfil &\hfil{\rm{\footnotesize $12$}}\hfil & \hfil{\rm{\footnotesize $13$}}\hfil & \hfil{\rm{\footnotesize $14$}}\hfil & \hfil{\rm{\footnotesize $15$}}\hfil&\hfil{\rm{\footnotesize $16$}}\hfil & \hfil{\rm{\footnotesize $17$}}\hfil&  \hfil{\rm{\footnotesize $18$}}& \hfil{\rm{\footnotesize $19$}}\hfil& \hfil{\rm{\footnotesize $20$}}\hfil  \\
  \hline
  \hfil{\rm{\footnotesize $\Theta_n$}}\hfil& \hfil{\rm{\footnotesize $0$}}\hfil & \hfil{\rm{\footnotesize $0$}}\hfil& \hfil{\rm{\footnotesize $0$}}\hfil & \hfil{\rm{\footnotesize $0$}}\hfil & \hfil{\rm{\footnotesize $0$}}\hfil & \hfil{\rm{\footnotesize $0$}}\hfil & \hfil{\rm{\footnotesize $\mathbb{Z}_{28}$}}\hfil& \hfil{\rm{\footnotesize $\mathbb{Z}_{2}$}}\hfil & \hfil{\rm{\footnotesize $\mathbb{Z}_{8}$}}\hfil & \hfil{\rm{\footnotesize $\mathbb{Z}_{6}$}}\hfil& \hfil{\rm{\footnotesize $\mathbb{Z}_{992}$}}\hfil &\hfil{\rm{\footnotesize $0$}}\hfil & \hfil{\rm{\footnotesize $\mathbb{Z}_{3}$}}\hfil & \hfil{\rm{\footnotesize $\mathbb{Z}_{2}$}}\hfil & \hfil{\rm{\footnotesize $\mathbb{Z}_{16256}$}}\hfil&\hfil{\rm{\footnotesize $\mathbb{Z}_{2}$}}\hfil & \hfil{\rm{\footnotesize $\mathbb{Z}_{16}$}}\hfil& \hfil{\rm{\footnotesize $\mathbb{Z}_{16}$}}\hfil& \hfil{\rm{\footnotesize $\mathbb{Z}_{523264}$}}\hfil& \hfil{\rm{\footnotesize $\mathbb{Z}_{24}$}}\hfil \\
  \hline
\hfil{\rm{\footnotesize $bP_{n+1}$}}\hfil& \hfil{\rm{\footnotesize $0$}}\hfil & \hfil{\rm{\footnotesize $0$}}\hfil& \hfil{\rm{\footnotesize $0$}}\hfil & \hfil{\rm{\footnotesize $0$}}\hfil & \hfil{\rm{\footnotesize $0$}}\hfil & \hfil{\rm{\footnotesize $0$}}\hfil & \hfil{\rm{\footnotesize $\mathbb{Z}_{28}$}}\hfil& \hfil{\rm{\footnotesize $0$}}\hfil & \hfil{\rm{\footnotesize $\mathbb{Z}_{2}$}}\hfil & \hfil{\rm{\footnotesize $0$}}\hfil& \hfil{\rm{\footnotesize $\mathbb{Z}_{992}$}}\hfil &\hfil{\rm{\footnotesize $0$}}\hfil & \hfil{\rm{\footnotesize $0$}}\hfil & \hfil{\rm{\footnotesize $0$}}\hfil & \hfil{\rm{\footnotesize $\mathbb{Z}_{8128}$}}\hfil&\hfil{\rm{\footnotesize $0$}}\hfil & \hfil{\rm{\footnotesize $\mathbb{Z}_{2}$}}\hfil& \hfil{\rm{\footnotesize $0$}}\hfil& \hfil{\rm{\footnotesize $\mathbb{Z}_{261632}$}}\hfil& \hfil{\rm{\footnotesize $0$}}\hfil \\
  \hline
\hfil{\rm{\footnotesize $\Theta_n/bP_{n+1}$}}\hfil& \hfil{\rm{\footnotesize $0$}}\hfil & \hfil{\rm{\footnotesize $0$}}\hfil& \hfil{\rm{\footnotesize $0$}}\hfil & \hfil{\rm{\footnotesize $0$}}\hfil & \hfil{\rm{\footnotesize $0$}}\hfil & \hfil{\rm{\footnotesize $0$}}\hfil & \hfil{\rm{\footnotesize $0$}}\hfil& \hfil{\rm{\footnotesize $\mathbb{Z}_2$}}\hfil & \hfil{\rm{\footnotesize $\mathbb{Z}_{4}$}}\hfil & \hfil{\rm{\footnotesize $\mathbb{Z}_6$}}\hfil& \hfil{\rm{\footnotesize $0$}}\hfil &\hfil{\rm{\footnotesize $0$}}\hfil & \hfil{\rm{\footnotesize $\mathbb{Z}_3$}}\hfil & \hfil{\rm{\footnotesize $\mathbb{Z}_2$}}\hfil & \hfil{\rm{\footnotesize $\mathbb{Z}_{2}$}}\hfil&\hfil{\rm{\footnotesize $\mathbb{Z}_2$}}\hfil & \hfil{\rm{\footnotesize $\mathbb{Z}_{8}$}}\hfil& \hfil{\rm{\footnotesize $\mathbb{Z}_{16}$}}\hfil& \hfil{\rm{\footnotesize $\mathbb{Z}_{2}$}}\hfil& \hfil{\rm{\footnotesize $\mathbb{Z}_{24}$}}\hfil \\
  \hline
\hfil{\rm{\footnotesize $\pi^s_{n}/J$}}\hfil& \hfil{\rm{\footnotesize $0$}}\hfil & \hfil{\rm{\footnotesize $\mathbb{Z}_{2}$}}\hfil& \hfil{\rm{\footnotesize $0$}}\hfil & \hfil{\rm{\footnotesize $0$}}\hfil & \hfil{\rm{\footnotesize $0$}}\hfil & \hfil{\rm{\footnotesize $\mathbb{Z}_{2}$}}\hfil & \hfil{\rm{\footnotesize $0$}}\hfil& \hfil{\rm{\footnotesize $\mathbb{Z}_{2}$}}\hfil & \hfil{\rm{\footnotesize $\mathbb{Z}_{4}$}}\hfil & \hfil{\rm{\footnotesize $\mathbb{Z}_{6}$}}\hfil& \hfil{\rm{\footnotesize $0$}}\hfil &\hfil{\rm{\footnotesize $0$}}\hfil & \hfil{\rm{\footnotesize $\mathbb{Z}_{3}$}}\hfil & \hfil{\rm{\footnotesize $\mathbb{Z}_{4}$}}\hfil & \hfil{\rm{\footnotesize $\mathbb{Z}_{2}$}}\hfil&\hfil{\rm{\footnotesize $\mathbb{Z}_2$}}\hfil & \hfil{\rm{\footnotesize $\mathbb{Z}_{8}$}}\hfil& \hfil{\rm{\footnotesize $\mathbb{Z}_{16}$}}\hfil& \hfil{\rm{\footnotesize $\mathbb{Z}_{2}$}}\hfil& \hfil{\rm{\footnotesize $\mathbb{Z}_{24}$}}\hfil \\
  \hline
\hfil{\rm{\footnotesize $\pi_n^s$}}\hfil& \hfil{\rm{\footnotesize $\mathbb{Z}_2$}}\hfil&\hfil{\rm{\footnotesize $\mathbb{Z}_2$}}\hfil& \hfil{\rm{\footnotesize $\mathbb{Z}_{24}$}}\hfil& \hfil{\rm{\footnotesize $0$}}\hfil& \hfil{\rm{\footnotesize $0$}}\hfil&\hfil{\rm{\footnotesize $\mathbb{Z}_2$}}\hfil& \hfil{\rm{\footnotesize $\mathbb{Z}_{240}$}}\hfil&\hfil{\rm{\footnotesize $\mathbb{Z}_{4}$}}\hfil&\hfil{\rm{\footnotesize $\mathbb{Z}_{8}$}}\hfil&\hfil{\rm{\footnotesize $\mathbb{Z}_{6}$}}\hfil&\hfil{\rm{\footnotesize $\mathbb{Z}_{504}$}}\hfil&\hfil{\rm{\footnotesize $0$}}\hfil&\hfil{\rm{\footnotesize $\mathbb{Z}_{3}$}}\hfil&\hfil{\rm{\footnotesize $\mathbb{Z}_{4}$}}\hfil&\hfil{\rm{\footnotesize $\mathbb{Z}_{960}$}}\hfil&\hfil{\rm{\footnotesize $\mathbb{Z}_{4}$}}\hfil&\hfil{\rm{\footnotesize $\mathbb{Z}_{16}$}}\hfil&\hfil{\rm{\footnotesize $-$}}\hfil&\hfil{\rm{\footnotesize $-$}}\hfil&\hfil{\rm{\footnotesize $-$}}\hfil\\
\hline
  \hfil{\rm{\footnotesize $J$}}\hfil& \hfil{\rm{\footnotesize $\mathbb{Z}_2$}}\hfil&\hfil{\rm{\footnotesize $0$}}\hfil& \hfil{\rm{\footnotesize $\mathbb{Z}_{24}$}}\hfil& \hfil{\rm{\footnotesize $0$}}\hfil& \hfil{\rm{\footnotesize $0$}}\hfil&\hfil{\rm{\footnotesize $0$}}\hfil& \hfil{\rm{\footnotesize $\mathbb{Z}_{240}$}}\hfil&\hfil{\rm{\footnotesize $\mathbb{Z}_{2}$}}\hfil&\hfil{\rm{\footnotesize $\mathbb{Z}_{2}$}}\hfil&\hfil{\rm{\footnotesize $0$}}\hfil&\hfil{\rm{\footnotesize $\mathbb{Z}_{504}$}}\hfil&\hfil{\rm{\footnotesize $0$}}\hfil&\hfil{\rm{\footnotesize $0$}}\hfil&\hfil{\rm{\footnotesize $0$}}\hfil&\hfil{\rm{\footnotesize $\mathbb{Z}_{480}$}}\hfil&\hfil{\rm{\footnotesize $\mathbb{Z}_{2}$}}\hfil&\hfil{\rm{\footnotesize $\mathbb{Z}_{2}$}}\hfil&\hfil{\rm{\footnotesize $-$}}\hfil&\hfil{\rm{\footnotesize $-$}}\hfil&\hfil{\rm{\footnotesize $-$}}\hfil\\
\hline
\multicolumn{21}{l}{\rm{\footnotesize $\dim_{\mathbb{Z}}\Theta_n=$ number of differential structures on the $n$-dimensional homotopy sphere.}}\\
   \multicolumn{21}{l}{\rm{\footnotesize $bP_{n+1}\triangleleft \Theta_n$: subgroup of $n$-dimensional homotopy spheres bounding parallelizable manifolds.}}\\
  \multicolumn{21}{l}{\rm{\footnotesize $bP_{n+1}$ is a finite cyclic group that vanishes if $n$ is even.}}\\
   \multicolumn{21}{l}{\rm{\footnotesize Kervaire-Milnor formula: $\dim_{\mathbb{Z}}bP_{4n}=\frac{3-(-1)^n}{2}2^{2n-2}(2^{2n-1}-1){\rm Numerator}(\frac{B_{4n}}{4n})$, $n\ge 2$, with $B_{4n}$ Bernoulli numbers.}}\\
   \multicolumn{21}{l}{\rm{\footnotesize  $\dim_{\mathbb{Z}}bP_{4n+2}=0$, $n=0,1,3,7,15$;  $\dim_{\mathbb{Z}}bP_{4n+2}=0$, or $\mathbb{Z}_2$, $n=31$; $\dim_{\mathbb{Z}}bP_{4n+2}=\mathbb{Z}_2$ otherwise.}}\\
  \multicolumn{21}{l}{\rm{\footnotesize $\pi^s_n\equiv\mathop{\lim}\limits_{\overrightarrow{k}}\pi_{n+k}(S^k)$: stable homotopy groups or $n$-stems. (Serre's theorem. The groups $\pi^s_n$ are finite.)}}\\
\multicolumn{21}{l}{\rm{\footnotesize There is an injective map $\Theta_n/bP_{n+1}\to\pi^s_n/J$ where $J$ is the image of Whitehead's $J$-homomorphisms $\pi_n(SO)\to\pi_n^s$.}}\\
\multicolumn{21}{l}{\rm{\footnotesize For $n=0$ one has $\pi_n^s=\mathbb{Z}$ and $J=0$.}}\\
\end{tabular}$}
\end{table}

\begin{definition}
{\em 1)} The {\em$n$-dimensional handle}, of {\em index
$p$}, is $h^{p}\equiv D^{p}\times
D^{n-p}$. Its {\em core} is $D^{p}\times\{0\}$. The
{\em boundary of the core} is $S^{p-1}\times\{0\}$. Its
{\em cocore} is $\{0\}\times  D^{n-p}$ and its {\em
transverse sphere} is $\{0\}\times
S^{n-p-1}$.

{\em 2)} Given a topological space $Y$, the images of continuous maps $D^{n}\to Y$ are called the {\em$n$-cells} of
$Y$.

{\em 3)} Given a topological space $X$ and a continuous map
$\alpha:S^{n-1}\to X$, we call $Y\equiv X\bigcup_\alpha
D^{n}$ obtained from $X$ by {\em attaching a $n$-dimensional cell} to $X$.

{\em 4)} We call {\em CW-complex} a topological space $X$ obtained from $\varnothing$ by successively attaching cells of non-decreasing dimension:
\begin{equation}
X\equiv(\cup D^{0})\cup D^{1}\cup D^{2})\cup\cdots
\end{equation}
We call $X^{n}\equiv\bigcup_{1\le i\le n}D^{i}$, $n\ge 0$, the {\em$(n)$-skeleta}.
\end{definition}
\begin{definition}{\em(Homotopy groups.)}
{\em 1)} We define {\em homotopy groups} of manifold, (resp. CW-complex), $M$, the groups
\begin{equation}
\pi_{p}(M)=[S^{p},M],\: p\ge 0.
\end{equation}

{\em 2)} Let $X$ be a manifold over a
CW-complex and an element $x\in\pi_{n}(X)$, $n\ge 1$. Let
$Y=X\bigcup_{\phi^{n}}D^{n+1}$ be the
CW-complex obtained from $X$ by attaching an $(n+1)$-cell
with map $\phi^{n}: S^{n}\to X$, with
$x=[\phi^{n}]\in\pi_{n}(X)$. The operation of attaching
the $(n+1)$-cell is said to {\em kill} $x$.
\end{definition}

\begin{theorem}{\em(CW-substitute)}\label{ma-CW}
{\em 1)} For any manifold, $M$, we can construct a CW-complex $X$ and a weak
homotopy equivalence $f:X\to M$, (i.e., the induced maps
$f_*:\pi_r(X')\to\pi_r(X)$ on the Hurewicz homotopy groups are
bijective for $r\ge 0$).\footnote{Note that an homotopy
equivalence is an weak homotopy equivalence, but the vice versa is
not true. Recall that two pointed topological spaces $(X,x_0)$ and
$(Y,y_0)$ have the same homotopy type if
$\pi_1(X,x_0)\cong\pi_1(Y,y_0)$, and $\pi_n(X,x_0)$ and $
\pi_n(Y,y_0)$ are isomorphic as modules over ${\mathbb
Z}[\pi_1(X,x_0)]$ for $n\ge 2$. A {\em simply homotopy equivalence} between $m$-dimensional manifolds, (or finite CW complexes), is a homotopy equivalence $f:M\approxeq N$ such that the Whitehead torsion $\tau(f)\in Wh(\pi_1(M))$, where $Wh(\pi_1(M))$ is the Whitehead group of $\pi_1(M)$. With this respect, let us recall that if $A$ is an associative ring with unity, such that $A^m$ is isomorphic to $A^n$ iff $m=n$, put $GL(A)\equiv\bigcup_{n-1}GL_n(A)$, the {\em infinite general linear group of $A$} and $E(A)\equiv[GL(A),GL(A)]\triangleleft GL(A)$. $E(A)$ is the normal subgroup generated by the elementary matrices $\scalebox{0.6}{$\left(
                                      \begin{array}{cc}
                                        1 & a \\
                                        0 & 1 \\
                                      \end{array}\right)$}$. The {\em torsion group} $K_1(A)$ is the abelian group $K_1(A)=GL(A)/E(A)$.  Let $A^\bullet$ denote the multiplicative group of units in the ring $A$. For a commutative ring $A$, the inclusion $A^\bullet\hookrightarrow K_1(A)$ splits by the determinant map $\det:K_1(A)\to A^\bullet$, $\tau(\phi)\mapsto \det(\phi)$ and one has the splitting $K_1(A)=A^\bullet\bigoplus SK_1(A)$, where $SK_1(A)=\ker(\det:K_1(A)\to A^\bullet)$. If $A$ is a field, then $K_1(A)\cong A^\bullet$ and $SK_1(A)=0$. The {\em torsion} $\tau(f)$ of an isomorphism $f:L\cong K$ of finite generated free $A$-modules of rank $n$, is the torsion of the corresponding invertible matrix $(f^i_j)\in GL_n(A)$, i.e., $\tau(f)=\tau(f^j_i)\in K_1(A)$. The isomorphism is {\em simple} if $\tau(f)=0\in K_1(A)$. The Whitehead group of a group $G$ is the abelian group $Wh(G)\equiv K_1(\mathbb{Z}[G])/\{\tau(\mp g) | g\in G\}$. $Wh(G)=0$ in the following cases: (a) $G=\{1\}$; (b) $G=\pi_1(M)$, with $M$ a surface; (c) $G=\mathbb{Z}^m$, $m\ge 1$. There is a conjecture, (Novikov) that extends the case (b) also to $m$-dimensional compact manifolds $M$ with universal cover $\widetilde{M}=\mathbb{R}^m$. This conjecture has been verified in many cases \cite{FERRY-RANICKI-ROSENBERG}. } Then $X'$ is called the
CW-substitute of $X$. This is unique up to homotopy.

{\em 2)} Furthermore if $h:X\to Y$ is a continuous map between
manifolds, and $(X',f)$, $(Y',g)$ are the
corresponding CW-substitutes, then we can find a cellular map
$hì:X'\to Y'$ so that the following diagram is commutative:

\begin{equation}
\xymatrix@C=60pt{ X'\ar[d]_{h'} \ar[r]^{f}& X\ar[d]^{h}\\
Y'\ar[r]_{g} & Y\\}
\end{equation}
$h'$ is unique up to homotopy.
\end{theorem}

\begin{definition}
An {\em$n$-dimensional bordism} $(W;M_0,f_0;M_1,f_1)$ consists
of a compact manifold $W$ of dimension $n$, and
closed $(n-1)$-dimensional manifolds $M_0$,
$M_1$, such that $\partial W=N_0\sqcup N_1$, and
diffeomorphisms $f_i:M_i\cong N_i$, $i=0,1$. An
{\em$n$-dimensional h-bordism} (resp. {\em$s$-bordism}) is a
$n$-dimensional bordism as above, such that the inclusions
$N_i\hookrightarrow W$, $i=0,1$, are homotopy equivalences (resp.
simply homotopy equivalences).\footnote{Let us emphasize that to state that the inclusions $M_i\hookrightarrow W$, $i=0,1$, are homotopy equivalences is equivalent to state the $M_i$ are deformation retracts of $W$.} $W$ is a {\em trivial h-bordism} if $W\cong M_0\times[0,1]$. In such a case $M_0$ is diffeomorphic to $M_1$: $M_0\cong M_1$

We will simply denote also by $(W;M_0,M_1)$ a $n$-dimensional
bordism.

If $\phi^{p}:S^{p+1}\times D^{n-p-1}\to M_1$ is
an embedding, then

\begin{equation}
W+(\phi^{p})\equiv W\bigcup_{\phi^{p}}D^{p}\times
D^{n-p}\equiv W\bigcup h^{p} \end{equation}

is said obtained from $W$ by {\em attaching a handle},
$h^{p}\equiv D^{p}\times D^{n-p}$, of {\em
index $p$} by $\phi^{p}$.\footnote{In general $W\bigcup
h^{p}$ is not a manifold but a CW-complex.} Put $\partial(W+\phi^{p})_0=M_0$,
$\partial(W+\phi^{p})_1=\partial(W+\phi^{p})-M_0$.
\end{definition}

\begin{theorem}[CW-substitute of manifold and Hurewicz morphisms]\label{ma2}
For any manifold $M$ we can construct a
CW-complex $M'$ and a weak homotopy equivalence $f:
M'\to M$. Then $M'$ is called the {\em
CW-substitute} of $M$. $M'$ is unique up to homotopy. Then the
homotopy groups of $M$ and  $M'$ are isomorphic, i.e., one has the top horizontal exact short sequence reported in the commutative diagram {\em(\ref{comm-diag1})}. There the vertical lines represent the Hurewicz morphisms relating homotopy groups and homology groups.
\begin{equation}\label{comm-diag1}
\xymatrix{0\ar[r]&\pi_p(M)\ar[d]_{a}\ar[r]&\pi_p( M'))\ar[d]_{a'}\ar[r]&0\\
0\ar[r]&H_p(M)\ar[r]&H_p( M')\ar[r]&0}
\end{equation}
If $M$ is $(n-1)$-connected, $n\ge 2$, then the morphisms $a$, $a'$,
become isomorphisms for $p\le n$ and epimorphisms for
$p=n+1$.

We call the morphisms $a$ and $a'$ the {\em Hurewicz morphisms} of
the manifold $M$ and $M'$ respectively.
\end{theorem}

\begin{definition}
A {\em $p$-surgery} on a manifold $M$
of dimension $n$ is the
procedure of construction a new $n$-dimensional manifold:\footnote{We say also that a
$p$-surgery removes a {\em framed
$p$-embedding} $g:S^{p}\times
D^{n-p}\hookrightarrow M$. Then it kills the
homotopy class $[g]\in\pi_{p}(M)$ of the core $g=g|:
S^{p}\times\{0\}\hookrightarrow M$.}
\begin{equation}\label{p-surgery-definition}
N\equiv\overline{(M\setminus S^{p}\times
D^{n-p})}\bigcup_{S^{p}\times S^{n-p-1}}
D^{p+1}\times S^{n-p-1}.
\end{equation}
\end{definition}

\begin{example}
Since for the $n$-dimensional sphere $S^{n}$
we can write
\begin{equation}
\begin{array}{ll}
S^{n}&=\partial D^{n+1}=\partial(D^{p+1}\times D^{n-p})\\
&=S^{p}\times D^{n-p}\bigcup D^{p+1}\times S^{n-p-1}\\
\end{array}
\end{equation}
it follows that the surgery removing $S^{p}\times D^{n-p}\subset S^{n}$ converts $
S^{n}$ into the product of two spheres
\begin{equation}
D^{p+1}\times S^{n-p-1}\bigcup_{S^{p}\times S^{n-p-1}}
D^{p+1}\times S^{n-p-1}= S^{p+1}\times S^{n-p-1}.
\end{equation}
\end{example}

\begin{theorem}[Surgery and Euler characteristic]
{\em 1)} Let $M$ be a $2n$-dimensional smooth manifold and let apply to $N$ obtained by $M$ with a $p$-surgery as defined in {\em(\ref{p-surgery-definition})}. Then the Euler characteristic of $N$ is related to the $M$ one, by the relation reported in {\em(\ref{p-surgery-relation-with-euler-characteristic})}.
\begin{equation}\label{p-surgery-relation-with-euler-characteristic}
\chi(N)=\left\{\begin{array}{ll}
\chi(M)+2&p={\rm odd}\\
\chi(M)-2&p={\rm even}.\\
\end{array}\right.
\end{equation}

{\em 2)} Let $M=2n+1$, $n\ge 0$. If $M=\partial V$, then $V$ can be chosen a manifold with $\chi(V)=0$, i.e., having the same Euler characteristic of $M$.

{\em 3)} Let $M=2n$, $n\ge 0$. If $M=\partial V$, then $\chi(M)=2\chi(V)$.
\end{theorem}

\begin{proof}
1) Let us first note that we can write $M=\overline{(M\setminus S^{p}\times
D^{n-p})}\bigcup(S^p\times D^{2n-p})$, hence we get
\begin{equation}\label{p-surgery-relation-with-euler-characteristic-a}
\left\{\begin{array}{ll}
\chi(M)&=\chi(\overline{(M\setminus S^{p}\times
D^{n-p})})+\chi(S^p\times D^{2n-p})\\
&=\chi(\overline{M\setminus S^{p}\times
D^{n-p}})+(1+(-1)^p).\\
\end{array}\right.\end{equation}
From (\ref{p-surgery-relation-with-euler-characteristic-a}) we get
\begin{equation}\label{p-surgery-relation-with-euler-characteristic-b}
\chi(\overline{M\setminus S^{p}\times
D^{n-p}})=\chi(M)-(1+(-1)^p)=\left\{\begin{array}{ll}
\chi(M)&p={\rm odd}\\
\chi(M)-2&p={\rm even}.\\
\end{array}\right.
\end{equation}
On the other hand one has
\begin{equation}\label{p-surgery-relation-with-euler-characteristic-c}
\left\{\begin{array}{ll}
\chi(N)&=\chi(\overline{M\setminus S^{p}\times
D^{n-p}})+\chi(D^{p+1}\times S^{2n-p-1})-\chi(S^{p}\times
S^{2n-p-1})\\
&=\chi(\overline{M\setminus S^{p}\times
D^{n-p}})+\left\{\begin{array}{ll}
2&p={\rm odd}\\
0&p={\rm even}\\
\end{array}\right\}.
\end{array}\right.
\end{equation}
Then from (\ref{p-surgery-relation-with-euler-characteristic-b}) and (\ref{p-surgery-relation-with-euler-characteristic-c}) we get
\begin{equation}\label{p-surgery-relation-with-euler-characteristic-d}
\chi(N)=\left\{\begin{array}{ll}
\chi(M)+2&p={\rm odd}\\
\chi(M)-2&p={\rm even}.\\
\end{array}\right.
\end{equation}

2) In fact, if $n=1$ then $M$ can be considered the boundary of a M\"obius strip $M_{ob}$, that has just $\chi(M_{ob})=0$. If $n\ge 3$, and $\chi(V)=2q$, we can add to $V$ $q$ times $p$-surgeries with $p$ even in order to obtain a manifold $V'$ that has the same dimension and boundary of $V$ but with Euler characteristic zero. Furthermore, if $\chi(V)=2q+1$, we consider the manifold $V''=V\sqcup \mathbb{R}P^{2n+1}$ that has the same dimension and boundary of $V$, but $\chi(V'')$ is even. Then we can proceed as before on $V''$.

3) Let us consider $V'=V\bigcup_MV$. Then one has $\chi(V')=0=2\chi(V)-\chi(M)$.
\end{proof}
\begin{example}
A {\em connected sum} of connected $n$-dimensional
manifolds $M$ and $N$ is the connected
$n$-dimensional manifold
\begin{equation}
M\sharp N=(M\setminus D^{n})\bigcup(S^{n-1}\times
D^{1})\bigcup(N\setminus D^{n}).
\end{equation}
$M\sharp N$ is the effect of the $0$-surgery on the disjoint
union $M\sqcup N$ which removes the framed $0$-embedding $
S^0\times D^{n}\hookrightarrow M\times N$ defined by the
disjoint union of the embeddings $D^{n}\hookrightarrow M$,
$D^{n}\hookrightarrow N$. \end{example}

\begin{example}
Given a $(n+1)$-dimensional manifold with boundary
$(M,\partial M)$ and an embedding $S^{i-1}\times
D^{n-i+1}\hookrightarrow\partial M$, $0\le i\le n+1$, we define the $(n+1)$-dimensional
manifold $(W,\partial W)$ obtained from $M$ by attaching a $i$-handle:
\begin{equation}
W=M\bigcup_{S^{i-1}\times D^{n-i+1}}D^{i}\times D^{n-i+1}=M\bigcup h^{i}.
\end{equation}
Then $\partial W$ is obtained from $\partial M$ by an
$(i-1)$-surgery:
\begin{equation}
\partial W=(\partial M\setminus S^{i-1}\times  D^{n-i+1})\bigcup_{
S^{i-1}\times S^{n-i}}D^{i}\times S^{n-1}.
\end{equation}
\end{example}

\begin{definition}
An {\em elementary $(n+1)$-dimensional bordism of index
$i$} is the bordism $(W;M,N)$ obtained from $M\times
D^{1}$ by attaching a $i$-handle at $
S^{i-1}\times D^{n-i+1}\hookrightarrow M\times\{1\}$.
The {\em dual of an elementary $(n+1)$-dimensional bordism}
$(W;M,N)$ of index $i$ is the elementary
$(n+1)$-dimensional bordism $(W;N,M)$ of index
$(n-i+1)$, obtained by reversing the ends and regarding the
$i$-handle attached to $M\times D^{1}$ as a
$(n-i+1)$-handle attached to $N\times D^{1}$.
\end{definition}

\begin{theorem}[Handle decomposition of bordisms in the category $\mathfrak{M}_\infty$]\label{handle-decomposition-bordisms}
{\em 1)} Every bordism $(W;M,N)$, $\dim W=n+1$, $\dim M=\dim
N=n$, has a handle decomposition of the union of a finite sequence
\begin{equation}
(W;M,N)=(W_1;M_0,M_1)\bigcup(W_2;M_1,M_2)\bigcup\cdots\bigcup(W_k;M_{k-1},M_k)
\end{equation}
of adjoining elementary bordisms $(W_s;M_{s-1},M_s)$ with index
$(i_s)$ such that $0\le i_1\le i_2\le\cdots\le i_k\le n+1$.

{\em 2)} Closed $n$-dimensional manifolds $M$, $N$
are  bordant iff $N$ can be obtained from $M$ by a sequence of
surgeries.

{\em 3)} Every closed $n$-dimensional manifold $M$
can be obtained from $\varnothing$ by attaching handles:
\begin{equation}
M=h^{i_0}\bigcup h^{i_1}\bigcup\cdots\bigcup h^{i_k}.
\end{equation}
Furthermore, $M$ has a Morse function $f:M\to\mathbb{R}$ with critical points $\{x_{i_0},x_{i_1},\cdots,x_{i_k},\}$, where $x_\lambda$ is a critical point with index $\lambda$, and the corresponding vector field $\zeta=\GRAD f:M\to TM$ has zero-value only at such critical points. (See Fig. \ref{passing-through-critical-point-attaching-handle}.)
\end{theorem}

\begin{proof}
In fact any $n$-dimensional manifold can be characterized by means of its
corresponding CW-substitute.
\end{proof}
\begin{figure}[h]
\centering
\centerline{\includegraphics[height=4cm]{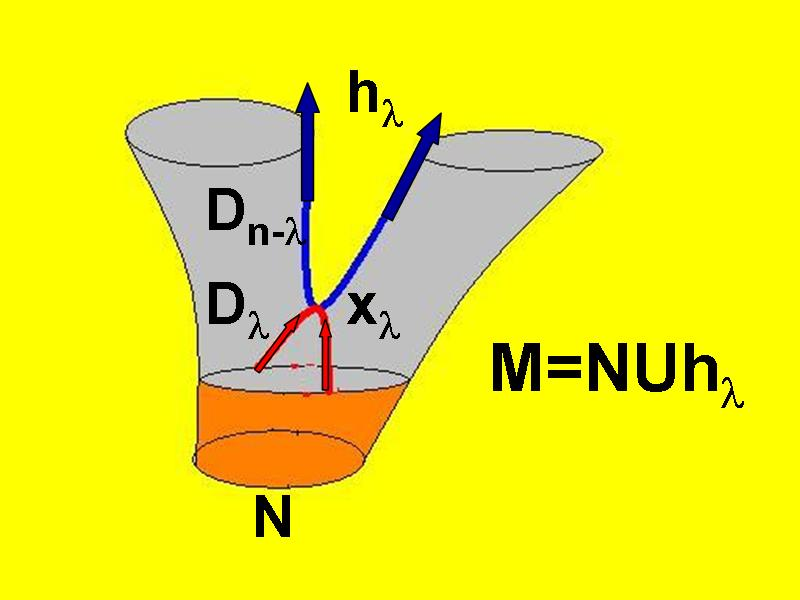}}
\caption{Passing through critical point $x_\lambda\in M$, (index $\lambda$), of Morse function $f:M\to\mathbb{R}$, identified by attaching handle to a manifold $N$. (Separatrix diagram.)}
\label{passing-through-critical-point-attaching-handle}
\end{figure}

\begin{example}[Sphere $S^{2}$]
In this case one has the following handle decomposition: $
S^{2}=h^{0}\bigcup h^{2}$, with $ h^{0}=
D^{0}\times D^{2}=\left\{0\right\}\times D^{2}$, the
{\em south hemisphere}, and $h^{2}=
D^{2}\times D^{0}=D^{2}\times\left\{1\right\}$, the
{\em north hemisphere}.
\end{example}

\begin{example}[Torus $T^{2}=S^{1}\times S^{1}$]
This $2$-dimensional manifold has the following
CW-complex structure: $T^{2}=
h^{0}\bigcup h^{1}\bigcup h^{1}\bigcup h^{2}$,
with $ h^{0}=\left\{0\right\}\times D^{2}$, $
h^{1}= D^{1}\times D^{1}$, $h^{2}=
D^{2}\times  D^{0}=D^{2}\times\left\{1\right\}$.
\end{example}

\begin{remark}
One way to prove whether two manifolds are
diffeomorphic is just to suitably use bordism and surgery techniques. (See, e.g.
Refs.{\em\cite{WALL1, WALL2}}.) In fact we should first prove that they are
bordant and then see if some bordism can be modified by successive
surgeries on the interior to become an s-bordism.
\end{remark}

\begin{theorem}[Homology properties in $\mathfrak{M}_\infty$.]\label{ma5}
Let $M$ be a $n$-dimensional manifold. One has the
following homology structures.

Let $\left\{C_\bullet(M;A)\cong
A\otimes_{\mathbb{R}}C_\bullet(M;\mathbb{R}),\partial\right\}$ be the
chain complex extension of the singular chain complex of $M$. Then
one has the following exact commutative diagram:

\begin{equation}
\xymatrix{&&0\ar[d]&0\ar[d]&&\\
&0\ar[r]&B_{\bullet} (M;A)\ar[d]\ar[r]&
          Z_{\bullet} (M;A)\ar[d]\ar[r]&
          H_{\bullet} (M;A)\ar[r]&0\\
&&C_{\bullet} (M;A)\ar[d]\ar@{=}[r]&
          C_{\bullet} (M;A)\ar[d]&&\\
0\ar[r]&{}^A{\underline{\Omega}}_{\bullet,s}(M)\ar[r]& Bor_{\bullet}
(M;A)\ar[d]\ar[r]&Cyc_{\bullet} (M;A)\ar[d]\ar[r]&0&\\
&&0&0&&}\end{equation}

where:
$$\left\{\begin{array}{l}
B_{\bullet} (M;A)=\ker(\partial|_{\bar C_{\bullet} (M;A)});\hskip
2pt Z_{\bullet} (M;A)=\IM(\partial|_{C_{\bullet} (M;A)});\\
H_{\bullet} (M;A)=Z_{\bullet} (M;A)/B_{\bullet} (M;A),\\
b\in[a]\in Bor_{\bullet}(M;A)\Rightarrow a-b=\partial c, \hskip 3pt
c\in C_{\bullet}(M;A);\\
b\in[a]\in Cyc_{\bullet}(M;A)\Rightarrow \partial(a-b)=0;\\
b\in[a]\in {}^A{\underline{\Omega}}_{\bullet,s}(M)\Rightarrow
                 \left\{\begin{array}{l}
                          \partial a=\partial b=0\\
                          a-b=\partial c,\quad c\in C_{\bullet}(M;A)\\
                          \end{array}
                                \right\}.\\
\end{array}\right.$$
Furthermore, one has the following canonical isomorphism:
${}^A{\underline{\Omega}}_{\bullet,s}(M)\cong H_{\bullet}(M;A)$. As
$C_{\bullet}(M;A)$ is a free two-sided projective $A$-module, one
has the unnatural isomorphism: $Bor_{\bullet}(M;A)\cong
{}^A{\underline{\Omega}}_{\bullet,s}(M)\bigoplus
Cyc_{\bullet}(M;A)$.
\end{theorem}

\begin{proof}
It follows from standard results in homological algebra and homology
in topological spaces. (For more details about see also \cite{PRA1}.)
\end{proof}

\begin{theorem}[Cohomology properties in $\mathfrak{M}_\infty$]\label{ma6}
Let $M$ be a $n$-dimensional manifold. One has the
following cohomology structures.

Let $\left\{C^\bullet(M;A)\equiv Hom_{A}(
C_\bullet(M;A);A)\cong
Hom_{\mathbb{R}}(C_\bullet(M;\mathbb{R});A),\delta\right\}$ be the dual of
the chain complex $C_\bullet(M;A)$ considered in above theorem. Then
one has the following exact commutative diagram:

\begin{equation}
\xymatrix{&&0\ar[d]&0\ar[d]&&\\
&0\ar[r]&B^{\bullet} (M;A)\ar[d]\ar[r]&
          Z^{\bullet} (M;A)\ar[d]\ar[r]&
          H^{\bullet} (M;A)\ar[r]&0\\
&&C^{\bullet} (M;A)\ar[d]\ar@{=}[r]&
          C^{\bullet} (M;A)\ar[d]&&\\
0\ar[r]&{}^A{\underline{\Omega}}^{\bullet}_s(M)\ar[r]& Bor^{\bullet}
(M;A)\ar[d]\ar[r]&Cyc^{\bullet} (M;A)\ar[d]\ar[r]&0&\\
&&0&0&&}\end{equation}

where:
$$\left\{\begin{array}{l}
B^{\bullet} (M;A)=\ker(\delta|_{\bar C^{\bullet} (M;A)});\hskip
2pt
Z^{\bullet} (M;A)=\IM(\delta|_{C^{\bullet} (M;A)});\\
H^{\bullet} (M;A)=Z^{\bullet} (M;A)/B^{\bullet} (M;A);\\
b\in[a]\in Bor^{\bullet}(M;A)\Rightarrow a-b=\delta c;\hskip 3pt
                 c\in C^{\bullet}(M;A);\\
b\in[a]\in Cyc^{\bullet}(M;A)\Rightarrow \delta(a-b)=0;\\
b\in[a]\in {}^A{\underline{\Omega}}^{\bullet}_s(M)\Rightarrow
                 \left\{\begin{array}{l}
                          \delta a=\delta b=0\\
                          a-b=\delta c,\quad c\in C^{\bullet}(M;A)\\
                          \end{array}
                                \right\}.\\
\end{array}\right.$$
Furthermore, one has the following canonical isomorphism:
${}^A{\underline{\Omega}}^{\bullet}_s(M)\cong H^{\bullet}(M;A)$. As
$C^{\bullet}(M;A)$ is a free two-sided projective $A$-module, one
has the unnatural isomorphism: $Bor_{\bullet}(M;A)\cong
{}^A{\underline{\Omega}}^{\bullet}_s(M)\bigoplus
Cyc^{\bullet}(M;A)$.

\end{theorem}

\begin{proof}
It follows from standard results in cohomological algebra and
cohomology in topological spaces. (See also \cite{PRA1}.)
\end{proof}

\begin{definition}
We say that a manifold $M$ is {\em cohomologically trivial} if all the
cohomology groups $H^r(M;A)$ vanish for $r\ge 1$.
\end{definition}

\begin{theorem}
Let $M$ be a $n$-dimensional manifold. The following
propositions are equivalent.

{\em(i)} $M$ is cohomologically trivial.

{\em(ii)} $H^r(M;\mathbb{K})=0$, $\forall r\ge 1$.

{\em(iii)} $H_r(M;A)\cong H_r(M;\mathbb{K})=0$, $\forall r\ge 1$.

{\em(iv)} The complex $\{C^\bullet(M;A),\delta\}$ is acyclic,
i.e., the sequence

\begin{equation}\label{complex}
\xymatrix{0\ar[r]&Z^0\ar[r]^(.4){\epsilon}&C^0(M;A)\ar[r]^{\delta}&C^1(M;A)\ar[r]^(.6){\delta}&\cdots\ar[r]^(.4){\delta}&
C^n(M;A)\ar[r]^(.6){\delta}&0}
\end{equation}
is exact.
\end{theorem}

\begin{proof}
(i)$\Leftrightarrow$(ii) since $H^r(M;A)\cong
H^r(M;\mathbb{K})\bigotimes_{\mathbb{K}}A$.

(i)$\Leftrightarrow$(iii) since $H^r(M;A)\cong Hom_A(H_r(M;A);A)$
and considering the following isomorphism $H_r(M;A)\cong
H_r(M;\mathbb{K})\bigotimes_{\mathbb{K}}A$.

(i)$\Leftrightarrow$(iv) since the exactness of the sequence
(\ref{complex}) is equivalent to $H^r(M;A)=0$, for $r\ge 1$.
\end{proof}

\begin{theorem}
Let $M$ be a $n$-dimensional manifold modeled on
the algebra $A$. The following propositions are
equivalent.

{\em(i)} $M$ is cohomologically trivial.

{\em(ii)} $H_r(M;A)=0$, $r\ge 1$.

{\em(iii)} $H_r(M;\mathbb{K})=H^r(M;\mathbb{K})=0$, $r\ge 1$.

\end{theorem}

\begin{example}
A manifold contractible to a point is cohomologically
trivial.
\end{example}

\begin{theorem}[h-Cobordism groups]
{\em 1)} If  $(W; X_0,X_1)$, $\partial W=X_0\sqcup X_1$, is a h-cobordant and $X_i$, $i=0,1$, are simply connected, then $W$ is a trivial h-cobordism. For $n=4$ the h-cobordism theorem is true topologically.

{\em 2)} A simply connected manifold $M$ is h-cobordant to the sphere $S^n$ iff $M$ bounds a contractible manifold.

{\em 3)} If $M$ is a homotopy sphere, then $M\sharp(-M)$ bounds a contractible manifold.

{\em 4)} If a homotopy sphere of dimension $2k$ bounds a stably-parallelizable manifold $M$ then it bounds a contractible manifold $M_1$. (See Definition \ref{parallelizable-manifold}.)

{\em 5)} Let $\Theta_n$ denote the collection of all h-cobordism classes of homotopy $n$-spheres. $\Theta_n$ is an additive group with respect the connected sum,\footnote{The {\em connected sum} of two connected $n$-dimensional manifolds $X$ and $Y$ is the $n$-dimensional manifold $X\sharp Y$ obtained by excising the interior of embedded discs $D^n\subset X$, $D^n\subset Y$, and joining the boundary components $S^{n-1}\subset\overline{X\setminus D^n}$, $S^{n-1}\subset\overline{Y\setminus D^n}$, by $S^{n-1}\times I$.} where the sphere $S^n$ serves as zero element. The opposite of an element $X$ is the same manifold with reversed orientation, denoted by $-X$. In Tab. \ref{calculated-h-cobordism-groups-homotopy-sphere} are reported the expressions of some calculated groups $\Theta_n$.\footnote{It is interesting to add that another related notion of cobordism is the {\em H-cobordism} of $n$-dimensional manifold, $(V;M,N)$, $\partial V=M\sqcup N$, with $H_\bullet(M)\cong H_\bullet(N)\cong H_\bullet(V)$. An $n$-dimensional manifold $\Sigma$ is a {\em homology sphere} if $H_\bullet(\Sigma)=H_\bullet(S^n)$. Let $\Theta_n^H$ be the abelian group of $H$-cobordism classes of $n$-dimensional homology spheres, with addition by connected sum. (Kervaire's theorem.) For $n\ge 4$ every $n$-dimensional homology sphere $\Sigma$ is H-cobordant to a homology sphere and the forgethful map $\Theta_n\to \Theta_n^H$ is an isomorphism.}

\end{theorem}

\begin{proof}
There are topological manifolds that have not smooth structure. Furthermore, there are examples of topological manifolds that have smooth structure everywhere except a single point. If a neighborhood of that point is removed, the smooth boundary is a homotopy sphere. Any smooth manifold may be {\em triangulated}, i.e. admits a PL structure, and the underlying PL manifold is unique up to a PL isomorphism.\footnote{A topological manifold $M$ is piecewise linear, i.e., admits a PL structure, if there exists an atlas $\{U_\alpha,\varphi_\alpha\}$ such that the composities $\varphi_\alpha\circ\varphi^{-1}_{\alpha'}$, are piecewise linear. Then there is a polyehdron $P\subset \mathbb{R}^s$, for some $s$ and a homeomorphism $\phi:P\thickapprox M$, ({\em triangulation}), such that each composite $\varphi_\alpha\circ\phi$ is piecewise linear.} The vice versa is false. No all topological or PL manifolds have at least one smooth structure.

The h-cobordism classes $\Theta_n$ of $n$-dimensional homotopy spheres are trivial for $1\le n\le 6$, i.e. $\Theta_n\cong0$, and $[S^n]=0$. For $n=1,2$ this follows from the fact that each of such topological manifolds have a unique smooth structure uniquely determined by its homology. For $n=3$ this follows from the proof of the Poincar\'e conjecture (see \cite{PRA14}). Furthermore, each topological $3$-manifold has an unique differential structure \cite{ MOISE1, MUNKRES, WHITEHEAD}. Therefore, since from the proof of the Poincar\'e conjecture it follows that all $3$-homotopy spheres are homeomorphic to $S^3$, it necessarily follows that all $3$-homotopy spheres are diffeomorphic to $S^3$ too. Furthermore, for $n=4$, the triviality of $\Theta_4$ follows from the works by Freedman \cite{FREEDMAN}. (See also J. Cerf \cite{CERF}.)
\begin{lemma}[Freedman's theorem]\cite{FREEDMAN}
Two closed simply connected $4$-manifolds are homeomorphic iff they have the same symmetric bilinear form $\sigma:H^2(M;\mathbb{Z})\otimes H^2(M;\mathbb{Z})\to H^4(M;\mathbb{Z})\cong\mathbb{Z}$, (with determinant $\pm 1$, induced by the cup product), and the same {\em Kirby-Siebermann invariant} $\kappa$.\footnote{$\kappa$ is $\mathbb{Z}_2$-valued and vanishes iff the product manifold $M\times\mathbb{R}$ can be given a differentiable structure.} Any $\sigma$ can be realized by such a manifold. If $\sigma(x\otimes x)$ is odd for some $x\in H^2(M;\mathbb{Z})$, then either value of $\kappa$ can be realized also. However, if $\sigma(x\otimes x)$ is always even, then $\kappa$ is determined by $\sigma$, being congruent to $\frac{1}{8}\sigma$.

In particular, if $M$ is homotopy sphere, then $H^2(M,\mathbb{Z})=0$ and $\kappa\equiv 0$, so $M$ is homeomorphic to $S^4$.\footnote{It is not known which $4$-manifolds with $\kappa=0$ actually possess differentiable structure, and it is not known when this structure is essentially unique.}
\end{lemma}

The cases $n=5,6$, can be proved by using surgery theory and depend by the Smale's h-cobordism theorem.\footnote{There exists also a s-cobordism version of such a theorem for non-simply connected manifolds. More precisely, an $(n+1)$-dimensional h-cobordism $(V;N,M)$ with $n\ge 5$, is trivial iff it is an $s$-cobordism. This means that for $n\ge 5$ h-cobordant $n$-dimensional manifolds are diffeomorphic iff they are s-cobordant. Since the Whitehead group of the trivial group is trivial, i.e., $Wh(\{1\})=0$, it follows that h-cobordism theorem is the simply-connected special case of the s-cobordism.}

\begin{lemma}[Smale's h-cobordism theorem]\cite{SMALE1}
Any $n$-dimensional simply connected h-cobordism $W$, $n>5$, with $\partial W=M\sqcup(-N)$, is diffeomorphic to $M\times[1,0]$. (All manifolds are considered smooth and oriented. $-N$ denotes the manifold $N$ with reversed orientation.)\footnote{The proof utilizes Morse theory and the fact that for an h-cobordism $H_\bullet(W,M)\cong H_\bullet(W,N)\cong 0$, gives $W\cong M\times[0,1]$. The motivation to work with dimensions $n\ge 5$ is in the fact that it is used the {\em Whitney embedding theorem} that states that a map $f:N\to M$, between manifolds of dimension $n$ and $m$ respectively, such that either $2n+1\le m$ or $m=2n\ge 6$ and $\pi_1(M)=\{1\}$, is homotopic to an embedding.}

If $n\ge 5$ any two homotopy $n$-sphere are PL homeomorphic, and diffeomorphic too except perhaps at a single point. (If $n=5$, (resp. $n=6$), then any homotopy $n$-sphere $\Sigma$ bounds a contractible $6$-manifold, (resp. $7$-manifold), and is diffeomorphic to $S^5$, (resp. $S^6$). Every smooth manifold $M$ of dimension $n>4$, having the homotopy of a sphere is a {\em twisted sphere}, i.e., $M$ can be obtained by taking two disks $D^n$ and gluing them together by a diffeomorphism $f:S^{n-1}\cong S^{n-1}$ of their boundaries. More precisely one has the isomorphism $\pi_0(Diff_+(S^n))\cong\Theta_{n+1}$, $[f]\mapsto \Sigma_f\equiv D^{n+1}\bigcup_f(-D^{n+1})$, where $\pi_0(Diff_+(S^n))$ denotes the group of isotopy classes of oriented preserving diffeomorphisms of $S^n$.
\end{lemma}
See the paper by M. A. Kervaire and J. W. Milnor \cite{KERVAIRE-MILNOR} and the following ones by S. Smale \cite{SMALE1, SMALE2}. In the following we shall give a short summary of this proof for $n\ge 5$. This is really an application of the Browder-Novikov theorem.

\begin{lemma}\label{gamma-group-a}
Let $\Xi_n$ denote the set of smooth $n$-dimensional manifolds homeomorphic to $S^n$. Let $\sim_d$ denote the equivalence relation in $\Xi_n$ induced by diffeomorphic manifolds. Put $\Gamma_n\equiv \Xi_n/\sim_d$.\footnote{$\Gamma_n$ can be identified with the set of {\em twisted $n$-spheres} up to orientation-preserving diffeomorphisms, for $n\not= 4$. One has the exact sequences given in (\ref{exact-sequences-twisted-spheres}).
\begin{equation}\label{exact-sequences-twisted-spheres}
    \xymatrix{\pi_0(Diff^+(D^n))\ar[r]& \pi_0(Diff^+(S^{n-1}))\ar[r]&\Gamma_n\ar[r]&0}
\end{equation}
If the used diffeomorphism $S^{n-1}\to S^{n-1}$ to obtain a twisted $n$-sphere by gluing the corresponding boundaries of two disks $D^n$, is not smoothly isotopic to the identity, one obtains an exotic $n$-sphere. For $n>4$ every exotic $n$-sphere is a twisted sphere. For $n=4$, instead, twisted spheres are standard ones \cite{CERF}.} Then the operation of connected sum makes $ \Gamma_n$ an abelian group for $n\ge 1$.
\end{lemma}
\begin{proof}
Let us first remark that since we are working in $\Xi_n$, the operation of connected sum there must be considered in smooth sense. Then it is easy to see that the for $M_1, M_2,M_3\in \Xi_n$, one has $M_1\sharp M_2\cong M_2\sharp M_1$ and that $(M_1\sharp M_2)\sharp M_3\cong M_1\sharp( M_2\sharp M_3)$. Therefore, it is well defined the commutative and associative composition map $+: \Gamma_n\times \Gamma_n\to\Gamma_n$, $[M_1]+[M_2]=[M_1\sharp M_2]$. The zero of this composition is the equivalence class $[S^n]\in \Gamma_n$. In fact, since $\overline{S^n\setminus D^n}\bigcup_{S^{n-1}}(S^{n-1}\times I)\cong D^n$, we get $M\sharp S^n\cong \overline{M\setminus D^n}\bigcup_{S^{n-1}}(\overline{S^n\setminus D^n}\bigcup_{S^{n-1}}(S^{n-1}\times I))\cong\overline{M\setminus D^n}\bigcup_{S^{n-1}}D^n\cong M$. Therefore, $[S^n]=0\in\Gamma_n$. Furthermore, each element $M\in\Xi_n$ admits, up to diffeomorphisms, an unique opposite $M'\in\Xi_n$. In fact, since $\overline{M\setminus D^n}\cong D^n$, it follows that $M\cong D^n\bigcup_{\lambda}D^n\cong D^n\bigcup_{S^{n-1}}D^n$, where $\lambda:S^{n-1}\to S^{n-1}$ is a given diffeomorphism that identifies the two copies of $S^{n-1}$. Then $M'$ is defined by $M'=D^n\bigcup_{\lambda^{-1}}D^n$. In fact one has $M\sharp M'\cong (D^n\bigcup_{\lambda}D^n)\sharp(D^n\bigcup_{\lambda^{-1}}D^n)\cong D^n\bigcup_{1}D^n\cong S^n$, where $1=\lambda\circ\lambda^{-1}:S^{n-1}\to S^{n-1}$. (See Fig. \ref{connected-sum-properties}.)
\end{proof}
\begin{lemma}\label{gamma-group-b}
One has the group isomorphisms $\Gamma_n\cong\Theta_n$ for any $n\ge 1$ and $n\not= 4$. So these groups classify all possible differentiable structures on $S^n$, up to orientation preserving diffeomorphisms, in all dimension $n\not=4$.
\end{lemma}
\end{proof}
\begin{remark}[Strange phenomena on dimension four]
It is well known that on $\mathbb{R}^4$ there are uncountably many inequivalent differentiable structures, i.e., one has {\em exotic $\mathbb{R}^4$}, say $\widetilde{\mathbb{R}^4}$. (This is a result by M. H. Freedman \cite{FREEDMAN}, starting from some results by S. K. Donaldson \cite{DONALDSON}.) On the other hand by the fact that $\Gamma_4\cong\Theta_4=0$ it follows for any $4$-dimensional homotopy sphere $\Sigma\cong S^4$. So taking in to account that $S^4\setminus\{pt\}\cong \mathbb{R}^4$, it natural arises the question: {\em Do exotic $4$-sphere $\widetilde{\Sigma}$ exist such that $\widetilde{\Sigma}\setminus\{pt\}\cong \widetilde{\mathbb{R}^4}$\hskip 2pt?}
The answer to this question, conjecturing the isomorphism $\Gamma_4\cong\Theta_4=0$, should be in the negative. This means that all exotic  $\widetilde{\mathbb{R}^4}$ collapse on the unique one $S^4$ by the process of one point compactification ! (However this is generally considered an open problem in geometric topology and called the {\em smooth Poincar\'e conjecture}. (See, e.g., \cite{FREEDMAN-GOMPF-MORRISON-WALKER}.)
\end{remark}

\begin{lemma}[M. Hirsch and J. Munkres]\label{hirsch-munkres-lemma}\cite{HIRSCH1, MUNKRES}
The obstructions to the existence of a smooth structure on a $n$-dimensional combinatorial (or PL) manifold lie in the groups $H^{k+1}(M;\Gamma_k)$; while the obstruction to the uniqueness of such a smooth structures, when it exists, are elements of $H^k(M;\Gamma_k)$.
\end{lemma}

\begin{lemma}[Kirby and Siebenman]\label{kirby-siebenman-lemma}\cite{KIRBY-SIEBENMAN}
For a topological manifold $M$ of dimension $n\ge 5$, there is only one obstruction to existence of a PL-structure, living in $H^4(M;\mathbb{Z}_2)$, and only one obstruction to the uniqueness of this structure (when it exists), living in $H^3(M;\mathbb{Z}_2)$.\footnote{This result by Kirby and Siebenman does not exclude that every manifold of dimension $n>4$ can possess some triangulation, even if it cannot be PL-homeomorphic to Euclidean space.}
\end{lemma}

\begin{figure}[h]
\centering
\centerline{\includegraphics[height=4cm]{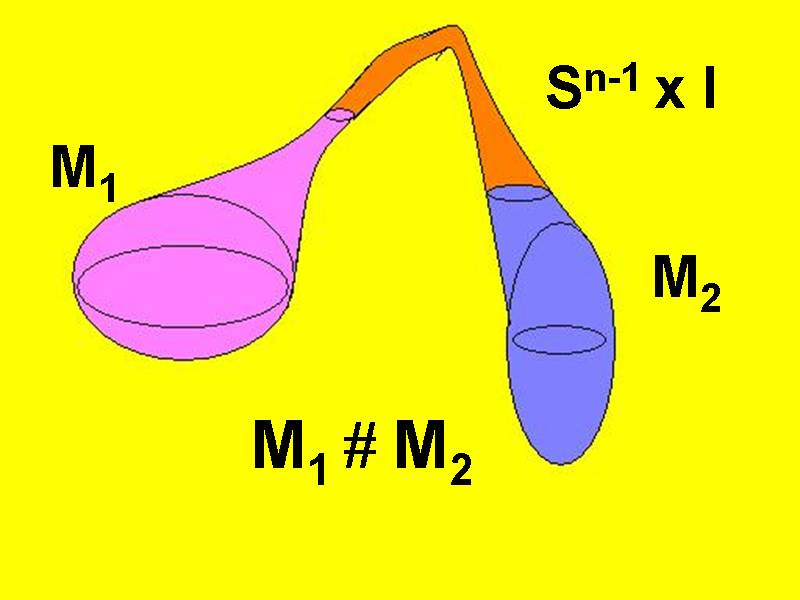} \includegraphics[height=4cm]{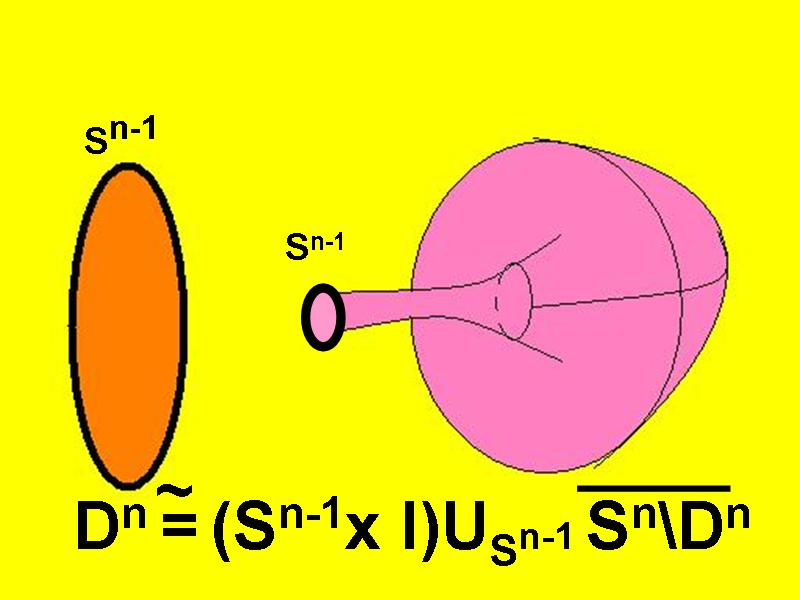}}
\caption{Connected sum properties.}
\label{connected-sum-properties}
\end{figure}

\begin{definition}[Intersection form and signature of manifold]
{\em 1)} The {\em intersection form} of a $2n$-dimensional topological manifold with boundary $(M,\partial M)$ is the $(-1)^n$-symmetric form $\lambda$ over the $\mathbb{Z}$-module $H\equiv H^n(M,\partial)/torsion$, $\lambda(x,y)=<x\cup y,[M]>\in \mathbb{Z}$.\footnote{Let $R$ be a commutative ring and $H$ a finite generated free $R$-module. A {\em $\epsilon$-symmetric form over $H$} is a bilinear mapping $\lambda:H\times H\to R$, such that $\lambda(x,y)=\epsilon\lambda(y,x)$, with $\epsilon\in\{+1,-1\}$. The form $\lambda$ is {\em nonsingular} if the $R$-module morphism $H\to H^*\equiv Hom_{R}(H;R)$, $x\mapsto(y\mapsto\lambda(x,y))$ is an isomorphism. A {\em$\epsilon$-quadratic form} associated to a $\epsilon$-symmetric form $\lambda$ over $H$, is a function
$\mu:H\to Q_\epsilon(R)\equiv\COKER(1-\epsilon:R\to R)$, such that: (i) $\lambda(x,y)=\mu(x+y)-\mu(x)-\mu(y)$; (ii)  $\lambda(x,x)=(1+\epsilon)\mu(x)\in \IM(1+\epsilon:R\to R)\subseteq\ker(1-\epsilon:R\to R)$, $\forall x,y\in H$, $a\in R$.
If $R=\mathbb{Z}$ and $\epsilon=1$, we say {\em signature of $\lambda$}, $\sigma(\lambda)=p-q\in\mathbb{Z}$, where $p$ and $q$ are respectively the number of positive and negative eigenvalues of the extended form on $\mathbb{R}\bigotimes_{\mathbb{Z}}H$. Then $\lambda$ has a $1$-quadratic function $\mu:H\to Q_{+1}(\mathbb{Z})$ iff $\lambda$ has even diagonal entries, i.e., $\lambda(x,x)\equiv 0\: (\MOD\, 2)$, with $\mu(x)=\lambda(x,x)/2$,
$\forall x\in H$. If $\lambda$ is nonsingular then $\sigma(\lambda)\equiv 0\: (\MOD\, 8)$. Examples. 1) $R=H=\mathbb{Z}$, $\lambda=1$, $\sigma(\lambda)=1$.

2) $R=\mathbb{Z}$, $H=\mathbb{Z}^8$, $\lambda=E^8$-form,
given by
$\left(\lambda_{ij}\right)=\scalebox{0.6}{$\left(\begin{array}{cc}
a_{rs}&b_{rs}\\
 c_{rs}&d_{rs}\\
 \end{array}\right)$}$ with
     $(a_{rs})=\scalebox{0.4}{$\left(\begin{array}{cccc}
  2 & 1 & 0 & 0 \\
  1 & 2 & 1 & 0 \\
  0 & 1 & 2 & 1 \\
  0 & 0 & 1 & 2 \\
  \end{array}\right)$}$,     $(b_{rs})=\scalebox{0.4}{$\left(\begin{array}{cccc}
  0 & 0 & 0 & 0 \\
  0 & 0 & 0 & 0 \\
  0 & 0 & 0 & 0 \\
  1 & 0 & 0 & 0 \\
  \end{array}\right)$}$, $(c_{rs})=\scalebox{0.4}{$\left(\begin{array}{cccc}
  0 & 0 & 0 & 1 \\
  0 & 0 & 0 & 0 \\
  0 & 0 & 0 & 0 \\
  0 & 0 & 0 & 0 \\
  \end{array}\right)$}$, $(d_{rs})=\scalebox{0.4}{$\left(\begin{array}{cccc}
  2 & 1 & 0 & 1 \\
  1 & 2 & 1 & 0 \\
  0 & 1 & 2 & 1 \\
  1 & 0 & 1 & 2 \\
  \end{array}\right)$}$.}

{\em 2)(Milnor's plumbing theorem.)} For $n\ge 3$ every $(-1)^n$-quadratic form $(H,\lambda,\mu)$ over $\mathbb{Z}$ is realized by an $(n-1)$-connected $2n$-dimensional framed manifold with boundary $(V,\partial V)$ with $H_n(V)=H$. The form $(H,\lambda,\mu)$ is nonsingular iff $H_\bullet(\partial V)=H_\bullet(S^{2n-1})$.
Let $(H,\lambda)$ be a nonsingular $(-1)$-symmetric form over $\mathbb{Z}$ and $\{b_j,c_j\}_{1\le j\le p}$ a basis for the $\mathbb{Z}$-module $H$, such that $\lambda(b_r,b_s)=0$, $\lambda(c_r,c_s)=0$, $\lambda(b_r,c_s)=0$, for $r\not=s$, and $\lambda(b_r,c_r)=1$. Let $\mu:H\to Q_{-1}(\mathbb{Z})=\mathbb{Z}_2$ be a $(-1)$-quadratic function associated to $(H,\lambda)$. Then the {\em Arf invariant} of a nonsingular $(-1)$-quadratic form $(H,\lambda,\mu)$ over $\mathbb{Z}$ is $Arf(H,\lambda,\mu)=\sum_{1\le j\le p}\mu(b_j)\mu(c_j)\in\mathbb{Z}_2\equiv\{0,1\}$.

If $\partial M=\varnothing$, or $H_\bullet(\partial M)=H_\bullet(S^{2n-1})$, then $\lambda$ is nonsingular.

{\em 3)} The {\em signature} $\sigma(M)$ of a $4k$-dimensional manifold $(M,\partial M)$ is $\sigma(M)=\sigma(\lambda)\in\mathbb{Z}$, where $\lambda$ is the symmetric form over the $\mathbb{Z}$-module $H_{2k}(M)/torsion$.

{\em 4)} Let $M$ be an oriented manifold with empty boundary, $\partial M=\varnothing$, $\dim M=n=4k$. Then $H^i(M;\mathbb{Z})$ is finitely generated for each $i$ and $H^{2k}(M;\mathbb{Z})\cong\mathbb{Z}^s\bigoplus Tor$, where $Tor$ is the torsion subgroup. Let $[M]\in H_{4k}(M;\mathbb{Z})$ be the orientation class of $M$. Let $<,>:H^{2k}(M;\mathbb{Z})\times H^{2k}(M;\mathbb{Z})\to\mathbb{Z}$, be the symmetric bilinear form given by $(a,b)\mapsto<a,b>=<a\cup b,[M]>\in\mathbb{Z}$. This form vanishes on the torsion subgroup, hence it factors on $H^{2k}(M;\mathbb{Z})/Tor\times H^{2k}(M;\mathbb{Z})/Tor\cong \mathbb{Z}^s\times \mathbb{Z}^s\to\mathbb{Z}$. This means that the adjoint map $\phi: H^{2k}(M;\mathbb{Z})/Tor\to Hom_{\mathbb{Z}}(H^{2k}(M;\mathbb{Z})/Tor;\mathbb{Z})=(H^{2k}(M;\mathbb{Z})/Tor))^*$, $a\mapsto \phi(a)(b)=<a,b>$, is an isomorphism. This is a just a direct consequence of Poincar\'e duality: $H_{2k}(M;\mathbb{Z})/Tor=(H^{2k}(M;\mathbb{Z})/Tor)^*$ and $a(b\cap[M])=<a\cup b,[M]>$. Then the signature of this bilinear form is the usual signature of the form after tensoring with the rationals $\mathbb{Q}$, i.e., of the symmetrix matrix associated to the form, after choosing a basis for $\mathbb{Q}^s$. Hence the {\em signature} is the difference between the number of $+1$ eigenvalues with the number of $-1$ eigenvalues of such a matrix. Let us denote by $\sigma(M)$ the signature of the above nondegenerate symmetric bilinear form, and call it {\em signature} of $M$.
\end{definition}
\begin{table}[t]
\caption{Polynomials $s_I(\sigma_1,\cdots,\sigma_n)$, for $0\le n\le 4$.}
\label{polynomials-si}
\scalebox{0.9}{$\begin{tabular}{|l|l|}
  \hline
  \hfil{\rm{\footnotesize $n$}}\hfil& \hfil{\rm{\footnotesize $s_I(\sigma_1,\cdots,\sigma_n)$}}\hfil\\
 \hline
 {\rm{\footnotesize $0$}}&{\rm{\footnotesize $s=1$}}\\
 \hline
 {\rm{\footnotesize $1$}}& {\rm{\footnotesize $s_{(1)}(\sigma_1)=\sigma_1$}}\\
  \hline
{\rm{\footnotesize $2$}}&{\rm{\footnotesize $s_{(2)}(\sigma_1,\sigma_2)=\sigma_1^2-2\sigma_2$}}\\
&{\rm{\footnotesize $s_{(1,1)}(\sigma_1,\sigma_2)=\sigma_2$}}\\
\hline
{\rm{\footnotesize $3$}}&{\rm{\footnotesize $s_{(3)}(\sigma_1,\sigma_2,\sigma_3)=\sigma_1^3-3\sigma_1\sigma_2+3\sigma_3$}}\\
 &{\rm{\footnotesize $s_{(1,2)}(\sigma_1,\sigma_2,\sigma_3)=\sigma_1\sigma_2-3\sigma_3$}}\\
&{\rm{\footnotesize $s_{(1,1,1)}(\sigma_1,\sigma_2,\sigma_3)=\sigma_3$}}\\
 \hline
{\rm{\footnotesize $4$}}&{\rm{\footnotesize $s_{(4)}(\sigma_1,\sigma_2,\sigma_3,\sigma_4)=\sigma_1^4-4\sigma_1^2\sigma_2+2\sigma_2^2+4\sigma_1\sigma_3-4\sigma_4$}}\\
&{\rm{\footnotesize $s_{(1,3)}(\sigma_1,\sigma_2,\sigma_3,\sigma_4)=\sigma_1^2\sigma_2-2\sigma_2^2-\sigma_1\sigma_3+4\sigma_4$}}\\
&{\rm{\footnotesize $s_{(2,2)}(\sigma_1,\sigma_2,\sigma_3,\sigma_4)=\sigma_2^2-2\sigma_1\sigma_3+2\sigma_4$}} \\
 &{\rm{\footnotesize $s_{(1,1,2)}(\sigma_1,\sigma_2,\sigma_3,\sigma_4)=\sigma_1\sigma_3-4\sigma_4$}}\\
  &{\rm{\footnotesize $s_{(1,1,1,1)}(\sigma_1,\sigma_2,\sigma_3,\sigma_4)=\sigma_4$}} \\
\hline
\end{tabular}$}
\end{table}

\begin{theorem}[R. Thom's properties of signature]
{\em 1)} If $M$ has $\dim M=4k$, and it is a boundary, then $\sigma(M)=0$.

{\em 2)} $\sigma(-M)=-\sigma(M)$.

{\em 3)} Let $M$ and $L$ be two $4k$-dimensional closed, compact, oriented manifolds without boundary. Then we have
$\sigma(M\sqcup L)=\sigma(M)+\sigma(L)$ and $\sigma(M\times L)=\sigma(M).\sigma(L)$, where the orientation on $M\times L$ is $[M\times L]=[M]\otimes [L]$.

{\em 4) (Rohlin's signature theorem).} The signature of a closed oriented $4k$-dimensional manifold is an oriented cobordism invariant, i.e., if $\partial W=M\sqcup N$, it follows that $\sigma(M)=\sigma(N)\in\mathbb{Z}$. More precisely, the signature for oriented boundary $4k$-dimensional manifolds is zero and it defines a linear form $\sigma:{}^+\Omega_{4k}\to\mathbb{Z}$.

Furthermore, let $M$ and $N$ be $4k$-dimensional manifolds with differentiable boundaries: $\partial M=\bigcup_jX_j$, $\partial N=\bigcup_iY_i$, such that $X_1=Y_1$. Then one has the formula {\em(\ref{additivity-signature-manifolds-with-differentiable-boundaries})}.
\begin{equation}\label{additivity-signature-manifolds-with-differentiable-boundaries}
\sigma(M\bigcup_{X_1=Y_1}N)=\sigma(M)+\sigma(N).
\end{equation}

{\em 5) (Hirzebruch's signature theorem (1952)).}
The signature of a closed oriented $4k$-dimensional manifold $M$ is given by
\begin{equation}\label{hirzebruch-signature-formula}
    \sigma(M)=<\mathcal{L}_k(p_1,\cdots,p_k),[M]>\in\mathbb{Z}
\end{equation}
with $\mathcal{L}_k(p_1,\cdots,p_k)$ polynomial in the Pontrjagin classes $p_j$ of $M$, i.e., $p_j(M)\equiv p_j(TM)\in H^{4j}(M)$, representing the $\mathcal{L}$ genus, i.e., the genus of the formal power series given in {\em(\ref{formal-power-series-for-l-genus})}.\footnote{A {\em genus} for closed smooth manifolds with some $X$-structure, is a ring homomorphism $\Omega_\bullet^X\to R$, where $R$ is a ring. For example, if the $X$-structure is that of oriented manifolds, i.e., $X=SO$, then the signature of these manifolds just identifies a genus $\sigma:{}^+\Omega_\bullet=\Omega_\bullet^{SO}\to \mathbb{Z}$, such that $\sigma(1)=1$, and $\sigma:{}^+\Omega_p\to 0$ if $p\not=4q$. Therefore the genus identifies also a $\mathbb{Q}$-algebra homomorphism $\Omega_\bullet^{SO}\bigotimes_{\mathbb{Q}}\mathbb{Q}\to\mathbb{Q}$. More precisely, let $A\equiv \mathbb{Q}[t_1,t_2,\cdots]$ be a graded commutative algebra, where $t_i$ has degree $i$. Set $\mathcal{A}\equiv A[[a_0,a_1,\cdots]]$, where $a_i\in A$ is homogeneous of degree $i$, i.e., the elements of $\mathcal{A}$ are infinite formal sums $a\equiv a_0+a_1+a_2+\cdots$. Let $\mathcal{A}^\bullet\subset\mathcal{A}$ denote the subgroup of the multiplicative group of $\mathcal{A}$ of elements with leading term $1$. Let $K_1(t_1)$, $K_2(t_1,t_2)$, $K_3(t_1,t_2,t_3),\cdots\in A$, be a sequence of polynomials of $A$, where $K_n$ is homogeneous of degree $n$. For $a= a_0+a_1+a_2+\cdots\in \mathcal{A}^\bullet$, we define $K(a)\in\mathcal{A}^\bullet$ by $K(a)=1+K_1(a_1)+K_2(a_1,a_2)+\cdots$. We say that $K_n$ form a {\em multiplicative sequence} if $K(ab)=K(a)K(b)$, $\forall a,b\in\mathcal{A}^\bullet$. An example is with $K_n(t_1,\cdots,t_n)=\lambda^n t_n$, $\lambda\in\mathbb{Q}$. Another example is given by the formal power series given in (\ref{formal-power-series-for-l-genus}) with $\lambda_k=(-1)^{k-1}\frac{2^{2k}B_{2k}}{(2k)!}$. For any partition $I=(i_1,i_2,\cdots,i_k)$ of $n$, set $\lambda_I=\lambda_{i_1}\lambda_{i_2}\cdots\lambda_{i_k}$. Now define polynomials $\mathcal{L}_n(t_1,\cdots,t_n)\in A$ by $\mathcal{L}_n(t_1,\cdots,t_n)=\sum_I\lambda_Is_I(t_1,\cdots,t_n)$, where the sum is over all partitions of $n$ and $s_I$ is the unique polynomial belonging to $\mathbb{Z}[t_1,\cdots,t_n]$ such that $s_I(\sigma_1,\cdots,\sigma_n)=\sum t^I$, where $\sigma_1,\cdots,\sigma_n$ are the elementary symmetric functions that form a polynomial basis for the ring $\mathcal{S}_n$ of symmetric functions in $n$ variables. ($\mathcal{S}_n$ is the graded subring of $\mathbb{Z}[t_1,\cdots,t_n]$ of polynomials that are fixed by every permutation of the variables. Therefore we can write $\mathcal{S}_n=\mathbb{Z}[\sigma_1,\cdots,\sigma_n]$, with $\sigma_i$ of degree $i$. In Tab. \ref{polynomials-si} are reported the polynomials $s_I(\sigma_1,\cdots,\sigma_n)$, for $0\le n\le 4$.) $\mathcal{L}_n$ form a multiplicative sequence. In fact $\mathcal{L}(ab)=\sum_Is_I(ab)=\sum_I\lambda_I\sum_{I_1I_2=I}s_{I_1}(a)s_{I_2}(b)
=\sum_{I_1I_2=I}\lambda_{I_1}s_{I_1}(a)\lambda_{I_2}s_{I_2}(b)=\mathcal{L}(a)\mathcal{L}(b)$. Then for an $n$-dimensional manifold $M$ one defines {\em $\mathcal{L}$-genus}, $\mathcal{L}[M]=0$ if $n\not=4k$, and $\mathcal{L}[M]=<K_k(p_1(TM),\cdots,p_k(TM)),\mu_M>$ if $n=4k$, where $\mu_M$ is the rational fundamental class of $M$ and $K_k(p_1(TM),\cdots,p_k(TM))\in H^n(M;\mathbb{Z})$.}
\begin{equation}\label{formal-power-series-for-l-genus}
    \frac{\sqrt{z}}{\tanh(\sqrt{z})}=\sum_{k\ge0}\frac{2^{2k}B_{2k}z^k}{(2k)!}=1+\frac{z}{3}-\frac{z^2}{45}+\cdots
\end{equation}
where the numbers $B_{2k}$ are the {\em Bernoulli numbers}.\footnote{Formula (\ref{hirzebruch-signature-formula}) is a direct consequence of Thom's computation of ${}^+\Omega_\bullet\bigotimes_{\mathbb{Z}}\mathbb{Q}\cong\mathbb{Q}[y_{4k}|k\ge 1]$, with $y_{4k}=[\mathbb{C}P^{2k}]$. (${}^+\Omega_j\bigotimes_{\mathbb{Z}}\mathbb{Q}=0$ for $j\not=4k$.) In Fact, one has $p_j(\mathbb{C}P^{n})=p_j(T\mathbb{C}P^{n})=\binom{n+1}j \in H^{4j}(\mathbb{C}P^n)=\mathbb{Z}$, $0\le j\le \frac{n}{2}$.
For $n=2k$ the evaluation $<\mathcal{L}_k,[\mathbb{C}P^{2k}]>=1\in\mathbb{Z}$ coincides with the signature of $\mathbb{C}P^{2k}$: $\sigma(\mathbb{C}P^{2k})=\sigma(H^{2k}(\mathbb{C}P^{2k}),\lambda)=\sigma(\mathbb{Z},1)=1\in\mathbb{Z}$. Therefore, the signature identifies a $\mathbb{Q}$-algebra homomorphism $\Omega_\bullet^{SO}\bigotimes_{\mathbb{Z}}\mathbb{Q}\to\mathbb{Q}$. So the Hirzebruch signature theorem states that this last homomorhism induced by the signature, coincides with the one induced by the genus. In Tab. \ref{examples-hirzebruch-polynomials} are reported some Hirzebruch's polynomials for $\mathbb{C}P^{2k}$. In Tab. \ref{bernoulli-numbers} are reported also the Bernoulli numbers $B_n$, with the Kronecker's formula, and explicitly calculated  for $0\le n\le 18$.}

\begin{table}[t]
\caption{$\mathcal{L}_k$-polynomials for $\mathbb{C}P^{2k}$, with $k=1,2,3,4$ and related objects.}
\label{examples-hirzebruch-polynomials}
\scalebox{0.6}{$\begin{tabular}{|l|l|l|l|}
  \hline
  \hfil{\rm{\footnotesize $k$}}\hfil& \hfil{\rm{\footnotesize $\mathcal{L}_k$}}\hfil&\hfil{\rm{\footnotesize $p_j(\mathbb{C}P^{2k})$}}\hfil& \hfil{\rm{\footnotesize $\sigma(\mathbb{C}P^{2k})$}}\hfil\\
 \hline
 {\rm{\footnotesize $1$}}&{\rm{\footnotesize $\mathcal{L}_1(p_1)=\frac{1}{3}$}}&{\rm{\footnotesize $p_1(\mathbb{C}P^{2})=3$}}&{\rm{\footnotesize $\sigma(\mathbb{C}P^{2})=\frac{1}{3}\times 3=1$}}\\
 \hline
 {\rm{\footnotesize $2$}}& {\rm{\footnotesize $\mathcal{L}_2(p_1,p_2)=\frac{1}{45}(7p_2-p_1^2)$}} &{\rm{\footnotesize $p_1(\mathbb{C}P^{4})=5$}}&{\rm{\footnotesize $\sigma(\mathbb{C}P^{4})=\frac{1}{45}(7\times 10-25)=1$}}\\
&&{\rm{\footnotesize $p_2(\mathbb{C}P^{4})=10$}}&\\
\hline
 {\rm{\footnotesize $3$}}&{\rm{\footnotesize $\mathcal{L}_3(p_1,p_2,p_3)=\frac{1}{945}(62p_3-13p_2p_1+2p_1^3)$}}&{\rm{\footnotesize $p_1(\mathbb{C}P^{6})=7$}}&{\rm{\footnotesize $\sigma(\mathbb{C}P^{6})=\frac{1}{945}(62\times 35-13\times 21\times 7+2\times 343)=1$}} \\
 &&{\rm{\footnotesize $p_2(\mathbb{C}P^{6})=21$}}&\\
 &&{\rm{\footnotesize $p_3(\mathbb{C}P^{6})=35$}}& \\
 \hline
{\rm{\footnotesize $4$}}&{\rm{\footnotesize $\mathcal{L}_3(p_1,p_2,p_3,p_4)=\frac{1}{14175}(381p_4-71p_3p_1-19p_2^2+22p_2p_1^2-3p_1^4)$}}&
{\rm{\footnotesize $p_1(\mathbb{C}P^{8})=9$}}&{\rm{\footnotesize $\sigma(\mathbb{C}P^{8})=\frac{1}{14175}(48006-53676-24624+64152-19683)=1$}} \\
 &&{\rm{\footnotesize $p_2(\mathbb{C}P^{8})=36$}}& \\
  &&{\rm{\footnotesize $p_3(\mathbb{C}P^{8})=84$}}& \\
   &&{\rm{\footnotesize $p_4(\mathbb{C}P^{8})=126$}}& \\
 \hline
\end{tabular}$}
\end{table}

\begin{table}[t]
\caption{\rm{\footnotesize Bernoulli numbers $B_n=-\sum_{1\le k\le n+1}\frac{(-1)^k}{k}\binom{n+1}k \sum_{1\le j\le k}j^n\in\mathbb{Q}$, $n\ge 0$.}}
\label{bernoulli-numbers}
\scalebox{0.7}{$\begin{tabular}{|l|l|l|}
  \hline
\hfil{\rm{\footnotesize $n$}}\hfil&\hfil{\rm{\footnotesize Numerator}}\hfil&\hfil{\rm{\footnotesize Denominator}}\hfil\\
 \hline
\hfil{\rm{\footnotesize $0$}}\hfil&\hfil{\rm{\footnotesize $1$}}\hfil&\hfil{\rm{\footnotesize $1$}}\hfil\\
\hfil{\rm{\footnotesize $1$}}\hfil&\hfil{\rm{\footnotesize $-1$}}\hfil&\hfil{\rm{\footnotesize $2$}}\hfil\\
\hfil{\rm{\footnotesize $2$}}\hfil&\hfil{\rm{\footnotesize $1$}}\hfil&\hfil{\rm{\footnotesize $6$}}\hfil\\
\hfil{\rm{\footnotesize $4$}}\hfil&\hfil{\rm{\footnotesize $-1$}}\hfil&\hfil{\rm{\footnotesize $30$}}\hfil\\
\hfil{\rm{\footnotesize $6$}}\hfil&\hfil{\rm{\footnotesize $1$}}\hfil&\hfil{\rm{\footnotesize $42$}}\hfil\\
\hfil{\rm{\footnotesize $8$}}\hfil&\hfil{\rm{\footnotesize $-1$}}\hfil&\hfil{\rm{\footnotesize $30$}}\hfil\\
\hfil{\rm{\footnotesize $10$}}\hfil&\hfil{\rm{\footnotesize $5$}}\hfil&\hfil{\rm{\footnotesize $66$}}\hfil\\
\hfil{\rm{\footnotesize $12$}}\hfil&\hfil{\rm{\footnotesize $-691$}}\hfil&\hfil{\rm{\footnotesize $2730$}}\hfil\\
\hfil{\rm{\footnotesize $14$}}\hfil&\hfil{\rm{\footnotesize $7$}}\hfil&\hfil{\rm{\footnotesize $6$}}\hfil\\
\hfil{\rm{\footnotesize $16$}}\hfil&\hfil{\rm{\footnotesize $-3617$}}\hfil&\hfil{\rm{\footnotesize $510$}}\hfil\\
\hfil{\rm{\footnotesize $18$}}\hfil&\hfil{\rm{\footnotesize $43862$}}\hfil&\hfil{\rm{\footnotesize $798$}}\hfil\\
\hline
\multicolumn{3}{l}{\rm{\footnotesize  $B_n=0$, $n$= odd $>1$.}}\\
  \end{tabular}$}
\end{table}

{\em 6)} The intersection form of a $4k$-dimensional manifold $M$ has a $1$-quadratic function iff it has $2k^{th}$ Wu class $v_{2k}(M)=0\in H^{2k}(M;\mathbb{Z}_2)$, in which case $\sigma(M)\equiv 0\: (\MOD\, 8)$.

\end{theorem}

\begin{definition}\label{parallelizable-manifold}
 An $n$-dimensional manifold $M$ is {\em parallelizable} if its tangent $n$-plane bundle $\tau_M:M\to BO(n)$ is trivial, i.e., isomorphic to $M\times\mathbb{R}^n\to M\equiv \epsilon^n$.

Two vector bundles $\xi_1$ and $\xi_2$ over a same base $M$ are {\em stably isomorphic} if $\xi_1\bigoplus\epsilon_1\cong\xi_2\bigoplus\epsilon_1$, where $\epsilon_i$, $i=1,2$ are vector bundles over $M$ with dimensions such that if $M$ is a complex of dimension $r$, the total fiber dimensions of the Whitney sums exceeds $r$. Such bundles are said to be in {\em stable range}.
\end{definition}
\begin{proposition}
A connected compact $n$-manifold $M$ with non-trivial boundary, is parallelizable iff it is stably parallelizable.
\end{proposition}

\begin{proof}
In fact $M$ has the homotopy type of an $(n-1)$-complex and thus $TM$ is in the stable range.
\end{proof}

\begin{proposition}
The set of framed $n$-manifolds properly contains the set of parallelizable $n$-manifolds.
\end{proposition}

\begin{example}[Bott-Milnor \cite{BOTT-MILNOR}]
{\em 1)} The sphere $S^n$ is framed with $S^n\times\mathbb{R}\subset\mathbb{R}^{n+1}$, $TS^n\bigoplus\epsilon\cong\epsilon^{n+1}$, but not necessarily parallelizable. (See Tab. \ref{parallelizable-n-spheres}.)
\begin{table}[h]
\caption{Parallelizable $S^n$.}
\label{parallelizable-n-spheres}
\begin{tabular}{|c|c|c|c|c|c|c|c|c|c|c|}
  \hline
  \hfil{\rm{\footnotesize $n$}}\hfil& \hfil{\rm{\footnotesize $1$}}\hfil&\hfil{\rm{\footnotesize $2$}}\hfil& \hfil{\rm{\footnotesize $3$}}\hfil&\hfil{\rm{\footnotesize $4$}}\hfil& \hfil{\rm{\footnotesize $5$}}\hfil&\hfil{\rm{\footnotesize $6$}}\hfil& \hfil{\rm{\footnotesize $7$}}\hfil&\hfil{\rm{\footnotesize $>7$}}\hfil\\
 \hline
\hfil{\rm{\footnotesize parallelizable}}\hfil& \hfil{\rm{\footnotesize YES}}\hfil&\hfil{\rm{\footnotesize NO}}\hfil& \hfil{\rm{\footnotesize YES}}\hfil&\hfil{\rm{\footnotesize NO}}\hfil& \hfil{\rm{\footnotesize NO}}\hfil&\hfil{\rm{\footnotesize NO}}\hfil& \hfil{\rm{\footnotesize YES}}\hfil&\hfil{\rm{\footnotesize NO}}\hfil\\
 \hline
\end{tabular}
\end{table}

{\em 2}) All spheres are stably parallelizable.

{\em 3)} Every homotopy sphere $\Sigma^n$ is stably parallelizable.
\end{example}

\begin{definition}[Pontrjagin-Thom construction framed-cobordism]
If $(M,\varphi)$ is any $n$-manifold with {\em framing} $\varphi:\nu(M)\cong M\times\mathbb{R}^{k-n}$ of the normal bundle in $\mathbb{R}^k$, the same definition yields a map $p(M,\varphi)\in\pi_n^s$. If $(M_1,\varphi_1)\sqcup (M_2,\varphi_2)\subset\mathbb{R}^k$ form the framed boundary of a $(n+1)$-manifold $(W,\partial W,\Phi)\subset(\mathbb{R}^k\times[0,\infty),\mathbb{R}^k\times \{0\})$, we say that are {\em framed cobordant}. The framed cobordism is an equivalence relation and the corresponding set of framed cobordism classes is denoted by $\Omega^{fr}_n$. This is an abelian additive group with respect the operation of disjoint union $\sqcup$. The class $[\varnothing]$ is the zero of $\Omega^{fr}_n$. Then one has the canonical isomorphism given in {\em(\ref{isomorphisms-framed-cobordism-groups-n-stems})}.\footnote{In Tab. \ref{calculated-h-cobordism-groups-homotopy-sphere} are reported the $n$-stems for $0\le n\le 17$.}

\begin{equation}\label{isomorphisms-framed-cobordism-groups-n-stems}
\left\{
\begin{array}{l}
\Omega^{fr}_n\cong  \pi_n^s \\
 M\mapsto f_M:S^{n+k}\to S^{n+k}/(S^{n+k}\setminus(M\times\mathbb{R}^k))=\Sigma^kM^+\to S^k\\
\end{array}
\right.
\end{equation}
where $M^+\equiv M\bigcup\{pt\}$, and $\Sigma^k$ is the $k$-suspension functor for spectra.
\end{definition}

\begin{example}[Smale's paradox and framed cobordism]\label{smale-paradox-framed-cobordism}
Let us consider the so-called {\em Smale's paradox} turning a sphere $S^2\subset\mathbb{R}^3$ inside out. (See also \cite{PRA16}.) Let us denote by ${}^{-}S^2$ the sphere $S^2$ with reversed orientation. Let us first note that these surfaces are characterized by the same generalized curvature integra. It is useful to recall here some definitions and results about this topological invariant. The {\em generalized Gauss map} of an $n$-dimensional framed manifold $M$ with $f:M\hookrightarrow\mathbb{R}^{n+k}$, $\nu_f\cong\epsilon^k$, is the map $c:M\to V_{n+k,k}$, $x\mapsto ((\nu_f)_x =\mathbb{R}^k\hookrightarrow T_{f(x)}\mathbb{R}^{n+k}=\mathbb{R}^{n+k})$ and classifies the tangent $n$-planes bundle $\tau_M:M\to BO(n)$, with the $k$-stable trivialization
$TM\bigoplus\epsilon^k\cong
TM\bigoplus\nu_f=\epsilon^{n+k}=T\mathbb{R}^{n+k}|_M$.\footnote{$V_{n+k,k}\cong O(n+k)/O(n)$ is the {\em Stiefel space} of orthonormal $k$-frames in $\mathbb{R}^{n+k}$, or equivalently of isometries
$\mathbb{R}^k\to\mathbb{R}^{n+k}$. $V_{n+k,k}$ is $(n-1)$-connected with $H_n(V_{n+k,k})=\mathbb{Z}$, if $n\equiv 0\: (\MOD\, 2)$ or if $k=1$, and $H_n(V_{n+k,k})=\mathbb{Z}_2$,
if $n\equiv 1\: (\MOD\, 2)$ and $k>1$. One has $G_{n+k,k}=V_{n+k,k}/O(k)$, where $G_{n+k,k}$ is the Grassmann space of $k$-dimensional subspaces of $\mathbb{R}^{n+k}$. Then the classifying space for $n$-planes is $BO(n)=\mathop{\lim}\limits_{\overrightarrow{k}}G_{n+k,k}$, and the corresponding stable classifying space is $BO=\mathop{\lim}\limits_{\overrightarrow{n}}BO(n)$.} The {\em generalized curvatura integra} of $M$ is the degree of the generalized Gauss map.
\begin{equation}\label{generalized-curvatura-integra}
  c_*[M]\in H_{n}(V_{n+k,k})=\left\{
  \begin{array}{l}
  \mathbb{Z},\: \hbox{\rm if $n\equiv 0\: (\MOD\, 2)$}\\
   \mathbb{Z}_2,\: \hbox{\rm if $n\equiv 1\: (\MOD\, 2)$}\\
\end{array}
  \right\}\: (k>1).
\end{equation}
The curvatura integra of an $n$-dimensional framed manifold $M$ can be expressed with the {\em Kervaire's formula} given in {\em(\ref{kervaire-formula-curvatura integra})}.
\begin{equation}\label{kervaire-formula-curvatura integra}
     c_*[M]= Hopf(M)+\left\{
  \begin{array}{l}
 \chi(M)/2\in \mathbb{Z},\: \hbox{\rm if $n\equiv 0\:  (\MOD\, 2)$}\\
  \chi_{1/2}(M)\in \mathbb{Z}_2,\: \hbox{\rm if $n\equiv 1\:  (\MOD\, 2)$}\\
\end{array}
  \right.
\end{equation}
where
\begin{equation}\label{kervaire-semicharacteristic}
     \chi_{1/2}(M)=\sum_{0\le j\le (n-1)/2}\dim_{\mathbb{Z}_2}H_j(M;\mathbb{Z})\in\mathbb{Z}_2
\end{equation}
is called the {\em Kervaire semicharacteristic}, and
\begin{equation}\label{hopf-invariant-curvatura-integra}
     Hopf[M]= \left\{
  \begin{array}{l}
 0\in \mathbb{Z},\: \hbox{\rm if $n\equiv 0\:  (\MOD\, 2)$}\\
  H_2(F)\in \mathbb{Z}_2,\: \hbox{\rm if $n\equiv 1\:  (\MOD\, 2)$}\\
\end{array}
  \right.
\end{equation}
with $F:S^{n+k}\to S^k$ the Pontrjagin-Thom map, and $H_2(F)$ determined by the {\em mod 2 Hopf invariant}. This is the morphism
\begin{equation}\label{mod-2-hopf-invariant}
 \left\{
  \begin{array}{l}
 H_2:\pi_{n+k}(S^k)\to \mathbb{Z}\\
  (F:S^{n+k}\to S^k)\mapsto H_2(F),\: (m\ge 1)\\
\end{array}
  \right.
\end{equation}
determined by the Steenrod square in the mapping cone $X=S^k\bigcup_{F}D^{n+k+1}$. If $a=1\in H^k(X;\mathbb{Z}_2)=\mathbb{Z}_2$, $b=1\in H^{n+k+1}(X;\mathbb{Z}_2)=\mathbb{Z}_2$, then $S_q^{n+1}(a)=H_2(F)b\in H^{n+k+1}(X;\mathbb{Z}_2)$. One has (Adams) that $H_2=0$ for $n\not=1,3,7$. $Hopf(M)$ is a framed cobordism invariant.

Taking into account that $\chi(S^2)=\chi({}^{-}S^2)=2$, and that $Hopf(S^2)=Hopf({}^{-}S^2)=0$, we get $c_*[S^2]=c_*[{}^{-}S^2]=1$. Furthermore one has $\Omega_2^{fr}\cong\pi_2^s=\mathbb{Z}_2\cong\Omega_2$. In order to see that $S^2$ is cobordant with ${}^{-}S^2$ it is enough to prove that ${}^{-}S^2$ can be obtained by $S^2$ by a sequence of surgeries. (See Theorem \ref{handle-decomposition-bordisms}.) In fact, we can write the oriented $S^2$ as $S^2=D^2_{W}\bigcup_{S^1}D^2_{E}$, where $D^2_{W}$ and $D^2_{E}$ are two oriented discs in such a way that $S^2$ is oriented with outgoing normal unitary vector field. (Fig. \ref{surgeries-smale-paradox}(a).) By a surgery we can remove $D^2_{E}$ and smoothly add another  $D^2_{E}$ on the left of  $D^2_{W}$. (Fig. \ref{surgeries-smale-paradox}(b).) Next by an orientation preserving diffeomorphisms, we get ${}^{-}S^2$, the surface represented in Fig. \ref{surgeries-smale-paradox}(c).

\begin{figure}[ht]
\centering
\scalebox{0.7}{$\centerline{\includegraphics[height=4cm]{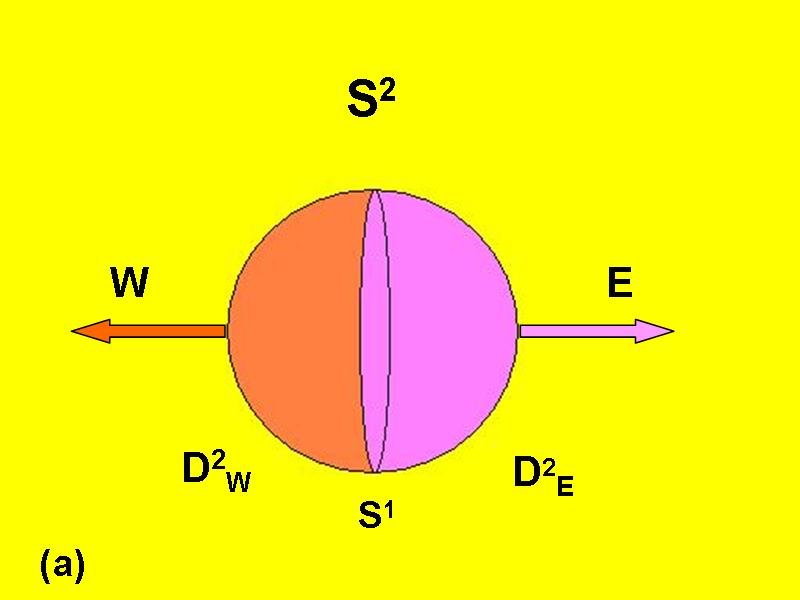} \includegraphics[height=4cm]{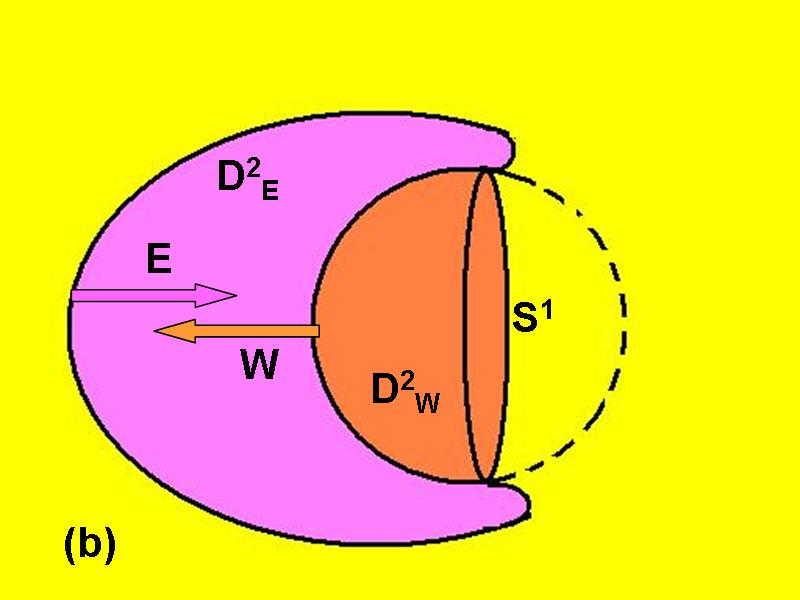} \includegraphics[height=4cm]{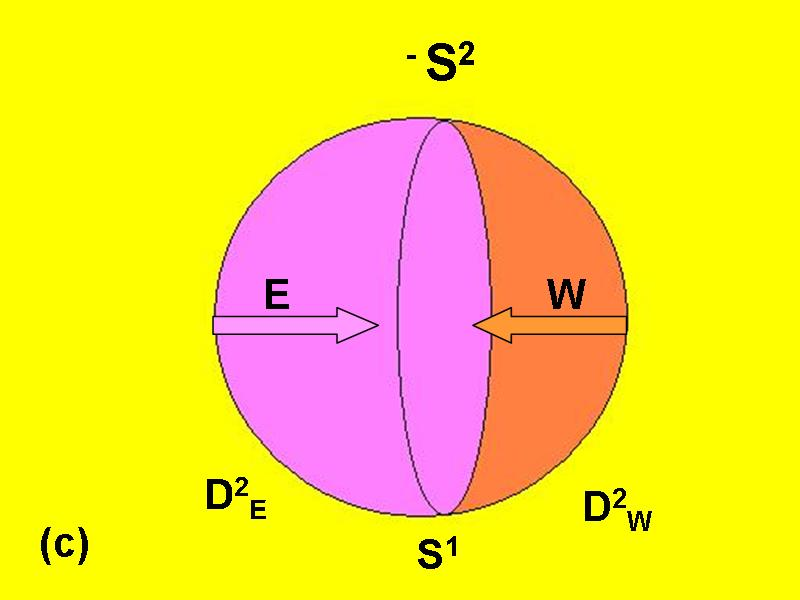}}$}
\caption{Surgery and Smale's paradox turning a sphere $S^2\subset\mathbb{R}^3$ inside out.}
\label{surgeries-smale-paradox}
\end{figure}
In conclusion $S^2\sqcup {}^{-}S^2=\partial W$, where $W\cong (S^2\times\{0\})\times I\cong ({}^{-}S^2\times\{1\})\times I\subset \mathbb{R}^3\times[0,\infty)$. Therefore, $S^2$ is framed cobordant with $ {}^{-}S^2$. Furthermore $S^2\cong {}^{-}S^2$ (diffeomorphism reversing orientation), that agrees with the well known result in differential topology that two connected, compact, orientable surfaces are diffeomorphic iff they have the same genus, the same Euler characteristics and the same number of boundaries. (See, e.g. \cite{HIRSCH2}.)

Another proof that $S^2$ and $-S^2$ are cobordant can be obtained by the Arf invariant. Let us recall that if $M$ is a framed surface $M\times\mathbb{R}^k\subset\mathbb{R}^{k+2}$, the intersection form $(H^1(M),\lambda)$ has a canonical $(-1)$-quadratic function $\mu:H^1(M)=H_1(M)\to Q_{-1}(\mathbb{Z})=\Omega_1^{fr}=\mathbb{Z}_2$, given by $x\mapsto (x:S^1\hookrightarrow M)$, sending each $x\in H^1(M)$ to an embedding $x:S^1\hookrightarrow M$ with a corresponding framing $S^1\times\mathbb{R}^{k+1}\subset\mathbb{R}^{k+2}$, $\delta\nu_x:\nu_x\bigoplus\epsilon^k\cong\epsilon^{k+1}$. Then one has the isomorphism $Arf:\Omega_2^{fr}=\pi_2^s\cong\mathbb{Z}_2$, $[M]\mapsto Arf(H^1(M),\lambda,\mu)$. In the particular case that $M=S^2\sqcup-S^2$, we get $H^1(M)=0$ and $Arf(H^1(M),\lambda,\mu)=0$, hence must necessarily be $[M=S^2\sqcup-S^2]=0\in\Omega_2^{fr}$. This agrees with the fact that $\Omega_2^{fr}=\mathbb{Z}_2=\Omega_2$, and that both surfaces $S^2$ and $-S^2$ belong to $0\in \Omega_2$ since are orientable ones.
\end{example}
\begin{definition}
An $n$-dimensional manifold $V$ with boundary $\partial V$, is {\em almost framed} if the open manifold $V\setminus\{pt\}$ framed: $(V\setminus\{pt\})\times\mathbb{R}^k\subset\mathbb{R}^{n+k}$ (for $k$ large enough).
\end{definition}

\begin{theorem}[Properties of almost framed manifold]

{\em 1)} An almost framed manifold $V$, with $\partial V\not=\varnothing$ is a framed manifold and a parallelizable manifold.

{\em 2)} If $V$ is an almost framed manifold with $\partial V=\varnothing$, then there is a {\em framing obstruction}
 \begin{equation}\label{framing-obstructions}
 \mathfrak{o}(V)\in\ker(J:\pi_{n-1}(O)\to\pi_{n-1}^s)
\end{equation}
in the sense that $V$ is framed iff $\mathfrak{o}(V)=0$.

{\em 3) (Kervaire invariant for almost framed manifolds).} Let $(M,\partial M)$ be a $(4k+2)$-dimensional almost framed manifold with boundary such that either $\partial M=\varnothing$ or $H_\bullet(M)=H_\bullet(S^{4k+1})$, so that $(H^{2k+1}(M;\mathbb{Z}_2),\lambda\,\mu)$ is a nonsingular quadratic form over $\mathbb{Z}_2$. The {\em Kervaire} of $M$ is defined in {\em(\ref{kervariant-inavriant})}.
\begin{equation}\label{kervariant-inavriant}
 Kervaire(M)=Arf(H^{2k+1}(M;\mathbb{Z}_2),\lambda\,\mu).
\end{equation}
One has the following propositions.

{\em (i)} If $\partial M=\varnothing$ and $M=\partial N$ is the boundary of a $(4k+3)$-dimensional almost framed manifold $N$, then $Kervaire(M)=0\in\mathbb{Z}_2$.

{\em (ii)} The Kervaire of a manifold identifies a framed cobordism invariant, i.e., it defines a map $Kervaire:\Omega^{fr}_{4k+2}=\pi_{4k+2}^s\to\mathbb{Z}_2$ that is $0$ if $k\not=2^i-1$.

{\em (iii)} There exist $(4k+2)$-dimensional framed manifolds $M$ with $Kervaire(M)=1$, for $k=0,1,3,7$. For $k=0,1,3$ can take $M=S^{2k+1}\times S^{2k+1}$.

{\em 4) (Kervaire-Milnor's theorem on almost framed manifolds).} Let us denote by $\Omega^{afr}_n$ the cobordism group of closed $n$-dimensional almost framed manifolds. One has the exact sequence in {\em(\ref{almost-framed-manifolds-exact-sequence})}.
 \begin{equation}\label{almost-framed-manifolds-exact-sequence}
\xymatrix{\Omega^{afr}_n\ar[r]^{\mathfrak{o}}&\pi_{n-1}(O)\ar[r]^{J}&\pi_{n-1}^s}
\end{equation}

$\bullet$\hskip 3pt For a $4k$-dimensional almost framed manifold $V$ one has the framing obstruction reported in {\em(\ref{framing-obstructions-b})}.\footnote{For $k=1$ one has $j_1=24$; $\mathfrak{o}(V)=p_1(V)/2\in 24\mathbb{Z}\subset\pi_{3}(O)=\mathbb{Z}$.}
 \begin{equation}\label{framing-obstructions-b}
\scalebox{0.9}{$\left\{\begin{array}{ll}
         \mathfrak{o}(V) & =p_k(V)/(a_k(2k-1)!) \\
         & \ker(J:\pi_{4k-1}(O)\to\pi_{4k-1}^s) \\
         &=j_k\mathbb{Z}\subset\pi_{4k-1}(O)=\mathbb{Z}\\
       \end{array}
\right\}\begin{array}{l}
p_k(V)\in H^{4k}(V)=\mathbb{Z}\: \hbox{\rm  (Pontryagin class)}\\
a_k=\left\{
\begin{array}{l}
1\: \hbox{\rm  for $k\equiv 0\: (\MOD\, 2)$}\\
2\: \hbox{\rm  for $k\equiv 1\: (\MOD\, 2)$}\\
\end{array}
\right\}\: j_k=\DEN(\frac{B_k}{4k})\\
\end{array}$}
\end{equation}

{\em 5) (Kervaire-Milnor's theorem on almost framed manifolds-2).}  Let $P_n$ be the cobordism group of $n$-dimensional framed manifolds with homotopy sphere boundary. ($P_n$ is called the {\em $n$-dimensional simply-connected surgery obstruction group}.) For $n\ge 4$ $\Theta_n$ is finite, with the short exact sequence given in {\em(\ref{almost-framed-manifolds-cobordism-groups-homotopy-spheres-short-exact-sequence})}.
  \begin{equation}\label{almost-framed-manifolds-cobordism-groups-homotopy-spheres-short-exact-sequence}
\xymatrix{0\ar[r]&\COKER(a:\Omega^{afr}_{n+1}\to P_{n+1})\ar[r]^{b}&\Theta_n\ar[r]^{c}&\ker(a:\Omega^{afr}_n\to P_n)\ar[r]&0\\}
\end{equation}
and
$$\ker(a)\subseteq \COKER(J:\pi_{n}\to\pi_n^s)=\ker(\mathfrak{o}:\Omega^{afr}_n\to\pi\pi_{n-1}(O)).$$

$\bullet$\hskip 3pt In {\em(\ref{calculated-pn})} are reported the calculated groups $P_n$.
 \begin{equation}\label{calculated-pn}
\left\{\begin{array}{l}
         P_{2n+1}=0 \\
         P_{n}=\left\{\begin{array}{l}
\mathbb{Z}\: \hbox{\rm  if $n\equiv 0\: (\MOD\, 4)$}\\
0\: \hbox{\rm  if $n\equiv 1\: (\MOD\, 4)$}\\
\mathbb{Z}_2\: \hbox{\rm if $n\equiv 2\: (\MOD\, 4)$}\\
0\: \hbox{\rm  if $n\equiv 3\: (\MOD\, 4)$}\\
\end{array}\right\}\\
       \end{array}\right.
\end{equation}

{\em 6) (Kervaire-Milnor's braid $n\ge5$).} For $n\ge 5$ there is the exact commutative braid diagram given in {\em(\ref{kervaire-milnor-braid})}.
 \begin{equation}\label{kervaire-milnor-braid}
\scalebox{0.6}{$\xymatrix{&\pi_{n+1}(G/PL)=P_{n+1}\ar[dr]\ar@/^2pc/[rr]^b&& \pi_{n}(PL/O)=\Theta_{n}\ar[dr]^c\ar@/^2pc/[rr]^0&&\pi_{n-1}(O)\\
\pi_{n+1}(G/O)=\Omega^{afr}_{n+1}\ar[ur]^a\ar[dr]_{\mathfrak{o}}&&\pi_n(PL)\ar[ur]\ar[dr]&&\pi_n(G/O)=\Omega^{afr}_n\ar[ur]^{\mathfrak{o}}\ar[dr]^a&\\
&\pi_n(O)\ar[ur]\ar@/_2pc/[rr]_J&&\pi_n(G)=\pi^s_n=\Omega^{fr}_n\ar[ur]\ar@/_2pc/[rr]&&\pi_n(G/PL)=P_n}$}
\end{equation}
In {\em(\ref{kervaire-milnor-braid})} the mappings $a$, $b$ and $c$ are defined in {\em(\ref{definitions-a-b-c-kervaire-milnor-braid})}.
\begin{equation}\label{definitions-a-b-c-kervaire-milnor-braid}
\left\{
\begin{array}{l}
 a:\Omega^{afr}_{2n}\to P_{2n},\:  a(M)=\left\{
\begin{array}{l}
  \frac{1}{8}\sigma(M)\in\mathbb{Z} \hskip 3pt\hbox{\rm if $n\equiv 0\: (\MOD\, 2)$} \\
  Kervaire(M)\in\mathbb{Z}_2 \hskip 3pt\hbox{\rm if $n\equiv 1\: (\MOD\, 2)$} \\
\end{array}\right\}\in P_{2n}.\\
 b:P_{2n}\to\Theta_{2n-1} ,\:  b(M)=\hbox{\rm plumbing construction $\Sigma=\partial M$.}\\
 c:\Theta_n\to\Omega^{afr}_n,\:  c(\Sigma)=[\Sigma]\in\Omega^{afr}_n\\
\end{array} \right.
\end{equation}
The image of $b$ is denoted $bP_n\triangleleft\Theta_{n-1}$. Then if $\Sigma\in bP_n$, then $\Sigma=\partial V$, where $V$ is a $n$-dimensional framed differentiable manifold. Furthermore, by considering the mapping $c$ as $c:\Theta_n\to\pi_n(G/O)$, it sends an $n$-dimensional exotic sphere $\Sigma$ to its fibre-homotopy trivialized stable normal bundle.

{\em 7) (Kervaire-Milnor's braid $n=4k+2\ge5$).} For $n=4k+2\ge 5$ the exact commutative braid diagram given in {\em(\ref{kervaire-milnor-braid})} becomes the one reported in {\em(\ref{kervaire-milnor-braid-b})}.

\begin{equation}\label{kervaire-milnor-braid-b}
\scalebox{0.6}{$\xymatrix{P_{4k+3}=0\ar[dr]\ar@/^2pc/[rr]^b&&\Theta_{4k+2} \ar[dr]^c\ar@/^2pc/[rr]^0&&\pi_{4k+1}(O)\ar[dr]^c\ar\ar@/^2pc/[rr]^J&&\pi_{4k+1}(G)\\
&\pi_{4k+2}(PL)\ar[ur]\ar[dr]&&\pi_{4k+2}(G/O)\ar[dr]^a\ar[ur]^{\mathfrak{o}}&&\pi_{4k+1}(PL)\ar[ur]\ar[dr]&\\
\pi_{4k+2}(O)=0\ar[ur]\ar@/_2pc/[rr]_J&&\pi_{4k+2}(G)\ar[ur]\ar@/_2pc/[rr]_{K}&&P_{4k+2}=\mathbb{Z}_2\ar[ur]
\ar@/_2pc/[rr]_b&&\Theta_{4k+1}}$}
\end{equation}
$K$ is the Kervaire invariant on the $(4k+2)$-dimensional stable homotopy group of spheres
\begin{equation}\label{k-invariant}
\left\{\begin{array}{ll}
K:\pi_{4k+2}(G)&=\pi_{4k+2}^s=\mathop{\lim}\limits_{\overrightarrow{j}}\pi_{j+4k+2}(S^j)\\
&=\Omega^{fr}_{4k+2}\to P_{4k+2}=\mathbb{Z}_2.\\
\end{array} \right.
\end{equation}

$\bullet$\hskip 3pt $K$ is the surgery obstruction: $K=0$ iff every $(4k+2)$-dimensional framed differentiable manifold is framed cobordant to a framed exotic sphere.

$\bullet$\hskip 3pt The exotic sphere group $\Theta_{4k+2}$ fits into the exact sequence {\em(\ref{(4k+2)-exotic-sphere-group-exact-sequence})}.
\begin{equation}\label{(4k+2)-exotic-sphere-group-exact-sequence}
\scalebox{0.9}{$\xymatrix{0\ar[r]&\Theta_{4k+2}\ar[r]&\pi_{4k+2}(G)\ar[r]_{K}&\mathbb{Z}_2\ar[r]&\ker(\pi_{4k+1}(PL)\to \pi_{4k+1}(G))\ar[r]&0}$}
\end{equation}
$\bullet$\hskip 3pt $a:\pi_{4k+2}(G/O)\to\mathbb{Z}_2$ is the surgery obstruction map sending a normal map $(f,b):M\to S^{4k+2}$ to the Kervaire invariant of $M$.

$\bullet$\hskip 3pt $b:P_{4k+2}=\mathbb{Z}_2\to\Theta_{4k+1}$ sends the generator $1\in\mathbb{Z}_2$ to the boundary $b(1)=\Sigma^{4k+1}=\partial W$ of the Milnor plumbing $W$ of two copies of $TS^{2k+1}$ using the standard rank $2$ quadratic form $\scalebox{0.4}{$\left(
                                    \begin{array}{cc}
                                      1 & 1 \\
                                      0 & 1 \\
                                    \end{array}
                                  \right)$}$ over $\mathbb{Z}$ with Arf invariant $1$.
The subgroup $bP_{4k+2}\triangleleft\Theta_{4k+1}$ represents the $(4k+1)$-dimensional exotic spheres $\Sigma^{4k+1}=\partial V$ that are boundaries of framed $(4k+2)$-dimensional differentiable manifolds $V$.

If $k$ is such that $K=0$ (e.g. $k=2$) then $bP_{4k+2}=\mathbb{Z}_2\triangleleft\Theta_{4k+1}$ and if $\Sigma^{4k+1}=1\in bP_{4k+2}$, then the $(4k+2)$-dimensional manifold $M=V\bigcup_{\Sigma^{4k+1}}D^{4k+2}$ is a PL manifold without a differentiable structure.

$\bullet$\hskip 3pt For any $k\ge 1$ the following propositions are equivalent.

{\em(i)} $K:\pi_{4k+2}(G)=\pi^s_{4k+2}\to\mathbb{Z}_2$ is $K=0$.

{\em(ii)} $\Theta_{4k+2}\cong\pi_{4k+2}(G)$.

{\em(iii)} $\ker(\pi_{4k+1}(PL)\to\pi_{4k+1}(G))\cong\mathbb{Z}_2$.

{\em(iv)} Every simply-connected $(4k+2)$-dimensional Poincar\'e complex $X$ with a vector bundle reduction $\widetilde{\nu}_x:X\to BO$ of the Spivak normal fibration $\nu_x:X\to BG$ is homotopy equivalent to a closed $(4k+2)$-dimensional differentiable manifold.

\end{theorem}

\begin{theorem}[Pontrjagin, Thom, Kervaire-Milnor]
{\em 1)} Let $bP_{n+1}$ denote the set of those h-cobordism classes of homotopy spheres which bound parallelizable manifolds.\footnote{$\: bP_{n+1}$ is a subgroup of $\Theta_n$. If $\Xi_1,\Xi_2\in bP_{n+1}$, with bounding parallelizable manifolds $W_1$ and $W_2$ respectively, then $\Xi_1\sharp\Xi_2$ bounds the parallelizable manifold $W_1\sharp W_2$, (commutative sum along the boundary).} For $n\not= 3$, there is a short exact sequence
\begin{equation}\label{split-short-exact-sequence-homotopy-spheres}
 \xymatrix{0\ar[r]&bP_{n+1}\ar[r]&\Theta_n\ar[r]&\Theta_n/bP_{n+1}\ar[r]&0}
\end{equation}
where the left hand group is finite cyclic. Furthermore, there exists an homomorphism $J:\pi_n(SO)\to\pi_n^s$ such that $\Theta_n/bP_{n+1}$ injects into $\pi_n^s/J(\pi_n(SO))$ via the Pontrjagin-Thom construction. When $n\not= 2^j-2$, the right hand group is isomorphic to $\pi_n^s/J(\pi_n(SO))$.

{\em 2)} If $\Sigma^n$ bounds a parallelizable manifold, it bounds a parallelizable manifold $W$ such that $\pi_j(W)=0$, $j<n/2$.

{\em 3)} For any $k\ge 1$, $bP_{2k+1}=0$.
\end{theorem}

\begin{proof}
For any manifold $M$ with stably trivial normal bundle with framing $\varphi$, there is a homotopy class $p(M,\varphi)$, depending on the framed cobordism class of $(M,\varphi)$. If $p(M)\subset\pi_n^s$ is the set of all $p(M,\varphi)$ where $\varphi$ ranges over framings of the normal bundle, it follows that $0\in p(M)$ iff $M$ bounds a parallelizable manifold. (This is a result by Pontrjagin and Thom.) In particular, the set $p(S^n)$ has an explicit description. More precisely, the Whitehead $J$-homomorphism $:\pi_n(SO(r))\to\pi_{n+r}(S^r)$ is defined by $J:(\alpha:S^n\to SO(r))\mapsto(J(\alpha):S^{n+r}\to S^r)$ that is the Pontrjagin-Thom map of $S^n\subset S^{n+r}$, with the framing $b_\alpha:S^n\times D^r\subset S^{n+r}=S^n\times D^r\bigcup D^{n+1}\times S^{r-1}$, $(x,y)\mapsto(x,\alpha(x)(y))$. Therefore, the map $J(\alpha):S^{n+r}\to S^r$, is obtained by considering $S^{n+r}=(S^n\times D^r)\bigcup(D^{n+1}\times S^{r-1})$ and sending $(x,y)\in D^n\times D^r$ to $\alpha(x)y\in D^r/\partial D^r=S^r$ and $D^{n+1}\times S^{r-1}$ to the collapsed $\partial D^r$. Then $J:\pi_n(SO)=\mathop{\lim}\limits_{\overrightarrow{r}}\pi_n(SO(r))\to\mathop{\lim}\limits_{\overrightarrow{r}}\pi_{n+r}(S^r)=\pi_n^s$ is the stable limit of the maps $J(\alpha)$ as $r\to\infty$, and $p(S^n)$ is the image $J(\pi_n(SO))\subset \pi_n^s=\Omega_n^{fr}$, hence one has that to $\alpha$ there corresponds the framed cobordism class $(S^n,b_\alpha)$.
\end{proof}

The characterization of global solutions of a PDE $
E_k\subseteq J^k_{n}(W)$, in the category ${\mathfrak
M}_\infty$, can be made by means of its integral bordism groups
$\Omega_{p}^{E_k}$, $p\in\{0,1,\dots,n-1\}$. Let us shortly recall some fundamental
definitions and results about.

\begin{definition}
Let $f_i:X_i\to  E_k$, $f_i(X_i)\equiv N_i\subset E_k$,
$i=1,2$, be $p$-dimensional admissible compact closed smooth
integral manifolds of $E_k$. The admissibility
requires that $N_i$ should be contained into some solution $V\subset
E_k$, identified with a $n$-chain, with coefficients in
$A$. Then, we say that they are {\em $E_k$-bordant} if there
exists a $(p+1)$-dimensional smooth manifolds
$f:Y\to E_k$, such that $\partial Y=X_1\sqcup X_2$,
$f|_{X_i}=f_i$, $i=1,2$, and $V\equiv f(Y)\subset E_k$ is an
admissible integral manifold of $E_k$ of dimension
$(p+1)$. We say that $N_i$, $i=1,2$, are {\em $
E_k$-bordant} if there exists a $(p+1)$-dimensional
smooth manifolds $f:Y\to J^k_{m|n}(W)$, such that
$\partial Y=X_1\sqcup X_2$, $f|_{X_i}=f_i$, $i=1,2$, and $V\equiv
f(Y)\subset J^k_{n}(W)$ is an admissible integral manifold of
$J^k_{n}(W)$ of dimension $(p+1)$. Let us denote the
corresponding bordism groups by $\Omega_{p}^{ E_k}$ and
$\Omega_{p}(E_k)$, $p\in\{0,1,\dots,n-1\}$, called respectively {\em$p$-dimensional
integral bordism group} of $E_k$ and {\em$p$-dimensional
quantum bordism group} of $E_k$. Therefore these bordism groups
work, for $p=(n-1)$, in the category of manifolds that are solutions of $E_k$, and $(J^k_{n}(W),E_k)$. Let us emphasize
that singular solutions of $E_k$ are, in general, (piecewise)
smooth manifolds into some prolongation $(
E_k)_{+s}\subset  J^{k+s}_{n}(W)$, where the set, $\Sigma(V)$,
of {\em singular points} of a solution $V$ is a non-where dense
subset of $V$. Here we consider {\em Thom-Boardman singularities},
i.e., $q\in\Sigma(V)$, if $(\pi_{k,0})_*(T_qV)\not\cong T_qV$.
However, in the case where $E_k$ is a differential equation of
finite type, i.e., the symbols $g_{k+s}=0$, $s\ge 0$, then it
is useful to include also in $\Sigma(V)$, discontinuity points,
$q,q'\in V$, with $\pi_{k,0}(q)=\pi_{k,0}(q')=a\in W$, or with
$\pi_{k}(q)=\pi_{k}(q')=p\in M$, where $\pi_k=\pi\circ\pi_(k,0):
J^k_{n}(W)\to M $. We denote such a set by $\Sigma(V)_S$, and, in
such cases we shall talk more precisely of {\em singular boundary}
of $V$, like $(\partial V)_S=\partial V\setminus\Sigma(V)_S$. Such
singular solutions are also called {\em weak solutions}.
\end{definition}

\begin{remark}
Let us emphasize that weak solutions are not simply exotic
solutions, introduced in Mathematical Analysis in order to describe
''non-regular phenomena''. But their importance is more fundamental
in a theory of PDE's. In fact, by means of such solutions we can
give a full algebraic topological characterization of PDE's. This
can be well understood in Theorem \ref{weak-sing-smooth} below,
where it is shown the structural importance played by weak
solutions. In this respect, let us, first, define some notation to
distinguish between some integral bordisms group types.
\end{remark}

\begin{equation}\label{comm-diag3}
\xymatrix{&0\ar[d]&0\ar[d]&0\ar[d]&\\
0\ar[r]&K^{E_k}_{n-1,w/(s,w)}\ar[d]\ar[r]&K^{
E_k}_{n-1,w}\ar[d]\ar[r]&
K^{E_k}_{n-1,s,w}\ar[d]\ar[r]&0\\
0\ar[r]&K^{E_k}_{n-1,s}\ar[d]\ar[r]&  \Omega^{
E_k}_{n-1}\ar[d]\ar[r]&
\Omega^{E_k}_{n-1,s}\ar[d]\ar[r] &0\\
 &0\ar[r]&\Omega^{
E_k}_{n-1,w}\ar[d]\ar[r]& \Omega^{E_k}_{n-1,w}\ar[d] \ar[r]& 0\\
&&0&0&}
\end{equation}

\begin{definition}
Let $\Omega_{n-1}^{E_k}$, (resp. $\Omega_{n-1,s}^{
E_k}$, resp. $\Omega_{n-1,w}^{E_k}$), be the {\em integral
bordism group} for $(n-1)$-dimensional smooth admissible regular
integral manifolds contained in $E_k$, borded by
smooth regular integral manifold-solutions, (resp.
piecewise-smooth or singular solutions, resp. singular-weak
solutions), of $ E_k$.
\end{definition}

\begin{theorem}\label{weak-sing-smooth}
Let $E_k\subset J^k_n(W)$ be a PDE on the fiber bundle $\pi:W\to M$,
with $\dim(W)=m+n$ and $\dim M=n$.

{\em 1)} One has the exact commutative diagram {\em(\ref{comm-diag3})}.
Therefore, one has the canonical isomorphisms:

\begin{equation}
\left\{
\begin{array}{ll}
K^{E_k}_{n-1,w/(s,w)}\cong K^{E_k}_{n-1,s};&
\Omega^{ E_k}_{n-1}/K^{E_k}_{n-1,s}\cong
\Omega^{E_k}_{n-1,s};\\

\Omega^{E_k}_{n-1,s}/K^{
E_k}_{n-1,s,w}\cong\Omega^{ E_k}_{n-1,w};& \Omega^{E_k}_{n-1}/K^{E_k}_{n-1,w}\cong\Omega^{E_k}_{n-1,w}.\\
\end{array}\right.
\end{equation}

If $E_k$ is formally integrable, then one has the following isomorphisms:

\begin{equation}
\Omega^{E_k}_{n-1}\cong\Omega^{
E_\infty}_{n-1}\cong\Omega^{E_\infty}_{n-1,s};\quad
\Omega^{E_k}_{n-1,w}\cong\Omega^{
E_\infty}_{n-1,w}.\end{equation}

{\em 2)} Let $E_k\subset J^k_{n}(W)$ be a quantum super
PDE that is formally integrable, and completely
integrable. We shall assume that the symbols $
g_{k+s}\not=0$, $s=0,1$. (This excludes the case $k=\infty$.) Then
one has the following isomorphisms: $\Omega_{p,s}^{
E_k}\cong\Omega_{p,w}^{E_k}\cong\Omega_{p}(E_k)$, with
$p\in\{0,\dots,n-1\}$.

{\em 3)} Let $E_k\subset J^k_{n}(W)$ be a
PDE, that is formally integrable and completely
integrable. One has the following isomorphisms:
$\Omega_{n-1,w}^{E_k}\cong\Omega_{n-1}(
E_k)\cong\Omega_{n-1,w}^{E_{k+h}}
\cong\Omega_{n-1,w}^{E_\infty}\cong\Omega_{n-1,w} (
E_{k+h})\cong\Omega_{n-1}( E_\infty)$.
\end{theorem}

\begin{proof}
See  \cite{PRA4, PRA14}.
\end{proof}

In order to distinguish between manifolds
$V$ representing singular solutions, where $\Sigma(V)$ has no
discontinuities, and integral manifolds where
$\Sigma(V)$ contains discontinuities, we can also consider
''conservation laws'' valued on integral manifolds
$N$ representing the integral bordism classes $[N]_{
E_k}\in\Omega_{p}^{E_k}$.

\begin{definition}
Set

\begin{equation}
\left\{\begin{array}{ll}
\mathfrak{I}(E_k)&\equiv\bigoplus_{p\ge 0 }\frac{\Omega^{p}(E_k)\cap d^{-1}(C\Omega^{p+1}(
E_k))}{d\Omega^{p-1}(E_k)\oplus\{C\Omega^{p}(E_k)\cap d^{-1}(C\Omega^{p+1}(E_k))\}}\\
&\equiv\bigoplus_{p\ge 0 }\mathfrak{I}(E_k)^{p}.\\
\end{array}\right.
\end{equation}
Here $C\Omega^{p}( E_k)$ denotes the space of all
Cartan quantum $p$-forms on $E_k$. Then we define {\em
integral characteristic numbers} of $N$, with $[N]_{
E_k}\in\Omega_{p}^{E_k}$, the numbers $
i[N]\equiv<[N]_{E_k},[\alpha]>\in \mathbb{R}$, for all
$[\alpha]\in\mathfrak{I}(E_k)^{p}$. \end{definition}

Then, one has the following theorems.

\begin{theorem}
Let us assume that $\mathfrak{I}(E_k)^{p}\not=0$. One
has a natural homomorphism:
\begin{equation}
\left\{\begin{array}{l}
{\underline{j}}_{p}:\Omega_{p}^{E_k}\to
Hom(\mathfrak{I}(E_k)^{p};\mathbb{R}),\quad [N]_{
E_k}\mapsto
{\underline{j}}_{p}([N]_{E_k}),\\
{\underline{j}}_{p}([N]_{
E_k})([\alpha])=\int_N\alpha\equiv<[N]_{E_k},[\alpha]>.\\
\end{array}\right.\end{equation}
 Then,
a necessary condition that $N'\in[N]_{E_k}$ is the following:
$ i[N]= i[N']$,  $\forall[\alpha]\in\mathfrak{I}(
E_k)^{p}$. Furthermore, if $N$ is orientable then above condition is sufficient also in order
to say that $N'\in[N]_{E_k}$.
\end{theorem}

\begin{proof}
See \cite{PRA01, PRA1, PRA4, PRA5}.
\end{proof}

\begin{cor}
Let $E_k\subseteq J^k_n(W)$ be a PDE.
Let us consider admissible $p$-dimensional, $0\le p\le n-1$, orientable integral manifolds. Let $N_1\in[N_2]_{E_k}\in\Omega_{p}^{E_k}$,
then there exists a $(p+1)$-dimensional admissible integral
manifold $V\subset E_k$, such that $\partial
V=N_1\sqcup N_2$, where $V$ is without discontinuities iff the
integral numbers of $N_1$ and $N_2$ coincide. \end{cor}

Above considerations can be generalized to include more
sophisticated solutions of PDEs.

\begin{definition}
Let $E_k\subset J^k_n(W)$ be a PDE and
let $B$ be an algebra. Let us consider the following
chain complex {\em(bar chain complex of $
E_k$)}: $\{\bar C_{\bullet}(E_k;B),\partial\}$, induced
by $B$ on the corresponding bar
chain complex of $E_k$, i.e., $\{\bar C_{\bullet}(
E_k;B),\partial\}$. (See refs.{\cite{PRA01, PRA1, PRA4}}.)
More precisely $\bar C_{p}(E_k;B)$ is the free two-sided
$B$-module of formal linear combinations with coefficients in $B$,
$\sum\lambda_i c_i$, where $c_i$ is a singular $p$-chain
$f:\triangle^p\to E_k$, that extends on a neighborhood
$U\subset\mathbb{R}^{p+1}$, such that $f$ on $U$ is differentiable and
$Tf(\triangle^p)\subset\mathbf{E}^k_{n}$, where $\mathbf{
E}^k_{n}$ is the Cartan distribution of $E_k$.
\end{definition}

\begin{theorem} The homology $\bar
H_{\bullet} (E_k;B)$ of the bar
chain complex of $E_k$ is isomorphic to {\em(closed) bar
integral singular $(p)$-bordism groups}, with coefficients in
$B$, of $E_k$: ${}^B\bar\Omega_{{p} ,s}^{E_k}\cong
\bar H_{q} (E_k;B) \cong (\bar\Omega_{p,s}^{
E_k}\otimes_{\mathbb{R}}B)$, $p\in\{0,1,\dots,n-1\}$. (If $B=\mathbb{R}$ we omit the apex $B$).
The relation between closed bordism and bordism, is given by the
following unnatural isomorphism:\footnote{Note that if $X$ is a
compact space with boundary $\partial X$, the boundary of $X\times
I$, $I\equiv[0,1]\subset\mathbb{R}$, is $\partial(X\times
I)=(X\times\{0\})\bigcup(\partial X\times
I)\bigcup(X\times\{1\})\equiv X_0\bigcup P\bigcup X_1$, with
$X_0\equiv X\times\{0\}$, $X_1\equiv X\times\{1\}$,
$P\equiv\partial X\times I$. One has $\partial P=(\partial
X\times\{0\})\bigcup (\partial X\times\{1\})=\partial
X_0\bigcup\partial X_1$. On the other hand, whether $X$ is closed,
then $\partial(X\times I)=X_0\bigcup X_1$. Furthermore we shall
denote by $[N]_{E_k}$ the equivalence class of the integral
admissible bordism of $N\subset E_k$, even if $N$ is not
necessarily closed. So, if $N$ is closed one has $[N]_{
E_k}\in{}^B\Omega^{E_k}_{\bullet,s}$, and if $N$ is
not closed one has $[N]_{E_k}\in\bar B_{\bullet}(
E_k;B)$.}
\begin{equation}Bor_{\bullet}(E_k;B)\cong
{}^B{\underline{\Omega}}_{\bullet|\bullet,s}(E_k)\bigoplus
Cyc_{\bullet}(E_k;B).\end{equation}
\end{theorem}

\begin{proof}
It follows from above results, and the following exact commutative
diagram naturally associated to the bar quantum chain
complex of $E_k$.

\begin{equation}
\xymatrix{&&0\ar[d]&0\ar[d]&&\\
&0\ar[r]&\bar B_{\bullet} (E_k;B)\ar[d]\ar[r]&
          \bar Z_{\bullet} (E_k;B)\ar[d]\ar[r]&
          \bar H_{\bullet} (E_k;B)\ar[r]&0\\
&&\bar C_{\bullet} (E_k;B)\ar[d]\ar@{=}[r]&
          \bar C_{\bullet|\bullet} (E_k;B)\ar[d]&&\\
0\ar[r]&{}^B{\underline{\Omega}}_{\bullet,s}(
E_k)\ar[r]&\bar Bor_{\bullet} (
E_k;B)\ar[d]\ar[r]&\bar Cyc_{\bullet} (E_k;B)\ar[d]\ar[r]&0&\\
&&0&0&&}\end{equation}

where $\bar B_{\bullet} (E_k;B)=\ker(\partial|_{\bar
C_{\bullet|\bullet} (E_k;B)})$, $\bar Z_{\bullet}
(E_k;B)=\IM(\partial|_{\bar C_{\bullet} (
E_k;B)})$, $\bar H_{\bullet} (E_k;B)=\bar
Z_{\bullet} (E_k;B)/\bar B_{\bullet} (
E_k;B)$. Furthermore,
$$\left\{\begin{array}{l}
b\in[a]\in \bar Bor_{\bullet}(E_k;B)\Rightarrow
a-b=\partial c,\quad
                 c\in \bar C_{\bullet}(E_k;B),\\
b\in[a]\in \bar Cyc_{\bullet}(E_k;B)\Rightarrow \partial(a-b)=0,\\
b\in[a]\in {}^A{\underline{\Omega}}_{\bullet,s}(
E_k)\Rightarrow
                 \left\{\begin{array}{l}
                          \partial a=\partial b=0\\
                          a-b=\partial c,\quad c\in \bar C_{\bullet}(E_k;B)\\
                          \end{array}
                                \right\}.\\
\end{array}\right.$$
Furthermore, one has the following canonical isomorphism:
${}^B{\underline{\Omega}}_{\bullet,s}(E_k)\cong \bar
H_{\bullet}(E_k;B)$. As $\bar C_{\bullet}(
E_k;B)$ is a free two-sided projective $B$-module, one has the
unnatural isomorphism: $\bar Bor_{\bullet}(E_k;B)\cong
{}^B{\underline{\Omega}}_{\bullet,s}(E_k)\bigoplus \bar
Cyc_{\bullet}(E_k;B)$.
\end{proof}

The spaces of conservation laws of PDEs, identify {\em
Hopf algebras}. (Hopf
algebras considered here are generalizations of usual Hopf algebras \cite{PRA1}.)

\begin{definition}
The {\em full space of $p$-conservation laws}, (or {\em full
$p$-Hopf algebra}), of $E_k$ is the following one:
${\bf H}_{p}(E_k)\equiv \mathbb{R}^{\Omega_{p}^{E_k}}$. We call {\em full Hopf
algebra}, of $E_k$, the following: $\mathbf{H}_{n-1}(E_\infty)\equiv \mathbb{R}^{\Omega_{n-1}^{
E_\infty}}$.
\end{definition}

\begin{definition} The {\em space of (differential)
conservation laws} of $E_k\subset J^k_n(W)$, is
$\mathfrak{C}ons(E_k)=\mathfrak{I}(
E_\infty)^{n-1}$.
\end{definition}

\begin{theorem}
The full $p$-Hopf algebra of a PDE $
E_k\subset J^k_n(W)$ has a natural structure of
Hopf algebra (in extended sense).
Furthermore, the space of conservation laws of $E_k$ has a
canonical representation in $\mathbf{H}_{n-1}( E_\infty)$.
\end{theorem}

\begin{proof}
See \cite{PRA01, PRA1}.
\end{proof}

\begin{theorem} Set: $\mathbf{H}_{n-1}(E_k)\equiv
\mathbb{R}^{\Omega^{E_k}_{n-1}}$, $\mathbf{H}_{n-1,s}(
E_k)\equiv \mathbb{R}^{\Omega^{E_k}_{n-1,s}}$, $\mathbf{
H}_{n-1,w}(E_k)\equiv \mathbb{R}^{\Omega^{E_k}_{n-1,w}}$.
One has the exact and commutative diagram reported in {\em(\ref{commutative-exact-diagram-conservation-laws})}, that define the following spaces: $\mathbf{K}^{E_k}_{n-1,w/(s,w)}$, $\mathbf{K}^{E_k}_{n-1,w}$,
$\mathbf{K}^{E_k}_{n-1,s,w}$, $\mathbf{K}^{E_k}_{n-1,s}$.

\begin{equation}\label{commutative-exact-diagram-conservation-laws}
 \xymatrix{&0&0&0&\\
0&\ar[l]\mathbf{K}^{E_k}_{n-1,w/(s,w)}\ar[u]& \ar[l]\mathbf{K}^{E_k}_{n-1,w}\ar[u]& \ar[l]\mathbf{K}^{E_k}_{n-1,s,w}\ar[u]& \ar[l]0\\
0&\ar[l]\mathbf{K}^{E_k}_{n-1,s}\ar[u]& \ar[l]\mathbf{H}_{n-1}(E_k)\ar[u]&
\ar[l]\mathbf{H}_{n-1,s}(E_k)\ar[u]& \ar[l]0\\
 &0\ar[u]&\ar[l]\mathbf{H}_{n-1,w}(E_k)\ar[u]&\ar[l]\mathbf{H}_{n-1,w}(E_k)\ar[u]& \ar[l]0\\
&&0\ar[u]&0\ar[u]&}
\end{equation}

More explicitly, one has the following canonical isomorphisms:
\begin{equation}
\left\{
\begin{array}{l} \mathbf{K}^{
E_k}_{n-1,w/(s,w)}\cong\mathbf{K}^{K^{
E_k}_{n-1,s}};\\
\mathbf{K}^{E_k}_{n-1,w}/\mathbf{K}^{
E_k}_{n-1,s,w}\cong\mathbf{K}^{K^{E_k}_{n-1,w/(s,w)}};\\
\mathbf{H}_{n-1}(E_k)/\mathbf{H}_{n-1,s}(E_k)\cong\mathbf{
K}^{E_k}_{n-1,s};\\
\mathbf{H}_{n-1}(E_k)/\mathbf{H}_{n-1,w}(E_k)\cong\mathbf{
K}^{E_k}_{n-1,w} \\
\cong\mathbf{H}_{n-1,s}(E_k)/\mathbf{H}_{n-1,w}(
E_k)\cong\mathbf{K}^{E_k}_{n-1,s,w}.\\
\end{array}\right.\end{equation}

Furthermore, under the same hypotheses of Theorem 4.44(2) one has
the following canonical isomorphism: $\mathbf{
H}_{n-1,s}(E_k)\cong\mathbf{H}_{n-1,w}(E_k)$.
Furthermore, we can represent differential conservation laws of
$E_k$ in $\mathbf{H}_{n-1,w}(E_k)$.
\end{theorem}

\begin{proof}
The proof follows directly for duality from the exact commutative diagram (\ref{comm-diag3}).
\end{proof}

\begin{definition}
We define {\em crystal obstruction} of $E_k$ the following quotient
algebra: $ cry(E_k)\equiv
\mathbf{H}_{n}((E_k)_\infty)/\mathbb{R}^{\Omega_{n}}$. We say that $E_k$ is a
{\em$0$-crystal PDE} if
$cry(E_k)=0$.
\end{definition}
\begin{remark}
An extended $0$-crystal equation $E_k\subset
J^k_{n}(W)$ does not necessitate to be a $0$-crystal PDE. In fact
$E_k$ is an extended $0$-crystal PDE if $\Omega_{n,w}^{E_k}=0$.
This does not necessarily imply that $\Omega_{n}^{E_k}=0$.
\end{remark}

\begin{cor}
Let $E_k\subset J^k_n(W)$ be a $0$-crystal PDE. Let $N_0, N_1\subset
E_k$ be two closed compact $(n-1)$-dimensional admissible integral
manifolds of $E_k$ such that $X\equiv N_0\sqcup
N_1\in[0]\in\Omega_{n}$. Then there exists a smooth solution
$V\subset E_k$ such that $\partial V=X$. (See also \cite{PRA14, PRA8, PRA9, PRA10, PRA11}.)
\end{cor}

Let us consider, now, the interaction between surgery and global
solutions in PDE's of the category $\mathfrak{M}_\infty$. Since the
surgery is a proceeding to obtain manifolds starting from other
ones, or eventually from $\varnothing$, in any theory of PDE's, where we are
interested to characterize nontrivial solutions, surgery is a
fundamental tool to consider. We have just seen that integral
bordism groups are the main structures able to characterize global
solutions of PDE's. On the other hand surgery is strictly related to
bordism groups, as it is well known in algebraic topology.
Therefore, in this section, we shall investigate as integral surgery
interacts with integral bordism groups.

\begin{definition}
Let $\pi:W\to M$ be a smooth fiber bundle of dimension $m+n$ over
a $n$-dimensional manifold $M$. Let $E_k\subset J^k_n(W)$ be a PDE
of order $k$ for $n$-dimensional submanifolds of $W$. Let
$N\subset E_k$ be an admissible integral manifold of dimension
$p\in\{0,1,\cdots,n-1\}$. Therefore, there exists a solution
$V\subset E_k$ such that $N\subset V$. An {\em admissible integral
$i$-surgery}, $0\le i\le n-1$, on $N$ is the procedure of
constructing a new $p$-dimensional admissible integral manifold
$N'$:
\begin{equation}
N'\equiv\overline{N\setminus S^i\times
D^{p-1-i}}\bigcup_{S^i\times S^{p-2-i}}D^{i+1}\times S^{p-2-i}
\end{equation}

such that $D^{i+1}\times S^{p-2-i}\subset V$. Here $\overline{Y}$
is the closure of the topological subspace $Y\subset X$, i.e., the
intersection of all closed subsets $Z\subset X$, with $Y\subset
Z$.
\end{definition}

\begin{theorem}\label{surgery-sol1}
Let $N_1,N_0\subset E_k$ be two integral compact (non-necessarily
closed) admissible $p$-dimensional submanifolds of the PDE
$E_k\subset J^k_n(W)$, such that there is an admissible
$(p+1)$-dimensional integral manifold $V\subset E_k$, such that
$\partial V=N_0\sqcup N_1$. Then it is possible to identify an integral
admissible manifold $N'_1$, obtained from $N_1$ by means of an
integral $i$-surgery, iff $N'_1$ is integral bording with $N_0$,
i.e., $N'_1\in[N_0]_{E_k}$.
\end{theorem}

\begin{proof}
As it is well known $N'_1$ is bording with $N_0$, i.e., there
exists a $(p+1)$-dimensional manifold $Y$, with $\partial Y=N_0\sqcup
N'_1$. More precisely we can take $Y=N_0\times I\bigcup
D^{i+1}\times D^{p-1-i}$. By the way, in order $N'_1$ should be
integral admissible, it is necessary that should be contained into
a solution passing from $N_1$. Then $N'_1$ is integral bording
with $N_1$, hence it is also integral bording with $N_0$.
\end{proof}

\begin{theorem}[Integral h-cobordism in Ricci flow PDE]\label{integral-h-cobordism-in-ricci-flow-pde}

The generalized Poincar\'e conjecture, for any dimension $n\ge 1$ is true, i.e., any $n$-dimensional homotopy sphere $M$ is homeomorphic to $S^n$: $M\approx S^n$.

For $1\le n\le 6$, $n\not=4$, one has also that $M$ is diffeomorphic to $S^n$: $M\cong S^n$. But for $n\ge 7$,  it does not necessitate that $M$ is diffeomorphic to $S^n$. This happens when the Ricci flow equation, under the {\em homotopy equivalence full admissibility hypothesis}, (see below for definition), becomes a $0$-crystal.

Moreover, under the {\em sphere full admissibility hypothesis}, the Ricci flow equation becomes a $0$-crystal in any dimension $n\ge 1$.
\end{theorem}

\begin{proof}
Let us first consider the following lemma.
\begin{lemma}\label{integral-characteristic-classes-in-ricci-flow-equation-a}
Let $N_0,N_1\subset(RF)$ be two space-like connected smooth compact Cauchy manifolds at two different instant $t_0\not= t_1$, identified respectively with two different Riemannian structures $(M,\gamma_0)$ and $(M,\gamma_1)$. Then one has $N_0\approxeq N_1$.
\end{lemma}
\begin{proof}
In fact this follows directly from the fact the diffeomorphisms $(M,\gamma_i)\cong N_i$, $i=0,1$, and from the fact that any Riemannian metric on $M$ can be continuously deformed into another one. More precisely we shall prove that there exists two continuous functions $f:N_0\to N_1$ and $h:N_1\to N_0$, such that $h\circ f\simeq 1_{N_0}$ and $f\circ h\simeq 1_{N_1}$. Realy we can always find  homotopies $F,G:I\times M\to M$, that  continuosly deform $\gamma_1$ into $\gamma_0$ and vice versa. More precisely $F_0=id_M$, $G_0=id_M$, $F_1^*\gamma_1=\gamma_0$, and $G_1^*\gamma_0=\gamma_1$. Therefore, we get $G_1\circ F_1\simeq G_0\circ F_0=1_M$ and $F_1\circ G_1\simeq F_0\circ G_0=1_M$. Thus we can identify $f$ with $F_1$ and $h$ with $G_1$.
\end{proof}
\begin{lemma}\label{integral-characteristic-classes-in-ricci-flow-equation-b}
Let $N_0,N_1\subset(RF)_{+\infty}$ be two space-like, smooth, compact closed, homotopy equivalent Cauchy $n$-manifolds, corresponding to two different times $t_0\not=t_1$. Then $N_0$ and $N_1$, have equal all the the integral characteristic numbers.
\end{lemma}
\begin{proof}
Since we have assumed $N_0\approxeq N_1$, there are two mappings $f:N_0\to N_1$ and $h:N_1\to N_0$, such that $h\circ f\simeq 1_{N_0}$ and $f\circ h\simeq 1_{N_1}$. These mappings for functorial property induce canonical homomorphisms between the  groups $\pi_p(N_i)$, $H_p(N_i)$, $H^p(N_i)$, $i=0,1$, that we shall simply denote by $f_*$ and $h_*$ or $f^*$ and $h^*$ according to respectively the direct or inverse character of functoriality. Then one has the following properties  $f_*\circ h_*=1$, $h_*\circ f_*=1$ and similarly for the controvariant cases, i.e., $f^*\circ h^*=1$, $h^*\circ f^*=1$. These relations means that the induced morphisms $f_*$ and $f^*$ are isomorphisms with inverse $h_*$ and $h^*$ respectively. In other words one has the isomorphisms $\pi_p(N_0)\cong \pi_p(N_1)$, $H_p(N_i)\cong H_p(N_i)$, $H^p(N_i)\cong H^p(N_i)$. As a by-product we get also the commutative diagram reported in (\ref{commutative-diagram-induced-isomorphisms-homotopic-equivalent-cauchy-data}).
\begin{equation}\label{commutative-diagram-induced-isomorphisms-homotopic-equivalent-cauchy-data}
 \xymatrix{H_{n}(N_0;\mathbb{R})\times H^{n}(N_0;\mathbb{R})\ar@{=}[dd]^{\wr}_{(f_*,(f^{-1})^*)}\ar[dr]^{<,>}&\\
 &\mathbb{R}\\
H_{n}(N_1;\mathbb{R})\times H^{n}(N_1;\mathbb{R})\ar[ur]_{<,>}& }
\end{equation}
Since $H_{n}(N_i;\mathbb{R})\cong\mathbb{R}\cong H^{n}(N_i;\mathbb{R})$, $i=0,1$, let the isomorphism $f_*$ be identified with a non-zero number $\lambda\in\mathbb{R}\setminus\{0\}$, then $(f^{-1})^*=1/\lambda$, and we get that $<f_*[N_0],(f^{-1})^*\alpha>=<\lambda,\mu/\lambda>=\mu$ where $\mu$ is the number that represents the $n$-differential form $\alpha$. On the other hand one has $<[N_0],\alpha>=1.\mu=\mu$.
\end{proof}

\begin{lemma}\label{zero-class-integral-bordism-homotopy-sphere}
Under the same hypotheses of Lemma \ref{integral-characteristic-classes-in-ricci-flow-equation-b}, let us add that we assume admissible only orientable Cauchy manifolds. Then $N_0\in[N_1]\in\Omega_n^{(RF)_{+\infty}}$.
In other words, $N=\partial V$, where $V$ is a smooth solution, iff $<[\alpha],[N]>=0$ for all the conservation laws $\alpha$.
\end{lemma}
\begin{proof}
If we assume admissible only orientable Cauchy manifolds, then the canonical homomorphism $j_n:\Omega_n^{(RF)_{+\infty}}\to(\mathcal{I}(RF)_{+\infty}^{n})^*$ is injective, (see \cite{PRA01}), therefore $N_0\in[N_1]\in\Omega_n^{(RF)_{+\infty}}$ iff $N_0$ and have equal all integral characteristic numbers.
\end{proof}

Since Lemma \ref{zero-class-integral-bordism-homotopy-sphere} is founded on the assumption that the space of conservation laws of $(RF)$ is not zero, in the following lemma we shall prove that such an assumption is true.
\begin{lemma}\label{conservation-laws-ricci-flow-equations}
The space $\mathcal{I}(RF)_{+\infty}^{n}\cong E^{0,n}_1$ of conservation laws of the $(RF)$ is not zero. In fact any differential $n$-form given in {\em(\ref{conservation-laws-ricci-flow-equations-a})} is a conservation law of $(RF)$.\footnote{$E^{0,n}_1$ is the spectral term, in the Cartan spectral sequence of a PDE $E_k\subset J^k_n(W)$, just representing the conservation laws space of $E_k$. (See e.g., \cite{PRA01, PRA1, PRA5}.)}
\begin{equation}\label{conservation-laws-ricci-flow-equations-a}
    \omega=T dx^1\wedge\cdots\wedge dx^n+X^p(-1)^p dt\wedge dx^1\wedge\cdots\wedge \widehat{dx^p}\wedge\cdots\wedge dx^n
\end{equation}
with
\begin{equation}\label{conservation-laws-ricci-flow-equations-b}
    \left\{
    \begin{array}{l}
      T=g_{ij}\varphi^{ij}\\
      X^p=\kappa\int\left(R_{ij}(g)\, \varphi^{ij}-R^{ij}(\varphi)\, g_{ij}\right)\, dx^p+c^p\\
    \end{array}
    \right\}:\hskip 3pt\left\{ \begin{array}{l}
    \varphi^{ij}{}_{,t}-\kappa R^{ij}(\varphi)=0\\
    g_{ij,t}+\kappa R_{ij}(g)=0\\
     \end{array}\right\}.
\end{equation}
$c^p\in\mathbb{R}$ are arbitrary constants and $\varphi^{ij}$, $1\le i,j\le n$, are functions on $\mathbb{R}\times M$, symmetric in the indexes, solutions of the equation given in {\em(\ref{conservation-laws-ricci-flow-equations-b})}.
\end{lemma}

\begin{proof}
Let us prove that $d\omega|_V=0$ for any solution $V$ of $(RF)$. In fact, by a direct calculation we get
\begin{equation}\label{conservation-laws-ricci-flow-equations-c}
d\omega=\left[(g_{ij,t}+\kappa R_{ij}(g))\varphi^{ij}+g_{ij}(\varphi^{ij}{}_{,t}-\kappa R^{ij}(\varphi))\right]dt\wedge dx^1\wedge\cdots\wedge dx^n.
    \end{equation}
    Therefore, the conservation laws in (\ref{conservation-laws-ricci-flow-equations-a}) are identified with the solutions of the PDE given in (\ref{conservation-laws-ricci-flow-equations-b}) for $\varphi_{ij}$. This is an equation of the same type of the Ricci flow equation, hence its set of solutions is not empty.
\end{proof}

Now, let $M$ belong to the same integral bordism class of $S^n$ in $(RF)$: $M\in[S^n]\in\Omega_n^{(RF)}$. It follows from Theorem \ref{change-topology-and-smooth-solutions} and Theorem \ref{tunnel-effect-pde}, that $M$ is necessarily homeomorphic to $S^n$. However, if $n\ge 4$, the smooth solution $V$ such that $\partial V=M\sqcup S^n$ does not necessitate to be a trivial h-cobordism. This happens iff the homotopy equivalence $f:M\approxeq S^n$ is such that $f\simeq 1_{S^n}$. (See Theorem \ref{properties-manifold-structure-set}.) This surely is the case at low dimensions $n=1,2,3,5,6$, and also for $n=4$, if holds the smooth Poincar\'e conjecture. But for $n\ge 7$ an homotopy sphere may have different structures with respect to this property. (See Tab. \ref{calculated-h-cobordism-groups-homotopy-sphere}, Lemma \ref{gamma-group-a} and Lemma \ref{gamma-group-b}.) In fact it is well known that there are homotopy spheres characterized by rational Pontrjagin numbers. Since rational Pontrjagin classes $p_q\in H^{4q}(M;\mathbb{Q})$ are homeomorphic invariants, such manifolds cannot admit a differentiable structure, taking into account the fact that the signature is a topological invariant. Such homotopy spheres are obtained by gluing a disk $D^n$, along its boundary $S^{n-1}$, with the boundary of a disk-$D^q$-fiber bundle over a sphere $S^m$, $E\to S^m$, such that $q=n-m$. When the $(n-1)$-dimensional boundary $\partial E$ is diffeomorphic to $S^{n-1}$, gives to $\widetilde{E}\equiv E\bigcup D^n$ a differentiable structure. But whether $\partial E\thickapprox S^{n-1}$, $\Sigma^n$ cannot, in general, have a differentiable structure, since it is characterized by rational Pontrjagin numbers. Therefore there are exotic spheres, (for example $\Sigma^{n-1}\equiv\partial E$), that are homeomorphic, but not diffeomorphic to $S^{n-1}$. In such cases the solution $V$ of the Ricci flow equation such that $\partial V=\Sigma^{n-1}\sqcup S^{n-1}$, cannot be, in general, a trivial h-cobordism.

By conclusion we get that not all $n$-dimensional homotopy spheres $M$  can, in general, belong to the same integral boundary class of $[S^n]\in\Omega^{(RF)}_n$, even if there exist singular solutions $V\subset(RF)$ such that $\partial V=M\sqcup S^n$. In fact, it does not necessitate, in general, that $V$ should be a trivial h-cobordism, i.e., that the homotopy equivalence between $M$ and $S^n$ should be a diffeomorphism of $S^n$ isotopic to the identity. This has, as a by-product, that in general $M$ is only homeomorphic to $S^n$ but not diffeomorphic to $S^n$. In order to better understand this aspect in the framework of PDE's algebraic topology, let us first show how solutions with neck-pinching singular points are related to smooth solutions. Recall the commutative diagram in Theorem 2.24 in \cite{PRA01}, here adapted in (\ref{commutative-diagram-smooth-bordism-singular-bar-bordism-group-ricci-flow-pde}) to $(RF)_{+\infty}$ and in dimension $p=n$.

\begin{equation}\label{commutative-diagram-smooth-bordism-singular-bar-bordism-group-ricci-flow-pde}
    \xymatrix{&&0\ar[d]&\\
&\Omega_{n}^{(RF)_{+\infty}}\ar[d]_{j_n}\ar[r]^(0.3){i_n}&\Omega^{(RF)_{+\infty}}_{n,s}\cong\bar H_n((RF)_{+\infty};\mathbb{R})\ar[d]&\\
0\ar[r]&{(\mathcal{I}((RF)_{+\infty}))^*}
\ar[r]& (\bar H_n((RF)_{+\infty};\mathbb{R}))^*\ar[r]&0 }
\end{equation}
There $\bar H^n((RF)_{+\infty};\mathbb{R})$ is the $n$-dimensional bar de Rham cohomology of $(RF)_{+\infty}$. The isomorphism $\Omega^{(RF)_{+\infty}}_{n,s}\cong\bar H_n((RF)_{+\infty};\mathbb{R})$ is a direct by-product of the exact commutative diagram in Definition 4.8(3) in \cite{PRA1}. Then, taking into account that in solutions of $(RF)_{+\infty}$ cannot be present Thom-Boardman singularities, it follows that solutions bording smooth Cauchy manifolds in the bordism classes of  $\widehat{\Omega_{n}^{(RF)_{+\infty}}}\equiv i_n(\Omega_{n}^{(RF)_{+\infty}})\triangleleft\:\Omega^{(RF)_{+\infty}}_{n,s}$, can have singularities of neck-pinching type. (See Fig. \ref{neck-pinching-singular-solutions}(a).)
\begin{figure}[h]
\centering
\centerline{\includegraphics[height=4cm]{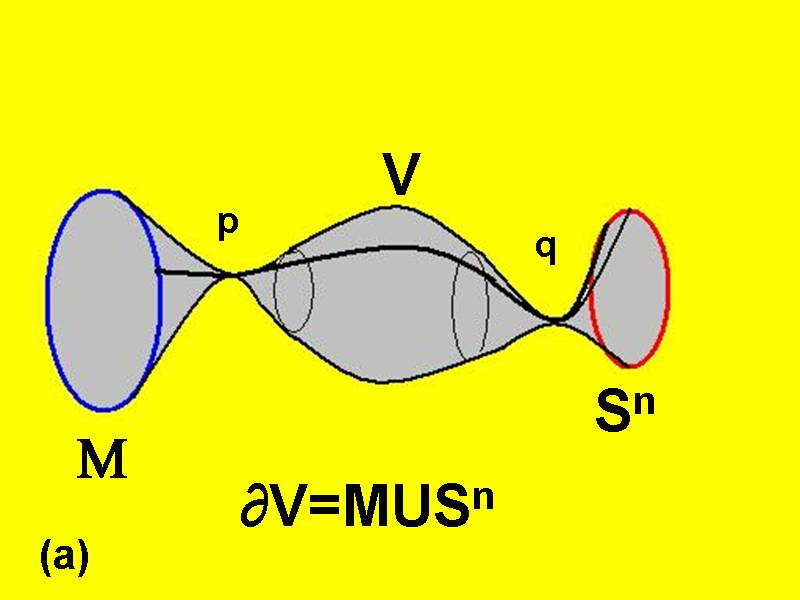}
\includegraphics[height=4cm]{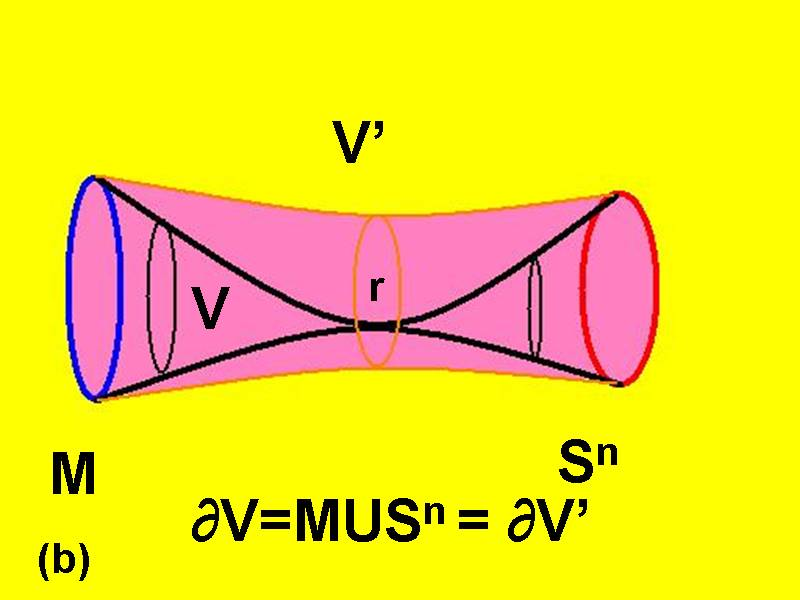}}
\caption{Neck-pinching singular solutions type, $V$, $\partial V=M\sqcup S^n$, in Ricci flow equations, with singular points $p$, $q$ in (a) and $r$ in (b). In (b) is reported also a smooth solution $V'$, bording $M$ and $S^n$ as well as a neck-pinching singular solution $V$ bording the same manifolds.}
\label{neck-pinching-singular-solutions}
\end{figure}

From Corollary 2.5 in \cite{PRA01} it follows that if $M$ is an homotopy sphere belonging to the integral bordism class $[S^n]\in\widehat{\Omega_{n}^{(RF)_{+\infty}}}$, one has surely $M\sqcup S^n=\partial V'$, for some smooth solution $V'$ of $(RF)$, but can be also $M\sqcup S^n=\partial V$ for some solution $V\subset(RF)_{+\infty}$ having some neck-pinching singularity. (See Fig. \ref{neck-pinching-singular-solutions}(b).) In general $V$ cannot be considered isotopic to $V'$. However $M$ is diffeomorphic to $S^n$, (the diffeomorphism is that induced by the smooth solution $V'$), and all the singular points, in the neck-pinching singular solutions, bording $M$ with $S^n$, are ''solved'' by the smooth bordism $V'$. Let us also emphasize that if an $n$-dimensional homotopy sphere $M\in[S^n]\in\Omega_n^{(RF)_{+\infty}}$, i.e., there exists a smooth solution $V\subset(RF)_{+\infty}$ such that $\partial V=M\sqcup S^n$, means that the characteristic flow on $V$ is without singular points, hence from Theorem \ref{tunnel-effect-pde} it follows that $V\cong M\times I$, and $V\cong S^n\times I$, hence $M\cong S^n$. If this happens for all $n$-dimensional homotopy sphere, then $\Theta_n=0$ and vice versa. However, it is well known that there are homotopy spheres of dimensions $n\ge 7$ for the which $\Theta_n\not=0$. (For example the Milnor spheres.) This is equivalent to say that $\pi_0(Diff_+(S^{n-1}))\not=0$, since $\Theta_n\cong\pi_0(Diff_+(S^{n-1}))$ (Smale). This happens when there are homotopy spheres that bound non-contractible manifolds. In fact, if there exists a trivial h-bordism $V$ bording $S^n$ with $M$, then $W=V\bigcup_{S^n}D^{n+1}\cong D^{n+1}$ and $\partial W=M$. However, since the conservation laws of $(RF)$ depend on a finite derivative order (second order), the fact that all $n$-dimensional homotopy spheres have the same integral characteristic numbers of the sphere $S^n$, implies that there are smooth integral manifolds bording them at finite order. There the symbols of the Ricci flow equation, and its finite order prolongations, are not trivial ones, hence in general such solutions present Thom-Boardman singular points. As a by-product of Theorem 2.25 in \cite{PRA01}, (see also \cite{PRA1}), and Theorem 2.1, Theorem 2.12 and Theorem 3.6 in \cite{PRA4} between such solutions, there are ones that are not smooth, but are topological solutions inducing the homeomorphisms between $M$ and $S^n$: $M\approx S^n$. Therefore, if we consider admissible in $(RF)$ only space-like Cauchy integral manifolds, corresponding to homotopy spheres, ({\em homotopy equivalence full admissibility hypothesis}), then one has the short exact sequence (\ref{short-exact-sequence-homotopy-spheres-in-ricci-flow-pde}).
\begin{equation}\label{short-exact-sequence-homotopy-spheres-in-ricci-flow-pde}
    \xymatrix{0\ar[r]&K^{(RF)}_{n,s}\ar[r]&\Omega^{(RF)}_{n}\ar[r]&\Omega^{(RF)}_{n,s}=0\ar[r]&0}
\end{equation}
We get

\begin{equation}\label{short-exact-sequence-homotopy-spheres-in-ricci-flow-pde-a}
   \Omega^{(RF)}_{n} \cong K^{(RF)}_{n,s}=\left\{[M]|M=\partial V,\: V=\hbox{\rm singular solution of $(RF)$}\right\}
\end{equation}
and $M\in[S^n]\in\Omega^{(RF)}_{n} $ iff $M\cong S^n$. Furthermore, two $n$-dimensional homotopy spheres ${}'\Sigma^n$ and $\Sigma^n$ belong to the same bordism class in $\Omega^{(RF)}_{n}$ iff ${}'\Sigma^n\cong \Sigma^n$. Therefore we get the following canonical mapping $\Gamma_n\to \Omega^{(RF)}_{n}$, $[\Sigma^n]_{\Gamma_n}\mapsto[\Sigma^n]_{\Omega^{(RF)}_{n}}$, such that $0\:\left\{\in \Gamma_n\right\}\mapsto [S^n]_{\Omega^{(RF)}_{n}}$. (For $n\not= 4$, one can take $\Gamma_n=\Theta_n$.) This mapping is not an isomorphism. Therefore, in the homotopy equivalence full admissibility hypothesis, and in the case that $\Gamma_n=0$, we get that the Ricci flow equation becomes a $0$-crystal PDE, so all homotopy spheres are diffeomorphic to $S^n$. This is the case, for example, of $n=3$, corresponding to the famous Poincar\'e conjecture. Finally, if we consider admissible in $(RF)$ only space-like Cauchy integral manifolds, corresponding to manifolds diffeomorphic to spheres, ({\em sphere full admissibility hypothesis}), then $\Omega^{(RF)}_{n} \cong K^{(RF)}_{n,s}\cong \Omega^{(RF)}_{n,s}=0$ and one has $cry(RF)=0$, i.e. $(RF)$ becomes a $0$-crystal for any dimension $n\ge 1$.\footnote{From this theorem we get the conclusion that the Ricci flow equation for $n$-dimensional Riemannian manifolds, admits that starting from a $n$-dimensional sphere $S^n$, we can dynamically arrive, into a finite time, to any $n$-dimensional homotopy sphere $M$. When this is realized with a smooth solution, i.e., solution with characteristic flow without singular points, then $S^n\cong M$. The other homotopy spheres $\Sigma^n$, that are homeomorphic to $S^n$ only, are reached by means of singular solutions. So the titles of this paper and its companion \cite{PRA16} are justified now !

Results of this paper agree with previous ones by J. Cerf \cite{CERF}, M. Freedman \cite{FREEDMAN}, M. A. Kervaire and J. W. Milnor \cite{KERVAIRE-MILNOR, MILNOR}, E. Moise \cite{MOISE1, MOISE2} and S. Smale \cite{SMALE1, SMALE2, SMALE3},  and with the recent proofs of the Poincar\'e conjecture by R. S. Hamilton \cite{HAMIL1, HAMIL2, HAMIL3, HAMIL4, HAMIL5}, G. Perelman \cite{PER1, PER2}, and A. Pr\'astaro \cite{PRA14, AG-PRA1}.}
\end{proof}

\end{document}